\pdfoutput=1
\RequirePackage{ifpdf}
\ifpdf 
\documentclass[pdftex]{sigma}
\else
\documentclass{sigma}
\fi

\numberwithin{equation}{section}
\numberwithin{figure}{section}

\newcommand{\N}{{\mathbb N}}
\newcommand{\Z}{{\mathbb Z}}
\newcommand{\Q}{{\mathbb Q}}
\newcommand{\C}{{\mathbb C}}
\newcommand{\R}{{\mathbb R}}
\renewcommand{\P}{{\mathbb P}}
\renewcommand{\H}{{\mathbb H}}
\newcommand{\F}{{\mathbb F}}

\newcommand{\EE}{{\mathcal E}}

\newcommand{\MM}{{\mathcal M}}
\newcommand{\NN}{{\mathcal N}}

\newcommand{\OO}{{\mathcal O}}

\newcommand{\TT}{{\mathcal T}}
\newcommand{\UU}{{\mathcal U}}
\newcommand{\VV}{{\mathcal V}}

\newcommand{\ddd}{{\rm d}}
\newcommand{\www}{\widetilde}
\newcommand{\wwh}{\widehat}

\newcommand{\uuuu}{\underline}
\newcommand{\paa}{\partial}
\newcommand{\nnn}{\nabla}

\DeclareMathOperator{\Aut}{Aut}
\DeclareMathOperator{\Eig}{Eig}
\DeclareMathOperator{\id}{id}
\DeclareMathOperator{\Lie}{Lie}
\DeclareMathOperator{\pr}{pr}
\DeclareMathOperator{\rank}{rank}
\DeclareMathOperator{\rk}{rk}
\DeclareMathOperator{\Ree}{Re}
\DeclareMathOperator{\tr}{tr}

\newtheorem{Theorem}{Theorem}[section]
\newtheorem{Corollary}[Theorem]{Corollary}
\newtheorem{Lemma}[Theorem]{Lemma}
\newtheorem{claim}[Theorem]{Claim}
 { \theoremstyle{definition}
\newtheorem{Definition}[Theorem]{Definition}
\newtheorem{Remarks}[Theorem]{Remarks}
\newtheorem{Remark}[Theorem]{Remark}
\newtheorem{Notation}[Theorem]{Notation}
 }

\begin{document}
\allowdisplaybreaks

\newcommand{\arXivNumber}{2009.14314}

\renewcommand{\thefootnote}{}

\renewcommand{\PaperNumber}{082}

\FirstPageHeading

\ShortArticleName{Rank 2 Bundles with Meromorphic Connections with Poles of Poincar\'e Rank 1}

\ArticleName{Rank 2 Bundles with Meromorphic Connections\\
with Poles of Poincar\'e Rank 1\footnote{This paper is a~contribution to the Special Issue on Primitive Forms and Related Topics in honor of~Kyoji Saito for his 77th birthday. The full collection is available at \href{https://www.emis.de/journals/SIGMA/Saito.html}{https://www.emis.de/journals/SIGMA/Saito.html}}}

\Author{Claus HERTLING}

\AuthorNameForHeading{C.~Hertling}

\Address{Lehrstuhl f\"ur algebraische Geometrie, Universit\"at Mannheim,\\ B6, 26, 68159 Mannheim, Germany}
\Email{\href{mailto:hertling@math.uni-mannheim.de}{hertling@math.uni-mannheim.de}}
\URLaddress{\url{https://www.wim.uni-mannheim.de/hertling/team/prof-dr-claus-hertling/}}

\ArticleDates{Received September 30, 2020, in final form August 20, 2021; Published online September 07, 2021}

\Abstract{Holomorphic vector bundles on $\C\times M$, $M$ a complex manifold, with meromorphic connections with poles of Poincar\'e rank~1 along $\{0\}\times M$ arise naturally in algebraic geometry. They are called $(TE)$-structures here. This paper takes an abstract point of~view. It~gives a complete classification of all $(TE)$-structures of rank~2 over germs $\big(M,t^0\big)$ of~mani\-folds. In~the case of $M$ a point, they separate into four types. Those of three types have universal unfoldings, those of the fourth type (the logarithmic type) not. The classification of unfoldings of $(TE)$-structures of the fourth type is rich and interesting. The paper finds and lists also all $(TE)$-structures which are basic in the following sense: Together they induce all rank $2$ $(TE)$-structures, and each of them is not induced by any other $(TE)$-structure in the list. Their base spaces $M$ turn out to be 2-dimensional $F$-manifolds with Euler fields. The paper gives also for each such $F$-manifold a classification of all rank~2 $(TE)$-structures over it. Also this classification is surprisingly rich. The backbone of the paper are normal forms. Though also the monodromy and the geometry of the induced Higgs fields and of the bases spaces are important and are considered.}


\Classification{34M56; 34M35; 53C07}

\tableofcontents

\renewcommand{\thefootnote}{\arabic{footnote}}
\setcounter{footnote}{0}

\section{Introduction}

A holomorphic vector bundle $H$ on $\C\times M$, $M$ a complex
manifold, with a meromorphic connection $\nnn$ with a
pole of Poincar\'e rank 1 along $\{0\}\times M$ and no pole
elsewhere, is called a~$(TE)$-structure.
The aim of this paper is the local classification of all rank $2$
$(TE)$-structures, over arbitrary germs $\big(M,t^0\big)$ of manifolds.

Before we talk about the results, we will put these structures
into a context, motivate their definition,
mention their occurence in algebraic geometry,
and formulate interesting problems.
The rank $2$ case is the first interesting case and already
very rich. In~many aspects it is probably typical for
arbitrary rank, in some not. And it is certainly the only
case where such a thorough classification is feasible.

The pole of Poincar\'e rank 1 along $\{0\}\times M$ of the
pair $(H,\nnn)$ means the following. Let $t=(t_1,\dots ,t_n)$
be holomorphic coordinates on $M$ with coordinate
vector fields $\paa_1,\dots ,\paa_n$, and~let~$z$ be the standard
coordinate on $\C$. Then $\nnn_{\paa_z}\sigma$ for a
holomorphic section $\sigma\in\OO(H)$ of~$H$
is in~$z^{-2}\OO(H)$, and $\nnn_{\paa_j}\sigma$ is in
$z^{-1}\OO(H)$. The pole of order two along $\paa_z$
is the first case beyond the easy and tame case of a
pole of order~1, i.e., a logarithmic pole.
The pole of order~1 along~$\paa_i$ gives a good variation
property, a generalization of Griffiths transversality
for variations of Hodge structures. It~is the most natural
constraint for an isomonodromic family of bundles on~$\C$
with poles of order~2 at~$0$.
So, a pole of Poincar\'e rank~1 is in some sense the first
case beyond the case of connections with logarithmic poles.
(A pole of Poincar\'e rank $r\in\N_0$
is defined for example in~\cite[Section~0.14]{Sa02}.)

In algebraic geometry, such connections arise naturally.
A distinguished case is the Fourier--Laplace transformation
(with respect to the coordinate $z$)
of the Gauss--Manin connection of a~family of
holomorphic functions with isolated singularities
(see \cite[Chapter~8]{He03} and~\cite[Chapter~VII]{Sa02}). The paper~\cite{He03}
defines $(TERP)$-structures, which are $(TE)$-structures
with additional real structure and pairing
and which generalize variations of Hodge structures.
Also the notion $(TEZP)$-structure makes sense, which is a~$(TE)$-structures with a flat $\Z$-lattice bundle
on $\C^*\times M$ and a~certain pairing.
A family of holomorphic functions with isolated singularities
(and some topological well-behavedness) gives rise to a~$(TEZP)$-structure over the base space of the family
(see \cite[Chapter~11.4]{He02} and~\cite[Chapter~8]{He03}).

In~\cite{He02} and other papers of the author,
a Torelli problem is considered.
We formulate it here as the following question:
Does the $(TEZP)$-structure of a holomorphic function germ with
an~isolated singularity determine the $(TEZP)$-structure
of the universal unfolding of the function germ?
The first one is a $(TE)$-structure over a point $t^0$.
The second one is a $(TE)$-structure over a germ
$\big(M,t^0\big)$ of a manifold $M$. It~it an {\it unfolding}
of the first $(TE)$-structure with a {\it primitive
Higgs field}. The base space $M$ is an
{\it $F$-manifold with Euler field}.

We explain these notions. A second $(TE)$-structure over
a manifold $M$ is an unfolding of a~first $(TE)$-structure
over a submanifold of $M$ if the restriction of the second
$(TE)$-structure to the submanifold
is isomorphic to the first $(TE)$-structure.
If $\varphi\colon M'\to M$ is a morphism
and if $(H,\nnn)$ is a $(TE)$-structure over $M$,
then the pull back $\varphi^*(H,\nnn)$ is a $(TE)$-structure
over~$M'$. An unfolding of a $(TE)$-structure
is {\it universal} if it {\it induces}
any unfolding via a unique map $\varphi$
(see Definition~\ref{t3.15}$(b){+}(c)$ for details).

If $(H\to\C\times M,\nnn)$ is a $(TE)$-structure, then
define the vector bundle $K:=H|_{\{0\}\times M}$ on $M$
and the Higgs field
$C:=[z\nnn]\in \Omega^1(M,{\rm End}(K))$ on $K$.
The endomorpisms $C_X=[z\nnn_X]\colon \OO(K)\to\OO(K)$ for
$X\in \TT_M$ commute with one another,
and they commute with the
endomorphism $\UU:=\big[z^2\nnn_{\paa_z}\big]\colon \OO(K)\to\OO(K)$
(see Definition~\ref{t3.8} and Lemma~\ref{t3.12}).
The Higgs field~$C$ is {\it primitive} if on each sufficiently
small subset $U\subset K$ a section $\zeta_U$ exists such
that the map $\TT_U\to \OO(K),$ $X\mapsto C_X\zeta_U$, is
an isomorphism (see Definition~\ref{t3.13}).

An {\it $F$-manifold} with {\it Euler field}
is a complex manifold $M$
together with a holomorphic commutative and associative
multiplication $\circ$ on $\TT_M$ which comes equipped
with the integrability condition~\eqref{2.1},
with a unit field $e\in\TT_M$ (with $\circ e=\id$)
and an Euler field $E\in\TT_M$
with $\text{Lie}_E(\circ)=\circ$
(see~\cite{HM99} or Definition~\ref{t2.1}).
A $(TE)$-structure over $M$ with
primitive Higgs field indu\-ces on the base manifold $M$ the
structure of an $F$-manifold with Euler field
(see Theorem~\ref{t3.14} for details).

A result of Malgrange~\cite{Ma86} (cited in
Theorem~\ref{t3.16}$(c)$) says that a $(TE)$-structure
over a~point~$t^0$ has a universal unfolding if the
endomorphism $\UU\colon K\to K$ (here $K$ is a vector space)
is regular, i.e., it has only one Jordan block for each
eigenvalue. Theorem~\ref{t3.16}$(b)$ gives a generalization
from~\cite{HM04}. A special case of this generalization
says that a $(TE)$-structure with primitive
Higgs field over a germ $\big(M,t^0\big)$ is its own universal
unfolding (see Theorem~\ref{t3.16}$(a)$).
A supplement from~\cite{DH17} says that then the base space
is a {\it regular} $F$-manifold (see Definition~\ref{t2.4}
and Theorem~\ref{t2.5}).

Malgrange's result gives a universal unfolding
if one starts with a $(TE)$-structure over a~point
whose endomorphism $\UU$ is regular.
However, if one starts with a~$(TE)$-structure over a~point such that $\UU$ is not regular, then in general it
has no universal unfolding, and the study of all its
unfoldings becomes very interesting. The second half of
this paper (Sections~\ref{c6}--\ref{c8}) studies
this situation in rank~2. The Torelli problem for a
holomorphic function germ with an isolated singularity
is similar: The endomorphism
$\UU$ of its $(TEZP)$-structure is never regular
(except if the function has an $A_1$-singularity),
but I hope that the $(TEZP)$-structure
determines nevertheless somehow the specific
unfolding with primitive Higgs field, which comes from the
universal unfolding of the original function germ.

Now sufficient background is given.
We describe the contents of this paper.

The short Section~\ref{c2} recalls the classification of the
2-dimensional germs of $F$-manifolds with Euler fields
(Theorem~\ref{t2.2} from~\cite{He02} and Theorem~\ref{t2.3}
from~\cite{DH20-3}). It~treats also regular $F$-manifolds
(Definition~\ref{t2.4} and Theorem~\ref{t2.5} from~\cite{DH17}).

Section~\ref{c3} recalls many general facts on
$(TE)$-structures: their definition, their presentation
by matrices, formal $(TE)$-structures, unfoldings and
universal unfoldings of $(TE)$-structures, Malgrange's
result and the generalization in~\cite{HM04},
$(TE)$-structures over $F$-manifolds, $(TE)$-structures
with primitive Higgs fields, regular singular $(TE)$-structures
and elementary sections, Birkhoff normal form for
$(TE)$-structures (not all have one, Theorem~\ref{t3.20} cites
existence results of~Ple\-mely and of Bolibroukh and Kostov).
Not written before, but elementary is a correspondence
between $(TE)$-structures with trace free endomorphism
$\UU$ and arbitrary $(TE)$-structures (Lem\-mata~\ref{t3.9},~\ref{t3.10} and~\ref{t3.11}).

New is the notion of a marked $(TE)$-structure. It~is needed for the construction of moduli spaces.
Theorem~\ref{t3.29} (which builds on results in
\cite{HS10}) constructs such moduli spaces, but
only in the case of regular singular $(TE)$-structures. It~starts with a {\it good family} of regular
singular $(TE)$-structures. There are two open problems. It~is not clear how to generalize this notion of a good
family beyond the case of regular singular $(TE)$-structures.
We hope, but did not prove for rank $\geq 3$, that
any regular singular $(TE)$-structure (over $M$ with
$\dim M\geq 1$) is a good family of regular singular
$(TE)$-structures. For rank $2$ this is true, it follows
from Theorem~\ref{t8.5}.

Section~\ref{c4} gives the classification of rank $2$
$(TE)$-structures over a point $t^0$. There are 4 types,
which we call (Sem), (Bra), (Reg) and (Log)
(for {\it semisimple, branched, regular singular} and
{\it logarithmic}). In~the type (Sem) $\UU$ has two
different eigenvalues, in the type (Log) $\UU\in\C\cdot\id$,
in~the types (Bra) and (Reg) $\UU$ has a $2\times 2$
Jordan block. In~the cases when $\UU$ is trace free,
a $(TE)$-structure of type (Log) has a logarithmic pole,
a $(TE)$-structure of type (Reg) has a~regular singular,
but not logarithmic pole, and the pull back of a~$(TE)$-structure of type (Bra) by a~branched cover of $\C$
of order 4 has a meromorphic connection with
semisimple pole of order~3 (see Lemma~\ref{t4.9}).
The semisimple case (Sem) is not central in this paper.
Therefore we do not discuss it in detail
and do not introduce Stokes structures. For the other types
(Bra), (Reg) and (Log), Section~\ref{c4} discusses normal forms
and their parameters. All $(TE)$-structures of type (Bra)
have nice Birkhoff normal forms (Theorem~\ref{t4.11}),
but not all of type (Reg) (Theorem~\ref{t4.17} and Remark~\ref{t4.19}) and type (Log) (Theorem~\ref{t4.20} and
Remark~\ref{t4.22}). The types (Reg) and (Log) become
transparent by the use of elementary sections.

A $(TE)$-structure of type (Sem) or (Bra) or (Reg) over
a point $t^0$ satisfies the hypothesis of Malgrange's result,
namely, the endomorphism $\UU\colon K\to K$ is regular. Therefore it
has a~uni\-ver\-sal unfolding, and any unfolding of it is
induced by this universal unfolding. 
Section~\ref{c5} discusses this.
Also because of this fact, the semisimple case is not
central in this paper.

Sections~\ref{c6}--\ref{c8} are devoted to the study
of $(TE)$-structures over a germ $\big(M,t^0\big)$ such that
the restriction to $t^0$ is a $(TE)$-structure of type (Log).
Then the set of points over which the $(TE)$-structure
restricts to one of type (Log) is either a hypersurface or
the whole of $M$. In~the first case, it restricts to a fixed
{\it generic} type (Sem) or (Bra) or (Reg) over points
not in the hypersurface. In~the second case, the generic
type is (Log).

Section~\ref{c6} starts this study. It~considers the
cases with trace free $\UU$ and $\dim M=1$. It~has three parts. In~the first part, invariants of
such 1-parameter families are studied. In~a surprisingly
direct way, constraints on the difference of the
{\it leading exponents} (defined in Theorem~\ref{t4.20})
of the logarithmic $(TE)$-structure over $t^0$ are found,
and the monodromy in the generic cases (Sem) and (Bra) turns
out to be semisimple (Theorem~\ref{t6.2}).
By Plemely's result (and our direct calculations),
these cases come equipped with Birkhoff normal forms.
Theorem~\ref{t6.3} in the second part
classifies all $(TE)$-structures over $\big(M,t^0\big)$ with
trace free $\UU$, $\dim M=1$, logarithmic restriction to $t^0$
and Birkhoff normal form.
Theorem~\ref{t6.7} in the third part classifies all
generically regular singular $(TE)$-structures over $\big(M,t^0\big)$
with $\dim M=1$, logarithmic restriction to~$t^0$, and
whose monodromy has a $2\times 2$ Jordan block.
The majority of these cases has no Birkhoff normal
form. Theorems~\ref{t6.3} and~\ref{t6.7} overlap
in the cases which have Birkhoff normal forms.

Section~\ref{c7} makes the moduli spaces of marked regular
singular $(TE)$-structures from Theorem~\ref{t3.28}
explicit in the rank $2$ cases. It~builds on the classification
results for the types (Reg) and (Log) in Section~\ref{c4}.
The long Theorem~\ref{t7.4} describes the moduli spaces
and offers~5 figures in order to make this more transparent.
The moduli spaces have countably many topological components,
and each component consists of an infinite chain of
projective spaces which are either the projective line $\P^1$
or the Hirzebruch surface $\F_2$ or $\www\F_2$ (which is
obtained by blowing down in $\F_2$ the unique $(-2)$-curve).
These moduli spaces simplify in the generic case (Reg)
the main proof in Section~\ref{c8}, the proof of Theorem~\ref{t8.5}.

\looseness=1
Section~\ref{c8} gives complete classification results,
from different points of view. It~has three parts.
Theorem~\ref{t8.1} lists all rank~2 $(TE)$-structures
over a 2-dimensional germ $\big(M,t^0\big)$ such that the restriction
to $t^0$ has a logarithmic pole, such that the
Higgs field is generically primi\-tive,
and such that the induced structure of an $F$-manifold with
Euler field extends to all of~$M$. Theorem~\ref{t8.1}$(d)$
offers explicit normal forms.
Corollary~\ref{t8.3} starts with any logarithmic
rank~2 $(TE)$-structure over a point $t^0$ and lists
the $(TE)$-structures in Theorem~\ref{t8.1}$(d)$
which unfold~it.

Theorem~\ref{t8.5} is the most fundamental result of
Section~\ref{c8}. Table~\eqref{8.12} in it is a sublist
of the $(TE)$-structures in Theorem~\ref{t8.1}$(d)$. Theorem~\ref{t8.5} states that {\it any} unfolding of a rank~2
$(TE)$-structure of type (Log) over a point
is induced by one $(TE)$-structure in table~\eqref{8.12}. In~the generic cases (Reg) and (Log) these are precisely
those in Theorem~\ref{t8.1}$(d)$ with primitive Higgs field,
but in the generic cases (Sem) and (Bra) table~\eqref{8.12}
contains many $(TE)$-structures with only generically
primitive Higgs field. All the $(TE)$-structures in
table~\eqref{8.12} are universal unfoldings of themselves,
also those with only generically primitive Higgs field.
Almost all logarithmic $(TE)$-structures over a point
have several unfoldings which do not induce one another.
Only the logarithmic $(TE)$-structures over a point whose
monodromy has a $2\times 2$ Jordan block and whose
two leading exponents coincide have a universal unfolding.
This follows from Theorem~\ref{t8.5} and Corollary~\ref{t8.3}.

The second part of Section~\ref{c8} starts from the
2-dimensional $F$-manifolds with Euler fields and
discusses how many and which $(TE)$-structures exist over
each of them. It~turns out that the nilpotent $F$-manifold
$\NN_2$ with the Euler field
$E=t_1\paa_1+t_2^r\big(1+c_3t_2^{r-1}\big)\paa_2$ for $r\geq 2$
(case~\eqref{2.12} in~Theorem~\ref{t2.3}) does not have
any $(TE)$-structure over it if $c_3\neq 0$,
and it has no $(TE)$-structure with primitive Higgs field
over it if $c_3\neq 0$ or $r\geq 3$.
However, most 2-dimensional $F$-manifolds with Euler fields
have one or countably many families of $(TE)$-structures
with 1 or 2 parameters over them.

The third part of Section~\ref{c8} is the proof of Theorem~\ref{t8.5}.

In many aspects, the $(TE)$-structures of rank $2$ are probably
typical also for higher rank. But Section~\ref{c9} makes
one phenomenon explicit which arises only in rank $\geq 3$.
Section~\ref{c9} presents a~family of
rank 3 $(TE)$-structures with primitive Higgs fields
over a fixed 3-dimensional globally irreducible $F$-manifold
with nowhere regular Euler field, such that the family
has a {\it functional parameter}. The example is essentially
due to M.~Saito, it is a Fourier--Laplace transformation
of the main example in a preliminary version of~\cite{SaM17}
(though he considers only the bundle and connection over
a 2-dimensional submanifold of the $F$-manifold).

This paper has some overlap with~\cite{DH20-3}
and~\cite{DH20-2}. In~\cite{DH20-3} $(TE)$-structures over the 2-dimensional
$F$-manifold $\NN_2$ (with all possible Euler fields)
were studied. They are of generic types (Bra), (Reg) or (Log). In~\cite[Chapter~8]{DH20-2} $(TE)$-structures over the
2-dimensional $F$-manifolds $I_2(m)$ were studied.
They are of generic type (Sem).
However, in~\cite{DH20-3} and~\cite{DH20-2} the focus was on
$(TE)$-structures with primitive Higgs fields.
Those with generically primitive, but not primitive Higgs
fields were not considered. And the approach to the
classification was very different. It~relied on the formal
classification of rank $2$ $(T)$-structures in~\cite{DH20-1}.
The approach here is independent of these three papers.

\section[The two-dimensional $F$-manifolds and their Euler fields]{The two-dimensional $\boldsymbol{F}$-manifolds and their Euler fields}\label{c2}

$F$-manifolds were first defined in~\cite{HM99}.
Their basic properties were developed in~\cite{He02}.
An overview on them and on more recent results is given
in~\cite{DH20-2}.

\begin{Definition}\label{t2.1}\quad
\begin{enumerate}\itemsep=0pt
\item[$(a)$] An {\it $F$-manifold} $(M,\circ,e)$ (without Euler field)
is a complex manifold $M$ with a holomorphic
commutative and associative multiplication $\circ$
on the holomorphic tangent bundle $TM$, and with a
global holomorphic vector field $e\in\TT_M$ with
$e\circ=\id$ ($e$ is called a {\it unit field}),
which satisfies the following integrability condition:
\begin{gather}\label{2.1}
\Lie_{X\circ Y}(\circ)= X\circ\Lie_Y(\circ)+Y\circ\Lie_X(\circ)\qquad
\text{for}\quad X,Y\in\TT_M.
\end{gather}

\item[$(b)$] Given an $F$-manifold $(M,\circ,e)$, an {\it Euler field} on it is a global
vector field $E\in\TT_M$ with $\Lie_E(\circ)=\circ$.
\end{enumerate}
\end{Definition}

In this paper we are mainly interested in the 2-dimensional
$F$-manifolds and their Euler fields.
They were classified in~\cite{He02}.

\begin{Theorem}[{\cite[Theorem~4.7]{He02}}]\label{t2.2}
In dimension $2$, $($up to isomorphism$)$ the germs of
$F$-manifolds fall into three types:
\begin{enumerate}\itemsep=0pt
\item[$(a)$] The semisimple germ. It~is called $A_1^2$,
and it can be given as follows
\begin{gather*}
(M,0)=\big(\C^2,0\big)\qquad\text{with coordinates}\quad
u=(u_1,u_2)\quad\text{and}\quad e_k=\frac{\paa}{\paa u_k},
\\
e= e_1+e_2,\qquad e_j\circ e_k=\delta_{jk}\cdot e_j.
\end{gather*}
Any Euler field takes the shape
\begin{gather}\label{2.3}
E= (u_1+c_1)e_1+(u_2+c_2)e_2\qquad\text{for some}\quad
c_1,c_2\in\C.
\end{gather}

\item[$(b)$] Irreducible germs, which $($i.e., some holomorphic
representatives of them$)$ are at generic points semisimple.
They form a series $I_2(m)$, $m\in\Z_{\geq 3}$.
The germ of type $I_2(m)$ can be given as follows
\begin{gather*}
(M,0)=\big(\C^2,0\big)\qquad \text{with coordinates}\quad t=(t_1,t_2)\quad
\text{and}\quad\paa_k:=\frac{\paa}{\paa t_k},
\\
e=\paa_1,\qquad \paa_2\circ\paa_2 =t_2^{m-2}e.\label{2.4}
\end{gather*}
Any Euler field takes the shape
\begin{gather*}
E= (t_1+c_1)\paa_1 + \frac{2}{m}t_2\paa_2
\qquad\text{for some}\quad c_1\in\C.
\end{gather*}

\item[$(c)$] An irreducible germ, such that the multiplication is
everywhere irreducible. It~is called~$\NN_2$, and it
can be given as follows
\begin{gather*}
(M,0)=\big(\C^2,0\big)\qquad \text{with coordinates}\quad t=(t_1,t_2)
\quad \text{and}\quad
\paa_k:=\frac{\paa}{\paa t_k},\\
e=\paa_1,\quad \paa_2\circ\paa_2 =0.
\end{gather*}
Any Euler field takes the shape
\begin{gather}
E= (t_1+c_1)\paa_1 + g(t_2)\paa_2
\qquad\text{for some}\quad c_1\in\C\nonumber
\\
\text{and some function}\quad
g(t_2)\in\C\{t_2\}.\label{2.7}
\end{gather}
\end{enumerate}
\end{Theorem}

The family of Euler fields in~\eqref{2.7} on $\NN_2$
can be reduced by coordinate changes, which respect the
multiplication of $\NN_2$, to a family
with two continuous parameters and one discrete parameter.
This classification is proved in~\cite{DH20-3}. It~is recalled
in Theorem~\ref{t2.3}.
The group $\Aut(\NN_2)$ of automorphisms of the germ
$\NN_2$ of an $F$-manifold is the group of coordinate
changes of $\big(\C^2,0\big)$ which respect the multiplication
of $\NN_2$. It~is
\begin{gather*}
\Aut(\NN_2)=\{(t_1,t_2)\mapsto(t_1,\lambda(t_2))\,|\,
\lambda\in \C\{t_2\}\text{ with }\lambda'(0)\neq 0\text{ and }\lambda(0)=0\}.
\end{gather*}

\begin{Theorem}\label{t2.3}
Any Euler field on the germ $\NN_2$ of an $F$-manifold
can be brought by a coordinate change in $\Aut(\NN_2)$
to a unique one in the following family of Euler fields
\begin{gather}\label{2.9}
E = (t_{1} + c_1) \paa_{1} +\paa_{2},
\\
E = (t_{1}+c_1) \paa_{1},\label{2.10}
\\
E = (t_{1}+c_1) \paa_{1} +{c}_2 t_{2}\paa_{2},\label{2.11}
\\
E = (t_{1}+c_1) \paa_{1} + t_{2}^{r}\big(1 + c_3 t_{2}^{r-1}\big)\paa_{2},\label{2.12}
\end{gather}
where $c_1, c_3\in \C$, $c_2\in \C^*$ and $r\in
\Z_{\geq 2}.$
The group $\Aut(\NN_2,E)$ of coordinate changes of
$\big(\C^2,0\big)$ which respect the multiplication of $\NN_2$
and this Euler field is
\begin{gather}
\Aut(\NN_2,E)=\{(t_1,t_2)\mapsto (t_1,\gamma(t_2)t_2)\,|\,
\gamma\text{ as in~\eqref{2.14}}\},\nonumber
\\[1ex]
\def\arraystretch{1.3}
\begin{tabular}{c|c|c|c|c}
\hline
\text{Case} &~\eqref{2.9} &~\eqref{2.10} &~\eqref{2.11}&~\eqref{2.12}
\\ \hline
$\gamma\in$ & \{1\} & $\C\{t_2\}^*$ & $\C^*$
&$\big\{{\rm e}^{2\pi {\rm i} l/(r-1)}\,|\, l\in\Z \big\}$\label{2.14}
\\
\hline
\end{tabular}
\end{gather}
\end{Theorem}

A special class of $F$-manifolds, the regular $F$-manifolds,
is related to a result of Malgrange on universal unfoldings
of $(TE)$-structures, see Remarks~\ref{t3.17}.

\begin{Definition}[{\cite[Definition 1.2]{DH17}}]\label{t2.4}
A {\it regular $F$-manifold} is an $F$-manifold
$(M,\circ,e)$ with Euler field $E$ such that at each
$t\in M$ the endomorphism $E\circ|_t\colon T_tM\to T_tM$
is a regular endomorphism, i.e., it has for each eigenvalue
only one Jordan block.
\end{Definition}

\begin{Theorem}[{\cite[Theorem~1.3(ii)]{DH17}}]\label{t2.5}
For each regular endomorphism of a finite dimensional
$\C$-vector space, there is a unique (up to unique
isomorphism) germ $\big(M,t^0\big)$ of a regular $F$-manifold
such that $E\circ|_{t^0}$ is isomorphic to this endomorphism.
\end{Theorem}

\begin{Remarks}\label{t2.6}\qquad
\begin{enumerate}\itemsep=0pt
\item[$(i)$] For a normal form of this germ of an $F$-manifold,
see~\cite[Theorem~1.3(i)]{DH17}.

\item[$(ii)$]
In dimension 2, this theorem is an easy consequence
of Theorems~\ref{t2.2} and~\ref{t2.3}. The germs of
regular 2-dimensional $F$-manifolds are as follows:
\begin{enumerate}\itemsep=0pt
\item[$(a)$]
The germ $A_1^2$ in Theorem~\ref{t2.2}$(a)$ with any Euler field
$E=(u_1+c_1)e_1+(u_2+c_2)e_2$
as in~\eqref{2.3} with $c_1,c_2\in\C$, $c_1\neq c_2$.
\item[$(b)$]
The germ $\NN_2$ in Theorem~\ref{t2.2}$(c)$ with any Euler field
$E=(t_1+c_1)\paa_1+\paa_2$ as in~\eqref{2.9}
with $c_1\in\C$.
\end{enumerate}
\end{enumerate}
\end{Remarks}

\section[$(TE)$-structures in general]{$\boldsymbol{(TE)}$-structures in general}\label{c3}

\subsection{Definitions}

A $(TE)$-structure is a holomorphic vector bundle
on $\C\times M$, $M$ a complex manifold, with a~meromorphic connection $\nnn$ with a pole of
Poincar\'e rank 1 along $\{0\}\times M$ and no pole elsewhere.
Here we consider them together with the weaker notion
of $(T)$-structure and the more rigid notions
of a $(TL)$-structure and a $(TLE)$-structure.
The structures had been considered before in~\cite{HM04},
and they are related to structures in~\cite[Chapter~VII]{Sa02}
and in~\cite{Sa05}.

\begin{Definition}\label{t3.1}\qquad
\begin{enumerate}\itemsep=-2pt
\item[$(a)$] Definition of a {\it $(T)$-structure}
$(H\to\C\times M,\nabla)$:
$H\to\C\times M$ is a holomorphic vector bundle.
$\nabla$ is a map\vspace{-.5ex}
\begin{gather}\label{3.1}
\nabla\colon\ \OO(H)\to z^{-1}\OO_{\C\times M}\cdot \Omega^1_M\otimes \OO(H),
\end{gather}
which satisfies the Leibniz rule,\vspace{-.5ex}
\begin{gather*}
\nnn_X(a\cdot s)= X(a)\cdot s+a\cdot \nnn_X s
\qquad\text{for}\quad X\in\TT_M,\quad a\in\OO_{\C\times M},\quad s\in \OO(H),
\end{gather*}
and which is flat (with respect to $X\in\TT_M$,
not with respect to $\paa_z$),\vspace{-.5ex}
\begin{gather*}
\nnn_X\nnn_Y-\nnn_Y\nnn_X=\nnn_{[X,Y]}
\qquad\text{for}\quad X,Y\in \TT_M.
\end{gather*}
Equivalent: For any $z\in\C^*$, the restriction of $\nabla$ to
$H|_{\{z\}\times M}$ is a flat holomorphic connection.

\item[$(b)$] Definition of a {\it $(TE)$-structure}
$(H\to\C\times M,\nabla)\colon H\to\C\times M$ is a holomorphic vector bundle.
$\nabla$ is a flat connection on $H|_{\C^*\times M}$ with a pole
of Poincar\'e rank 1 along $\{0\}\times M$, so it is a map\vspace{-.5ex}
\begin{gather*}
\nabla\colon\ \OO(H)\to
\bigl(z^{-1}\OO_{\C\times M}\cdot\Omega^1_M
+z^{-2}\OO_{\C\times M}\cdot{\rm d}z\bigr)\otimes\OO(H)
\end{gather*}
which satisfies the Leibniz rule and is flat.

\item[$(c)$] Definition of a {\it $(TL)$-structure} $\big(H\to\P^1\times M,\nabla\big)\colon
H\to\P^1\times M$ is a holomorphic vector bundle.
$\nabla$ is a map\vspace{-.5ex}
\begin{gather*}
\nabla\colon\ \OO(H)\to \big(z^{-1}\OO_{\P^1\times M}+\OO_{\P^1\times M}\big)
\cdot \Omega^1_M\otimes \OO(H),
\end{gather*}
such that for any $z\in\P^1\setminus\{0\}$, the restriction
of $\nabla$ to $H|_{\{z\}\times M}$ is a flat connection. It~is called {\it pure} if for any $t\in M$ the restriction
$H|_{\P^1\times\{t\}}$ is a trivial holomorphic bundle on~$\P^1$.

\item[$(d)$] Definition of a {\it $(TLE)$-structure}
$\big(H\to\P^1\times M,\nabla\big)$:
It is simultaneously a $(TE)$-structure and a
$(TL)$-structure, where the connection $\nabla$
has a logarithmic pole along $\{\infty\}\times M$.
The $(TLE)$-structure is called {\it pure} if the
$(TL)$-structure is pure.
\end{enumerate}
\end{Definition}

\begin{Remark}\label{t3.2}
Here we write the data in Definition~\ref{t3.1}$(a)$--$(b)$
and the compatibility conditions between them in terms
of matrices.
Consider a $(TE)$-structure $(H\to\C\times M,\nnn)$
of rank $\rk H=r\in\N$.
We will fix the notations for a trivialization
of the bundle $H|_{U\times M}$ for some small neighborhood
$U\subset \C$ of 0. Trivialization means the choice of a
basis $\uuuu{v}=(v_1,\dots ,v_r)$ of the bundle $H|_{U\times M}$.
Also, we choose local coordinates $t=(t_1,\dots ,t_n)$
with coordinate vector fields $\paa_i=\paa/\paa t_i$ on $M$.
We write\vspace{-1ex}
\begin{gather}
\nnn\uuuu{v}=\uuuu{v}\cdot\Omega\qquad\text{with}\quad
\Omega = \sum_{i=1}^r z^{-1}\cdot A_i(z,t)\ddd t_i +z^{-2}B(z,t)\ddd z,
\label{3.4}
\\
A_i(z,t)= \sum_{k\geq 0}A_i^{(k)}z^k\in M_{r\times r}(\OO_{U\times M}),
\label{3.5}
\\
B(z,t)=\sum_{k\geq 0} B^{(k)}z^k\in M_{r\times r}(\OO_{U\times M}),
\label{3.6}
\end{gather}
with $A_i^{(k)},B^{(k)}\in M_{r\times r}(\OO_M)$,
but this dependence on $t\in M$ is usually not written
explicity.
The flatness $0=\ddd \Omega+\Omega\land\Omega$ of the connection
$\nnn$ says for $i,j\in\{1,\dots ,n\}$ with $i\neq j${\samepage
\begin{gather}\label{3.7}
0= z\paa_iA_j-z\paa_jA_i+[A_i,A_j],\\
0= z\paa_i B-z^2\paa_z A_i + zA_i + [A_i,B].
\label{3.8}
\end{gather}}

\pagebreak

\noindent
These equations split into the parts for the different powers
$z^k$ for $k\geq 0$ as follows \big(with $A_i^{(-1)}=B^{(-1)}=0$\big),
\begin{gather}\label{3.9}
0= \paa_iA_j^{(k-1)}-\paa_jA_i^{(k-1)}+\sum_{l=0}^k\big[A_i^{(l)},A_j^{(k-l)}\big],
\\
0= \paa_i B^{(k-1)}-(k-2)A_i^{(k-1)}+\sum_{l=0}^k\big[A_i^{(l)},B^{(k-l)}\big].
\label{3.10}
\end{gather}

In the case of a $(T)$-structure, $B$ and all
equations except~\eqref{3.4} which contain $B$ are dropped.

Consider a second $(TE)$-structure
$\big(\www H\to\C\times M,\www\nnn\big)$ of rank $r$ over $M$,
where all data except~$M$ are written with a tilde.
Let $\uuuu{v}$ and $\uuuu{\www v}$ be trivializations.
A holomorphic isomorphism from the first to the second
$(TE)$-structure maps $\uuuu{v}\cdot T$ to $\uuuu{\www v}$,
where
$T=T(z,t) =\sum_{k\geq 0}T^{(k)}z^k\in M_{r\times r}(\OO_{(\C,0)\times M})$
with $T^{(k)}\in M_{r\times r}(\OO_{M})$ and $T^{(0)}$
invertible satisfies
\begin{gather}\label{3.11}
\uuuu{v}\cdot\Omega\cdot T+\uuuu{v}\cdot\ddd T=
\nnn(\uuuu{v}\cdot T)=\uuuu{v}\cdot T\cdot\www\Omega.
\end{gather}
Equation~\eqref{3.11} says more explicitly
\begin{gather}\label{3.12}
0= z\paa_i T+A_i\cdot T-T\cdot\www A_i,
\\
0= z^2\paa_z T+B\cdot T-T\cdot \www B.\label{3.13}
\end{gather}
These equations split into the parts for the different
powers $z^k$ for $k\geq 0$ as follows (with $T^{(-1)}:=0$):
\begin{gather*}
0= \paa_i T^{(k-1)}+\sum_{l=0}^k \big(A_i^{(l)}\cdot T^{(k-l)}
-T^{(k-l)}\cdot\www A_i^{(l)}\big),
\\
0= (k-1)T^{(k-1)}+\sum_{l=0}^k\big(B^{(l)}\cdot T^{(k-l)}
-T^{(k-l)}\cdot \www B^{(l)}\big).
\end{gather*}
The isomorphism here fixes the base manifold $M$.
Such isomorphisms are called {\it gauge iso\-mor\-phisms}.
A general isomorphism is a composition of a
gauge isomorphism and a coordinate change on $M$
(a coordinate change induces an isomorphism of
$(TE)$-structures, see Lemma~\ref{t3.6}).
\end{Remark}

\begin{Remark}
In this paper we care mainly about $(TE)$-structures
over the 2-dimensional germs of $F$-manifolds with
Euler fields. For each of them except
$(\NN_2,E=(t_1+c_1)\paa_1)$, the group of coordinate
changes of $(M,0)=\big(\C^2,0\big)$ which respect the multiplication
and $E$ is quite small, see Theorem~\ref{t2.3}.
Therefore in this paper, we care mainly about {\it gauge}
isomorphisms of the $(TE)$-structures over these
$F$-manifolds with Euler fields.
\end{Remark}

\begin{Definition}
Let $M$ be a complex manifold.
\begin{enumerate}\itemsep=0pt
\item[$(a)$] The sheaf $\OO_M[[z]]$ on $M$ is defined by
$\OO_M[[z]](U):=\OO_M(U)[[z]]$
for an open subset $U\subset M$
(with $\OO_M(U)$ and $\OO_M[[z]](U)$ the sections of
$\OO_M$ and $\OO_M[[z]]$ on $U$).
Observe that the germ $(\OO_M[[z]])_{t^0}$ for $t^0\in M$
consists of formal power series $\sum_{k\geq 0}f_kz^k$ whose
coefficients $f_k\in\OO_{M,t^0}$ have a common
convergence domain. In~the case of $\big(M,t^0\big)=\big(\C^n,0\big)$ we write
$\OO_{\C^n}[[z]]_0=:\C\{t,z]]$.

\item[$(b)$] A {\it formal $(T)$-structure} over $M$ is a free
$\OO_M[[z]]$-module $\OO(H)$ of some finite rank $r\in\N$
together with a map
$\nnn$ as in~\eqref{3.1}, where $\OO_{\C\times M}$ is replaced
by $\OO_M[[z]]$ which satisfies properties analogous to
$\nnn$ in Definition~\ref{t3.1}$(a)$, i.e., the
Leibniz rule for $X\in\TT_M$, \mbox{$a\in\OO_M[[z]]$}, $s\in\OO(H)$
and the flatness condition for $X,Y\in\TT_M$.

A {\it formal $(TE)$-structure} is defined analogously:
In Definition~\ref{t3.1}$(b)$ one has to replace
$\OO_{\C\times M}$ by $\OO_M[[z]]$.
\end{enumerate}
\end{Definition}

\begin{Remark}
The formulas in Remark~\ref{t3.2} hold also
for formal $(T)$-structures and formal $(TE)$-structures
if one replaces $\OO_{\C\times M}$, $\OO_{U\times M}$
and $\OO_{(\C,0)\times M}$ by $\OO_M[[z]]$.
\end{Remark}

The following lemma is obvious.

\begin{Lemma}\label{t3.6}
Let $(H\to \C\times M,\nnn)$ be a $(TE)$-structure over $M$,
and let $\varphi\colon M'\to M$ be a~holo\-morphic map between
manifolds. One can pull back $H$ and $\nnn$ with
$\id\times\varphi\colon \C\times M'\to\C\times M$.
We~call the pull back $\varphi^*(H,\nnn)$. It~is a
$(TE)$-structure over $M'$. We say that the
pull back $\varphi^*(H,\nnn)$ is induced by the $(TE)$-structure
$(H,\nnn)$ via the map $\varphi$.
\end{Lemma}

\begin{Remarks}\qquad
\begin{enumerate}\itemsep=0pt
\item[$(i)$] We will give in Theorem~\ref{t8.5}
and in Corollary~\ref{t5.1} and Lemma~\ref{t5.2}$(iv)$
a~classification of rank $2$ $(TE)$-structures
over germs $\big(M,t^0\big)=\big(\C^2,0\big)$ of 2-dimensional manifolds
such that any rank $2$ $(TE)$-structure over a germ
$(M',s^0)$ is obtained as the pull back $\varphi^*(H,\nnn)$
of a~rank~2 $(TE)$-structure in the classification
via a holomorphic map $\varphi\colon \big(M',s^0\big)\to \big(M,t^0\big)$.

\item[$(ii)$] Here the behaviour of the $(TE)$-structure $(H,\nnn)$
over $\big(M,t^0\big)=\big(\C^2,0\big)$ with coordinates $t=(t_1,t_2)$
along $t_1$ is quite trivial. It~is convenient to
split it off. The next subsection does this in greater
generality.
\end{enumerate}
\end{Remarks}

\subsection[$(TE)$-structures with trace free pole part]{$\boldsymbol{(TE)}$-structures with trace free pole part}

\begin{Definition}\label{t3.8}
Let $(H\to\C\times M,\nnn)$ be a $(TE)$-structure.
Define the vector bundle $K:=H|_{\{0\}\times M}$ over $M$.
The {\it pole part} of the $(TE)$-structure is the endomorphism
$\UU\colon K\to K$ which is defined by
\begin{eqnarray}\label{3.16}
\UU:=\big[z^2\nnn_{\paa_z}\big]\colon\ K\to K.
\end{eqnarray}
The pole part is {\it trace free} if $\tr\UU=0$ on $M$.
\end{Definition}

The following lemma gives formal invariants of a $(TE)$-structure.

\begin{Lemma}\label{t3.9}
Let $(H\to\C\times M,\nnn)$ be a $(TE)$-structure
of rank $r\in\N$ over a manifold $M$.
By~a~for\-mal invariant of the $(TE)$-structure,
we mean an invariant of its formal isomorphism class.

\begin{enumerate}\itemsep=0pt
\item[$(a)$] Its pole part $\UU$, that means the pair $(K,\UU)$
up to isomorphism, is a formal invariant of the
$(TE)$-structure. Especially, the holomorphic functions
$\delta^{(0)}:=\det\UU\in\OO_M$ and
$\rho^{(0)}:=\frac{1}{r}\tr\UU\allowbreak\in\OO_M$ are formal invariants.

\item[$(b)$] For any $t^0\in M$, fix an $\OO_{M,t^0}$-basis $\uuuu{v}$
of $\OO(H)_{(0,t^0)}$,
consider the matrices in~\eqref{3.4}--\eqref{3.6},
consider the function $\rho^{(1)}:=\frac{1}{r}\tr B^{(1)}\in\OO_{M,t^0}$,
and consider the functions
$\delta^{(k)}\in\OO_{M,t^0}$ for~$k\in\N_0$
which are defined by writing $\det B$ as a power series
\begin{gather*}
\det B=\sum_{k\geq 0}\delta^{(k)}z^k.
\end{gather*}
Then the functions $\delta^{(1)}$ and $\rho^{(1)}$ are
independent of the choice of the basis $\uuuu{v}$.
The locally for any~$t^0$ defined functions $\delta^{(1)}$ and
$\rho^{(1)}$ glue to global holomorphic functions
$\delta^{(1)}\in\OO_M$ and $\rho^{(1)}\in\OO_M$.
They are formal invariants.
Furthermore, the function $\rho^{(1)}$ is constant
on any component of $M$.
\end{enumerate}
\end{Lemma}

\begin{proof} $\UU$, $\delta^{(0)}$, $\rho^{(0)}$ and $\delta^{(1)}$
are formal invariants because of~\eqref{3.13}:
$\www B = T^{-1}BT+z^2 \cdot T^{-1}\paa_z T.$
For $\rho^{(1)}$, observe additionally
\begin{gather*}
{\www B}^{(1)}= \big(T^{(0)}\big)^{-1}B^{(1)}T^{(0)}
+ \big[\big(T^{(0)}\big)^{-1}B^{(0)}T^{(0)}, \big(T^{(0)}\big)^{-1}T^{(1)}\big].
\end{gather*}
Recall also that the trace of a commutator of matrices is 0.
Therefore $\rho^{(1)}$ is a formal invariant.

Equation~\eqref{3.10} for $k=2$ implies
$\paa_i\tr \big(B^{(1)}\big)=0$, so the function $\rho^{(1)}$ is constant.
\end{proof}

The following lemma is obvious.

\begin{Lemma}\label{t3.10}
Let $(H\to\C\times M,\nnn)$ be a $(TE)$-structure
of rank $r\in\N$ over a manifold $M$.

\begin{enumerate}\itemsep=0pt
\item[$(a)$] Consider a holomorphic function $g\colon M\to\C$.
The trivial line bundle $H^{[1]}=\C\times\allowbreak(\C\times M)\to
\C\times M$ over $\C\times M$ with connection
$\nnn^{[1]}:={\rm d}+{\rm d}\big(\frac{g}{z}\big)$ defines a~$(TE)$-structure of rank~$1$ over~$M$, whose sheaf of sections
with connection is called $\EE^{g/z}$.

\item[$(b)$] $(\OO(H),\nnn)\otimes \EE^{g/z}$ for $g$ as in $(a)$
is a $(TE)$-structure.

\item[$(c)$] The $(TE)$-structure $\big(H^{[2]}\to\C\times M,\nnn^{[2]}\big)$ with
$\big(\OO\big(H^{[2]}\big),\nnn^{[2]}\big)=(\OO(H),\nnn)\otimes \EE^{\rho^{(0)}/z}$
has trace free pole part. And, of course,
$(\OO(H),\nnn)\cong \big(\OO\big(H^{[2]}\big),\nnn^{[2]}\big)\otimes\EE^{-\rho^{(0)}/z}$.
If $\uuuu{v}$ is a $\C\{t,z\}$-basis of $\OO(H)_0=\OO\big(H^{[2]}\big)_0$,
then the matrix valued connection $1$-forms $\Omega$ and
$\Omega^{[2]}$ of $\nnn$ and $\nnn^{[2]}$ with respect to this
basis satisfy
$\Omega=\Omega^{[2]}-{\rm d}\big(\frac{\rho^{(0)}}{z}\big)\cdot {\bf 1}_r$.

\item[$(d)$] $($Definition$)$ Consider a $(TE)$-structure
$\big(H^{[3]}\to \C\times M^{[3]},\nnn^{[3]}\big)$ with trace free
pole part. Consider the manifold $M^{[4]}:=\C\times M^{[3]}$
with $($local$)$ coordinates $t_1$ on $\C$ and $t'$ on $M^{[3]}$,
and the projection $\varphi^{[4]}\colon M^{[4]}\to M^{[3]}$,
$(t_1,t')\mapsto t'$.
Define the $(TE)$-structure $\big(H^{[4]}\to \C\times M^{[4]},
\nnn^{[4]}\big)$ with
$\big(\OO\big(H^{[4]}\big),\nnn^{[4]}\big)
=\big(\varphi^{[4]}\big)^*\big(\OO\big(H^{[3]}\big),\nnn^{[3]}\big)\otimes \EE^{t_1/z}$.

\item[$(e)$] If the $(TE)$-structure $\big(H^{[2]},\nnn^{[2]}\big)$ is induced
by the $(TE)$-structure $\big(H^{[3]},\nnn^{[3]}\big)$ via a map
$\varphi\colon M\to M^{[3]}$, then the $(TE)$-structure $(H,\nnn)$
is induced by the $(TE)$-structure $\big(H^{[4]},\nnn^{[4]}\big)$
via the map $(-\rho^{(0)},\varphi)\colon M\to M^{[4]}=\C\times M^{[3]}$.
\end{enumerate}
\end{Lemma}

Part $(c)$ allows to go from an arbitrary $(TE)$-structure to one
with trace free pole part, and to go back to the original one.
Part $(e)$ considers two $(TE)$-structures as in part $(c)$,
an original one and an associated one with trace free pole part.
If the associated one is induced by a third $(TE)$-structure,
then the original one is induced by a closely related
$(TE)$-structure with one parameter more.
Lemma~\ref{t3.11} continues Lemma~\ref{t3.10}.

\begin{Lemma}\label{t3.11}
Let $\big(H\to\C\times \big(M,t^0\big),\nnn\big)$ be a $(TE)$-structure
of rank $r\in\N$ over a germ $\big(M,t^0\big)$ of a manifold,
with coordinates $t=(t_1,\dots ,t_n)$ and $\paa_i:=\paa/\paa t_i$.
We suppose $t^0=0$ so that $\OO_{(\C\times M,(0,t^0))}=\C\{t,z\}$.
Recall the functions $\rho^{(0)}$ and $\rho^{(1)}$
of the $(TE)$-structure from Lemma~$\ref{t3.9}$.

Consider the $(TE)$-structure
$\big(H^{[2]},\nnn^{[2]}\big)$ from Lemma~$\ref{t3.10}$
with trace free pole part which is defined by
$\big(\OO\big(H^{[2]}\big),\nnn^{[2]}\big):= (\OO(H),\nnn)\otimes \EE^{\rho^{(0)}/z}$.
Here $H^{[2]}=H$, but
$\nnn^{[2]}=\nnn+{\rm d}\big(\frac{\rho^{(0)}}{z}\big)\cdot\id$.
The matrices $A_i$ and $B$ in~\eqref{3.4}--\eqref{3.6}
for the $(TE)$-structure $\big(H^{[2]},\nnn^{[2]}\big)$
of any $\C\{t,z\}$-basis $\uuuu{v}$ of $\OO\big(H^{[2]}\big)_0$
satisfy
\begin{gather}\label{3.18}
0=\tr A_i^{(0)}=\tr B^{(0)}=\tr \big(B^{(1)}-\rho^{(1)}{\bf 1}_r\big).
\end{gather}
The basis $\uuuu{v}$ can be chosen such that
the matrices satisfy
\begin{eqnarray}\label{3.19}
0=\tr A_i=\tr\big(B-z\rho^{(1)}{\bf 1}_r\big).
\end{eqnarray}
\end{Lemma}

\begin{proof}
Any $\C\{t,z\}$-basis $\uuuu{v}$ of
$\OO\big(H^{[2]}\big)_0=\OO(H)_0$ satisfies
\begin{gather*}
\tr B^{(0)}=\tr\UU^{[2]}=0 \qquad
\text{as}\ \big(H^{[2]},\nnn^{[2]}\big)\ \text{has trace free pole part,}
\\
\tr A_i^{(0)}=0 \qquad
\text{because of}\ \tr\paa_iB^{(0)}=\paa_i\tr B^{(0)}=0
\ \text{and}\ \eqref{3.10}\ \text{for}\ k=1 ,
\\
 \tr\big(B^{(1)}- \rho^{(1)}{\bf 1}_r\big)=0 \qquad
\text{by Lemma~\ref{t3.9} and especially},
\\
\Omega =\Omega^{[2]}-\ddd\bigg(\frac{\rho^{(0)}}{z}\bigg)\cdot{\bf 1}_r
=\Omega^{[2]}-\sum_{i=1}^n z^{-1}\frac{\paa\rho^{(0)}}{\paa t_i}
\cdot{\bf 1}_r\ddd t_i + z^{-2}\rho^{(0)}\cdot{\bf 1}_r \ddd z.
\end{gather*}

Start with an arbitrary basis $\uuuu{v}$,
consider the function
\begin{gather}\label{3.20}
g:=\frac{1}{r}\sum_{k\geq 2}
\frac{-\tr B^{(k)}}{k-1}\cdot z^{k-1}\in z\C\{t,z\},
\end{gather}
{\samepage
consider $T:={\rm e}^g\cdot{\bf 1}_r$, and
$\uuuu{\www v}:=\uuuu{v}\cdot T$.~\eqref{3.13} gives
\begin{gather*}
\www B=B+T^{-1}z^2\paa_z T = B+ \bigg({-}\sum_{k\geq 2}\tr B^{(k)}z^k\bigg)\cdot \frac{1}{r}{\bf 1}_r,
\end{gather*}
so $\tr \www B^{(k)}=0$ for $k\geq 2$, $\www B^{(1)}=B^{(1)}$,
$\www B^{(0)}=B^{(0)}$.}

Therefore now suppose $\tr \big(B-z\rho^{(1)}{\bf 1}_r\big)=0$.
\eqref{3.10} for $k\geq 3$ gives
$\tr A_i^{(l)}=0$ for $l\geq 2$, because
$\tr \paa_i B^{(l)}=\paa_i \tr B^{(l)}=0$.

Finally, we consider $T=T^{(0)}={\rm e}^h\cdot{\bf 1}_r$ for a
suitable function $h\in\C\{t\}$.
Then $\www B=B$, $\www A_i^{(k)}=A_i^{(k)}$ for $k\neq 1$,
and $\www A_i^{(1)}=A_i^{(1)}+\paa_i h\cdot{\bf 1}_r$.
So we need $h\in\C\{t\}$ with
$\paa_ih=-\frac{1}{r}\tr A_i^{(1)}$.
Such a function exists because~\eqref{3.9} for $k=2$
implies $\paa_i \tr A_j^{(1)}=\paa_j \tr A_i^{(1)}$.
We have obtained a basis $\uuuu{v}$ with
$\tr\big(B-z\rho^{(1)}{\bf 1}_r\big)=0$ and $\tr A_i=0$ for all $i$.
\end{proof}

\subsection[$(TE)$-structures over $F$-manifolds with Euler fields]{$\boldsymbol{(TE)}$-structures over $\boldsymbol{F}$-manifolds with Euler fields}

The pole part of a $(T)$-structure (or a $(TE)$-structure)
over $\C\times M$ along $\{0\}\times M$ induces a~Higgs bundle (together with $\UU$).
This is elementary (e.g.,~\cite{DH20-1} or~\cite{He03}).

\begin{Lemma}\label{t3.12}
Let $(H\to\C\times M,\nnn)$ be a $(T)$-structure.
Define $K:=H|_{\{0\}\times M}$.
Then $C:=[z\nnn]\in\Omega^1(M,{\rm End}(K))$,
more explicitly
\begin{eqnarray}\label{3.21}
C_X[a]:= [z\nnn_X a]\qquad\text{for}\quad X\in \TT_M,\quad a\in\OO(H),
\end{eqnarray}
is a Higgs field, i.e., the endomorphisms $C_X,C_Y\colon K\to K$
for $X,Y\in\TT_M$ commute.

If $(H\to\C\times M)$ is a $(TE)$-structure,
then its pole part $\UU\colon K\to K$ commutes with all
endomorphisms $C_X$, $X\in\TT_M$, short: $[C,\UU]=0$.
\end{Lemma}

\begin{Definition}\label{t3.13}
The Higgs field of a $(T)$-structure or a $(TE)$-structure
$(H\to\C\times M,\nnn)$ is {\it primitive}
if there is an open cover $\mathcal V$ of $M$ and for any
$U\in \mathcal V$
a section $\zeta_{U}\in {\mathcal O} (K\vert_{U})$
(called a~{\it local primitive section})
with the property that the map
$\TT_U\ni X\rightarrow C_{X} \zeta_{U} \in \OO(K)$
is an~isomorphism.
\end{Definition}

Theorems~\ref{t3.14} and~\ref{t3.16} show in two ways
that primitivity of a Higgs field is a good condition.
Theorem~\ref{t3.14} was first proved in
\cite[Theorem~3.3]{HHP10} (but see also
\cite[Lemma~10]{DH20-1}).

\begin{Theorem}\label{t3.14}
A $(T)$-structure $(H\rightarrow \mathbb{C}\times M, \nabla )$
with primitive Higgs field induces a multiplication $\circ$
on $TM$ which makes $M$ an $F$-manifold.
A $(TE)$-structure $(H\rightarrow \mathbb{C}\times M, \nabla )$
with primitive Higgs field induces in addition
a vector field $E$ on $M$, which, together with $\circ$,
makes $M$ an $F$-manifold with Euler field.
The multiplication $\circ$, unit field $e$ and Euler field $E$
$($the latter in the case of a $(TE)$-structure$)$, are defined by
\begin{gather*}
C_{X\circ Y} = C_{X} C_{Y},\qquad
C_{e} =\mathrm{Id},\qquad
C_{E} = - {\mathcal U},
\end{gather*}
where $C$ is the Higgs field defined by $\nabla$, and
$\mathcal U$ is defined in~\eqref{3.16}.
\end{Theorem}

Definition~\ref{t3.15} recalls the notions of an
{\it unfolding} and of a {\it universal unfolding}
of a $(TE)$-structure over a germ of a manifold
from~\cite[Definition 2.3]{HM04}. It~turns out that any $(TE)$-structure over a germ of a manifold
with primitive Higgs field is a universal unfolding of itself.
Interestingly, we will see in Theorem~\ref{t8.5}
also $(TE)$-structures which are universal unfoldings
of themselves, but where the Higgs bundle is only
generically primitive. Still in the examples which
we consider, the base manifold is an $F$-manifold with
Euler field globally.

Malgrange~\cite{Ma86} proved that a $(TE)$-structure
over a point $t^0$ has a universal unfolding
with primitive Higgs field if the endomorphism
$\UU\colon K_{t^0}\to K_{t^0}$ is {\it regular}, i.e.,
it has for each eigenvalue only one Jordan block.
A generalization was given by Hertling and Manin~\cite[Theorem~2.5]{HM04}.
Theorem~\ref{t3.16} cites in part $(b)$
the generalization.
Part $(a)$ is the special case of a $(TE)$-structure
with primitive Higgs field.
Part $(c)$ is the special case of a $(TE)$-structure
over a~point, Malgrange's result.

\begin{Definition}\label{t3.15}
Let $\big(H\to\C\times \big(M,t^0\big),\nnn\big)$ be a~$(TE)$-structure over a~germ $\big(M,t^0\big)$ of a~mani\-fold.
\begin{enumerate}\itemsep=0pt
\item[$(a)$] An {\it unfolding} of it is a $(TE)$-structure
$\big(H^{[1]}\to\C\times \big(M\times \C^{l_1},\big(t^0,0\big)\big),\nnn^{[1]}\big)$
over a germ $\big(M\times \C^{l_1},\big(t^0,0\big)\big)$
(for some $l_1\in\N_0$)
together with a fixed isomorphism
\begin{gather*}
i^{[1]}\colon \ \big(H\to\C\times \big(M,t^0\big),\!\nnn\big)
\to \big(H^{[1]}\to\C\times \big(M\times \C^{l_1},\big(t^0,0\big)\big),
\!\nnn^{[1]}\big)|_{\C\times (M\times\{0\},(t^0,0))} .
\end{gather*}

\item[$(b)$] One unfolding $\big(H^{[1]}\to\C\times\big(M\times \C^{l_1},
\big(t^0,0\big)\big),\nnn^{[1]},i^{[1]}\big)$ {\it induces}
a second unfolding $\big(H^{[2]}\to\C\times\big(M\times \C^{l_2},
\big(t^0,0\big)\big),\nnn^{[2]},i^{[2]}\big)$ if there
are a holomorphic map germ
\begin{gather*}
\varphi\colon\ \big(M\times \C^{l_2},\big(t^0,0\big)\big)\to \big(M\times \C^{l_1},\big(t^0,0\big)\big),
\end{gather*}
which is the identity on $M\times \{0\}$,
and an isomorphism $j$ from the second unfolding to the
pullback of the first unfolding by $\varphi$ such that
\begin{gather}\label{3.25}
i^{[1]}= j|_{\C\times (M\times \{0\},(t^0,0))}\circ i^{[2]}.
\end{gather}
(Then $j$ is uniquely determined by $\varphi$ and~\eqref{3.25}.)

\item[$(c)$] An unfolding is {\it universal} if it induces any unfolding
via a unique map $\varphi$.
\end{enumerate}
\end{Definition}

By definition of a universal unfolding in part $(c)$,
a $(TE)$-structure has (up to canonical isomorphism)
at most one universal unfolding, because any two universal
unfoldings induce one another by unique maps.

\begin{Theorem}\label{t3.16}\qquad

\begin{enumerate}\itemsep=0pt
\item[$(a)$] {\rm (\cite[Theorem~2.5]{HM04})}
A $(TE)$-structure over a germ $\big(M,t^0\big)$ with
primitive Higgs field is a~universal unfolding of itself.

\item[$(b)$] {\rm (\cite[Theorem~2.5]{HM04})}
Let $\big(H\to\C\times\big(M,t^0\big),\nnn\big)$ be a $(TE)$-structure over
a germ $\big(M,t^0\big)$ of a~manifold.
Let $\big(K\to \big(M,t^0\big),C\big)$ be the induced Higgs bundle over
$\big(M,t^0\big)$. Suppose that a~vec\-tor $\zeta_{t^0}\in K_{t^0}$
with the following properties exists:
\begin{enumerate}\itemsep=0pt\setlength{\leftskip}{0.25cm}
\item[$(IC)$] $($Injectivity condition$)$ The map
$C_\bullet \zeta_{t^0}\colon T_{t^0}M\to K_{t^0}$ is injective.
\item[$(GC)$] $($Generation condition$)$ $\zeta_{t^0}$ and
its images under iteration of the maps
$\UU|_{t^0}\colon K_{t^0}\to K_{t^0}$ and $C_X\colon K_{t^0}\to K_{t^0}$
for $X\in T_{t^0}M$ generate $K_{t^0}$.
\end{enumerate}
Then a universal unfolding of the $(TE)$-structure over
a germ $\big(M\times \C^l,\big(t^0,0\big)\big)$ $(l\in\N_0$ suitable$)$
exists. It~is unique up to isomorphism.
Its Higgs field is primitive.

\item[$(c)$] {\rm (\cite{Ma86})} A $(TE)$-structure over a point $t^0$ has
a universal unfolding with primitive Higgs field
if the endomorphism
$\big[z^2\nnn_{\paa_z}\big]=\UU\colon K_{t^0}\to K_{t^0}$
is regular, i.e., it has for each eigenvalue only one
Jordan block. In~that case, the germ of the $F$-manifold with Euler field
which underlies the universal unfolding,
is by definition $($Definition~$\ref{t2.4})$ regular.
\end{enumerate}
\end{Theorem}

\begin{Remarks}\label{t3.17}\qquad
\begin{enumerate}\itemsep=0pt
\item[$(i)$] As said above, the parts $(a)$ and $(c)$ are special cases
of part $(b)$.

\item[$(ii)$] A germ $\big(\big(M,t^0\big),\circ,e,E\big)$ of a regular $F$-manifold
is uniquely determined by the regular endomorphism
$E\circ|_{t^0}\colon T_{t^0}\to T_{t^0}$ (Theorem~\ref{t2.5}).

\item[$(iii)$] Consider the germ $(M,0)=\big(\C^2,0\big)$ of a 2-dimensional
$F$-manifold with Euler field $E$ in~Theorem~\ref{t2.2}. It~is regular if and only if
$E\circ|_{t=0}\notin\{\lambda\id\,|\, \lambda\in\C\}$. In~the semisimple case (Theorem~\ref{t2.2}$(a)$) this
holds if and only if $c_1\neq c_2$. In~the cases $I_2(m)$ $(m\geq 3)$ it does not hold. In~the case of $\NN_2$ with $E=t_1\paa_1+g(t_2)\paa_2$
it holds if and only if $g(0)\neq 0$.
See also Remark~\ref{t2.6}$(ii)$.

\item[$(iv)$] Theorem~\ref{t3.16}$(c)$ implies that a $(TE)$-structure
with primitive Higgs field
over a germ $\big(M,t^0\big)$ of a regular $F$-manifold with
Euler field is determined up to gauge isomorphism
by the restriction of the $(TE)$-structure to $t^0$.

\item[$(v)$] Lemma~\ref{t3.6}, Definition~\ref{t3.8},
Lemmata~\ref{t3.9}--\ref{t3.12}, Definition~\ref{t3.13},
Theorem~\ref{t3.14} and Definition~\ref{t3.15}
hold or make sense also for {\it formal} $(T)$-structures
or $(TE)$-structures. However, the proof of Theorem~\ref{t3.16}
used in an essential way {\it holomorphic} $(TE)$-structures.
We do not know whether Theorem~\ref{t3.16}
holds also for formal $(TE)$-structures.
\end{enumerate}
\end{Remarks}

\subsection{Birkhoff normal form}

\begin{Definition}\label{t3.18}
Let $(H\to\C\times M,\nnn)$ be a $(TE)$-structure
over a manifold $M$ with coordinates $t=(t_1,\dots ,t_n)$.
A {\it Birkhoff normal form} consists of a basis
$\uuuu{v}$ of $H$ and associated matrices
$A_1,\dots ,A_n,B$ as in~\eqref{3.4} such that
\begin{gather*}
A_1^{(k)}=\dots =A_n^{(k)}=0\qquad \text{for}\quad k\geq 1,\qquad
B^{(k)}=0\qquad\text{for}\quad k\geq 2,\qquad
\paa_i B^{(1)}=0.
\end{gather*}
\end{Definition}

\begin{Remarks}\label{t3.19}\qquad
\begin{enumerate}\itemsep=0pt
\item[$(i)$] Such a basis defines an extension of the
$(TE)$-structure to a pure $(TLE)$-structure.
Then it is a basis of the $(TLE)$-structure
whose restriction
to $\{\infty\}\times M$ is flat with respect to the
residual connection (that is just the restriction
of the connection $\nnn$ of the underlying $(TL)$-structure
to~$H|_{\{\infty\}\times M}$). Then the conditions
\eqref{3.9} and~\eqref{3.10} boil down to
\begin{gather}\label{3.27}
0= \big[A_i^{(0)},A_j^{(0)}\big],\qquad
\paa_iA_j^{(0)}=\paa_jA_i^{(0)},
\\
0= \big[A_i^{(0)},B^{(0)}\big],\qquad
\paa_iB^{(0)}+A_i^{(0)}+\big[A_i^{(0)},B^{(1)}\big],\qquad
0=\paa_iB^{(1)}.\label{3.28}
\end{gather}
Such a basis is relevant for the construction of
Frobenius manifolds (see, e.g.,~\cite{DH20-2}).

\item[$(ii)$] Vice versa, if the $(TE)$-structure has an extension
to a pure $(TLE)$-structure, then a~basis~$\uuuu{v}$
of the $(TLE)$-structure exists whose restriction to
$\{\infty\}\times M$ is flat with respect to the
residual connection.
Then this basis $\uuuu{v}$ and the associated matrices
form a Birkhoff normal form.

\item[$(iii)$] A Birkhoff normal form does not always exist.
But if a Birkhoff normal form of the restriction of
a $(TE)$-structure over $M$ to a point $t^0\in M$
exists, it extends to a Birkhoff normal form
of the $(TE)$-structure over the germ $\big(M,t^0\big)$~\cite[Chapter~VI, Theorem~2.1]{Sa02}
(or~\cite[Theorem~5.1(c)]{DH20-2}).

\item[$(iv)$] The problem whether a $(TE)$-structure over a point
has an extension to a pure $(TLE)$-structure is a
special case of the {\it Birkhoff problem},
which itself is a special case of the Riemann--Hilbert--Birkhoff
problem. The book~\cite{AB94} and Chapter~IV in~\cite{Sa02} are devoted to these problems and results
on them.
\end{enumerate}
\end{Remarks}

Here the following two results on the Birkhoff problem will
be useful. However, we will use part $(a)$ only in the
case of a $(TE)$-structure over a point $t^0$ with
a logarithmic pole at $z=0$, in which case it is trivial.

\begin{Theorem}\label{t3.20}
Let $\big(H\to\C\times\big\{t^0\big\},\nnn\big)$ be a $(TE)$-structure
over a point $t^0$.
\begin{enumerate}\itemsep=0pt
\item[$(a)$] {\rm (Plemely,~\cite[Chapter~IV, Corollary~2.6(1)]{Sa02})}
If the monodromy is semisimple, the $(TE)$-structure
has an extension to a pure $(TLE)$-structure.

\item[$(b)$] {\rm (Bolibroukh and Kostov,~\cite[Chapter~IV, Corollary 2.6(3)]{Sa02})}
The germ $\OO(H)_0\otimes_{\C\{z\}}\C\{z\}\big[z^{-1}\big]$ is a~$\C\{z\}\big[z^{-1}\big]$-vector space of dimension $r=\rk H\in\N$
on which $\nnn$ acts.

If no $\C\{z\}\big[z^{-1}\big]$ sub vector space of dimension
in $\{1,\dots ,r-1\}$ exists which is $\nnn$-invariant,
then the $(TE)$-structure has an extension to a
pure $(TLE)$-structure.
\end{enumerate}
\end{Theorem}

\subsection[Regular singular $(TE)$-structures]{Regular singular $\boldsymbol{(TE)}$-structures}

A $(TE)$-structure over a point $t^0$ is regular singular
if all its holomorphic sections have moderate growth
near 0. A good tool to treat this situation are
special sections of moderate growth, the
{\it elementary sections}.
Definition~\ref{t3.21} explains them and other basic notations.
We work with a simply connected manifold~$M$,
so that the only monodromy is the monodromy along closed
paths in the punctured $z$-plane going around~0.
One important case is the case of a germ $\big(M,t^0\big)$
of a manifold.
The most important case is the case of a point, $M=\big\{t^0\big\}$.

\begin{Definition}\label{t3.21}
Let $(H\to \C\times M,\nnn)$ be a $(TE)$-structure
of rank $r=\rk H\in\N$ over a simply connected manifold $M$.
We associate the following data to it.
\begin{enumerate}\itemsep=0pt
\item[$(a)$] $H':=H|_{\C^*\times M}$ is the flat bundle on
$\C^*\times M$.
$H^\infty$ denotes the $\C$-vector space (of dimen\-sion~$r$)
of global flat multivalued sections on $H'$.
Let the endomorphism $M^{\rm mon}\colon H^\infty\to H^\infty$
be the monodromy on it with semisimple part
$M^{\rm mon}_s$, unipotent part $M^{\rm mon}_u$
(with $M^{\rm mon}_sM^{\rm mon}_u=M^{\rm mon}_uM^{\rm mon}_s$),
nilpotent part $N^{\rm mon}:=\log M^{\rm mon}_u$
so that $M^{\rm mon}_u={\rm e}^{N^{\rm mon}}$,
and with eigenvalues in the finite set
$\Eig(M^{\rm mon})\subset\C$. For $\lambda\in\C$, let
\begin{gather*}
H^\infty_\lambda:=\ker\big(M^{\rm mon}_s-\lambda\id\colon H^\infty\to
H^\infty\big)
\end{gather*}
be the generalized eigenspace in $H^\infty$
of the monodromy with eigenvalue $\lambda$. It~is not $\{0\}$ if and only if $\lambda\in\Eig(M^{\rm mon})$.

\item[$(b)$] For $\alpha\in\C$, define the finite dimensional
$\C$-vector space $C^\alpha$
of the following global sections of $H'$,
\begin{gather*}
C^\alpha:= \big\{\sigma\in\OO(H')(\C^*)\,|\, (\nnn_{z\paa_z}-\alpha\id)^r(\sigma)=0,
\nnn_{\paa_i}(\sigma)=0\big\}
\end{gather*}
(where $t=(t_1,\dots ,t_n)$ are local coordinate and
$\paa_i$ are the coordinate vector fields).
Observe $z^k\cdot C^\alpha =C^{\alpha+k}$ for $k\in\Z$.
For each $\alpha$ the map
\begin{align*}
s(\cdot,\alpha)\colon\ H^\infty_{{\rm e}^{-2\pi {\rm i}\alpha}}&\to C^\alpha,
\\
A&\mapsto s(A,\alpha):=z^\alpha\cdot
{\rm e}^{-\log z\cdot N^{\rm mon}/2\pi {\rm i}}A(\log z),
\end{align*}
is an isomorphism. So, $C^\alpha\neq\{0\}$ if and only if
${\rm e}^{-2\pi {\rm i}\alpha}\in\Eig(M^{\rm mon})$.
The sections $s(A,\alpha)$ are called
{\it elementary sections}.

\item[$(c)$] A holomorphic section $\sigma$
of $H'|_{(U_1\setminus\{0\})\times U_2}$ for $U_1\subset \C$
a neighborhood of $0\in\C$ and $U_2\subset M$ open in $M$
can be written uniquely as an (in general infinite) sum
of elementary sections
$\operatorname{es}(\sigma,\alpha)\in \OO_{U_2}\cdot C^\alpha$
with coefficients in $\OO_{U_2}$,
\begin{gather*}
\sigma = \sum_{\alpha\colon {\rm e}^{-2\pi {\rm i}\alpha}\in\Eig(M^{\rm mon})}
\operatorname{es}(\sigma,\alpha).
\end{gather*}
In order to see this, choose numbers $\alpha_j\in\C$
and elementary sections $s_j\in C^{\alpha_j}$
for $j\in\{1,\dots ,r\}$ such that $s_1,\dots ,s_r$
form a global basis of $H'$. Then
\begin{gather}
\sigma =\sum_{j=1}^r a_js_j \qquad\text{with}\quad
a_j=a_j(z,t)= \sum_{k=-\infty}^\infty a_{kj}(t)z^k\in
\OO_{(U_1\setminus\{0\})\times U_2}.\label{3.33}
\end{gather}
Here~\eqref{3.33} is the expansion of $a_j$ as a Laurent
series in $z$ with holomorphic coefficients
$a_{kj}\in \OO_{U_2}$ in $t$. Then
\begin{eqnarray*}
\operatorname{es}(\sigma,\alpha)(z,t)&=& \sum_{j\colon \alpha-\alpha_j\in\Z}
a_{\alpha-\alpha_j,j}(t)z^{\alpha-\alpha_j}s_j.
\end{eqnarray*}

\item[$(d)$] A holomorphic section $\sigma$ as in $(c)$
has {\it moderate growth}
if a bound $b\in\R$ with $\operatorname{es}(\sigma,\alpha)=0$
for all $\alpha$ with $\Ree(\alpha)<b$ exists.
The sheaf $\VV^{>-\infty}$ on $\C\times M$
of all sections of moderate growth is
\begin{gather*}
\VV^{>-\infty}:= \bigoplus_{\alpha\colon -1<\Ree(\alpha)\leq 0}
\OO_{\C\times M}\big[z^{-1}\big]\cdot C^\alpha.
\end{gather*}
The {Kashiwara--Malgrange $V$-filtration} is given by the
locally free subsheaves for $r\in\R$,
 \begin{gather*}
\VV^{r}:= \bigoplus_{\alpha\colon \Ree(\alpha)\in[r,r+1[}
\OO_{\C\times M}\cdot C^\alpha.
\end{gather*}
\end{enumerate}

\end{Definition}

\begin{Definition}\qquad
\begin{enumerate}\itemsep=0pt
\item[$(a)$] A $(TE)$-structure $(H\to\C\times M ,\nnn)$
over a simply connected manifold $M$ is
{\it regular singular} if $\OO(H)\subset \VV^{>-\infty}$,
so if all its holomorphic sections have moderate growth
near 0.

\item[$(b)$] A $(TE)$-structure $(H\to\C\times M ,\nnn)$
over a simply connected manifold $M$ is {\it logarithmic} if
it has a basis $\uuuu{v}$ whose connection 1-form
$\Omega$ has a logarithmic pole along $\{0\}\times M$
(then this holds for any basis). In~the notations of~\eqref{3.4}--\eqref{3.6} that means
$A_i^{(0)}=B^{(0)}=0$.
Then the restriction of $\nnn$ to $K:=H|_{\{0\}\times M}$
is well-defined. It~is called the {\it residual connection}~$\nnn^{\rm res}$. The~{\it residue endomorphism} is
${\rm Res}_0=[\nnn_{z\paa_z}]\colon K\to K$.
\end{enumerate}
\end{Definition}

\begin{Theorem}[{well known, e.g.,~\cite[Theorems 7.10 and~8.7]{He02}}]\label{t3.23}
Let $(H\to\C\times M,\nnn)$ with $H|_{\C^*\times M}=H'$
be a logarithmic $(TE)$-structure over a simply
connected manifold.
\begin{enumerate}\itemsep=0pt
\item[$(a)$] The bundle $H$ has a global basis which consists of
elementary sections $s_j\in C^{\alpha_j}$,
$j\in\{1,\dots ,\rk H\}$, for some $\alpha_j\in\C$.
Especially, $(\OO(H),\nnn)=\varphi_{t^0}^*
\big(\OO\big(H|_{\C\times\{t^0\}}\big),\nnn\big)$
for any $t^0\in M$, where $\varphi_{t^0}\colon M\to\big\{t^0\big\}$
is the projection. So it is just the pull back of
a logarithmic $(TE)$-structure over a point.
Especially, it is a regular singular $(TE)$-structure.

\item[$(b)$] The residual connection $\nnn^{\rm res}$ is flat. In~the notations~\eqref{3.4}--\eqref{3.6}, its
connection $1$-form is $\sum_{i=1}^n A_i^{(1)}\ddd t_i$.
The residue endomorphism ${\rm Res}$ is $\nnn^{\rm res}$-flat. In~the notations~\eqref{3.4}--\eqref{3.6}, it~is given
by $B^{(1)}$.

\item[$(c)$] The endomorphism ${\rm e}^{-2\pi {\rm i}{\rm Res}_0}\colon K\to K$
has the same eigenvalues as the monodromy $M^{\rm mon}$,
but it might have a simpler Jordan block structure.
If no eigenvalues of ${\rm Res}_0$
differ by a~nonzero integer
$($nonresonance condition$)$ then ${\rm e}^{-2\pi {\rm i}{\rm Res}_0}$
has the same Jordan block structure as the mono\-dromy $M^{\rm mon}$.
\end{enumerate}
\end{Theorem}

\begin{Remarks}\quad
\begin{enumerate}\itemsep=0pt
\item[$(i)$] Part $(a)$ of Theorem~\ref{t3.23} implies that a
logarithmic $(TE)$-structure over a~simply connected
manifold $M$ is the pull back $\varphi^*((H,\nnn)|_{\C\times \{t^0\}})$
of its restriction to $t^0$ for any~$t^0\in M$.

\item[$(ii)$] In the case of a regular singular $(TE)$-structure
over a simply connected manifold $M$, one can choose
elementary sections $s_j\in C^{\alpha_j}$, $j\in\{1,\dots ,\rk H\}$,
for some $\alpha_j\in\C$, such that they form a basis of $H^*$
and such the extension to $\{0\}\times M$ which they define,
is a~logarithmic $(TE)$-structure.
Then the base change from any local basis of $H$ to the basis
$(s_1,\dots ,s_{\rk H})$ of this new $(TE)$-structure is meromorphic,
so the two $(TE)$-structures give the same meromorphic bundle.
This observation fits to the usual definition of meromorphic
bundle with regular singular pole.

\item[$(iii)$] The property of a section to have moderate growth,
is invariant under pull back. Therefore also the property
of a $(TE)$-structure to be regular singular is invariant
under pull back.
\end{enumerate}
\end{Remarks}

\subsection{Marked (\emph{TE})-structures and moduli spaces for them}

It is easy to give a $(TE)$-structure $(H\to\C\times M,\nnn)$
with nontrivial Higgs field and which is thus not the
pull back of the $(TE)$-structure over a point, such that
nevertheless the $(TE)$-structures over all points
$t^0\in M$ are isomorphic as abstract $(TE)$-structures.
Examples are given in Remark~\ref{t7.1}$(ii)$.
The existence of such $(TE)$-structures
obstructs the construction of nice Hausdorff
moduli spaces for $(TE)$-structures up to isomorphism.
The notion of a {\it marked} $(TE)$-structure hopefully
remedies this. However, in the moment, we have only results
in the regular singular cases.
Definition~\ref{t3.25} gives the notion of a {\it marked}
$(TE)$-structure. Definition~\ref{t3.26} defines {\it good}
families of marked regular singular $(TE)$-structures.
Definition~\ref{t3.28} defines a functor for such families.
Theorem~\ref{t3.29} states that this functor is represented
by a complex space. It~builds on results in~\cite[Chapter~7]{HS10}.
Several remarks discuss what is missing in the other cases
and what more we have in the regular singular rank $2$ case,
thanks to Theorems~\ref{t6.3},~\ref{t6.7} and~\ref{t8.5}.

\begin{Definition}\label{t3.25}\quad
\begin{enumerate}\itemsep=0pt
\item[$(a)$] A {\it reference pair} $\big(H^{{\rm ref},\infty},M^{\rm ref}\big)$
consists of a finite dimensional (reference)
$\C$-vector space $H^{{\rm ref},\infty}$
together with an automorphism $M^{\rm ref}$ of it.

\item[$(b)$] Let $M$ be a simply connected manifold.
A {\it marking} on a $(TE)$-structure
$(H\to\C\times M,\nnn)$ is an isomorphism
$\psi\colon (H^{\infty},M^{\rm mon})\to \big(H^{{\rm ref},\infty},M^{\rm ref}\big)$.
Here $H^{\infty}$ is (as in Definition~\ref{3.21})
the space of global flat multivalued sections
on the flat bundle $H':=H|_{\C^*\times M}$,
and $M^{\rm mon}$ is its monodromy.
$\big(H^{{\rm ref},\infty},M^{\rm ref}\big)$ is a reference pair.
The isomorphism $\psi$ of pairs means an isomorphism
$\psi\colon H^{\infty}\to H^{{\rm ref},\infty}$
with $\psi\circ M^{\rm mon}=M^{\rm ref}\circ\psi$.
A {\it marked} $(TE)$-structure is a $(TE)$-structure
with a marking.

{\sloppy\item[$(c)$] An isomorphism between two marked $(TE)$-structures
$\big(\big(H^{(1)},\nnn^{(1)}\big),\psi^{(1)}\big)$ and
$\big(\big(H^{(2)},\nnn^{(2)}\big),\psi^{(2)}\big)$ over the same base space
$M^{(1)}=M^{(2)}$ and with the same reference pair
$\big(H^{{\rm ref},\infty},M^{\rm ref}\big)$ is a gauge isomorphism $\varphi$
between the unmarked $(TE)$-structures such that
the induced isomorphism
$\varphi^{\infty}\colon H^{(1),\infty}\to H^{(2),\infty}$
is compatible with the marking,
\begin{gather*}
\psi^{(2)}\circ\varphi^{\infty}=\psi^{(1)}.
\end{gather*}

}

\item[$(d)$] ${\rm Set}^{(H^{{\rm ref},\infty},M^{\rm ref})}$ denotes the set of
marked $(TE)$-structures over a point with the same reference
pair $\big(H^{{\rm ref},\infty},M^{\rm ref}\big)$. Furthermore,
${\rm Set}^{(H^{{\rm ref},\infty},M^{\rm ref}),{\rm reg}}\subset
{\rm Set}^{(H^{{\rm ref},\infty},M^{\rm ref})}$
denotes the subset of marked regular singular
$(TE)$-structures over a point with the same reference pair
$\big(H^{{\rm ref},\infty},M^{\rm ref}\big)$.
\end{enumerate}
\end{Definition}

We hope that ${\rm Set}^{(H^{{\rm ref},\infty},M^{\rm ref})}$ carries for
any reference pair $\big(H^{{\rm ref},\infty},M^{\rm ref}\big)$ a natural
structure as a complex space.
Theorem~\ref{t3.29} says that this holds for
${\rm Set}^{(H^{{\rm ref},\infty},M^{\rm ref}),{\rm reg}}$
and that this space represents a functor of good families
of marked regular singular $(TE)$-structures.
Definition~\ref{t3.26} gives a notion of a
{\it family of marked $(TE)$-structures}
and the notion of a {\it good family of marked regular
singular $(TE)$-structures}.

\begin{Definition}\label{t3.26}
Let $X$ be a complex space.
Let $t^0$ be an abstract point and
$\varphi\colon X\to\big\{t^0\big\}$ be the projection.
Let $\big(H^{{\rm ref},\infty},M^{\rm ref}\big)$ be a reference pair.
Let $\big(H^{{\rm ref},*},\nnn^{\rm ref}\big)$ be a flat bundle
on~$\C^*\times\big\{t^0\big\}$ with monodromy $M^{\rm ref}$
and whose space of global flat multivalued sections
is identified with $H^{{\rm ref},\infty}$.
Let $i\colon \C^*\times X\hookrightarrow \C\times X$
be the inclusion.

$(a)$ A {\it family of marked $(TE)$-structures over $X$}
is a pair $(H,\psi)$ with the following properties:
\begin{enumerate}\itemsep=0pt\setlength{\leftskip}{0.65cm}
\item[$(i)$]
$H$ is a holomorphic vector bundle on $\C\times X$,
i.e., the linear space associated to a locally free sheaf
$\OO(H)$ of $\OO_{\C\times X}$-modules.
Denote $H':=H|_{\C^*\times X}$.
\item[$(ii)$]
$\psi$ is an isomorphism $\psi\colon H'\to \varphi^* H^{{\rm ref},*}$
such that the restriction of the induced flat connection on $H'$
to $\C^*\times\{x\}$ for any $x\in X$ makes $H|_{\C\times\{x\}}$
into a $(TE)$-structure over the point $x$, i.e.,
the connection has a pole of order $\leq 2$ on holomorphic
sections of~$H|_{\C\times\{x\}}$.
\end{enumerate}

$(b)$ Consider a family $(H,\psi)$ of marked
regular singular $(TE)$-structures over $X$.
The marking~$\psi$ induces for each $x\in X$
canonical isomorphisms
\begin{gather*}
\psi\colon\quad H^\infty(x)\to H^{{\rm ref},\infty},
\\
\psi\colon\quad C^{\alpha}(x)\to C^{{\rm ref},\alpha}\qquad
\big(\alpha\in\C\text{ with }{\rm e}^{-2\pi {\rm i}\alpha}\in\Eig\big(M^{\rm ref}\big)\big),
\nonumber
\\
\psi\colon\quad V^r(x)\to V^{{\rm ref},r} \qquad (r\in\R), \nonumber
\end{gather*}
where $H^{\infty}(x)$, $C^{\alpha}(x)$, $V^r(x)$ and
$C^{{\rm ref},\alpha}$, $V^{{\rm ref},r}$ are defined for the
$(TE)$-structure over $x$ respectively for $(H^{{\rm ref},*},\nnn)$
as in Definition~\ref{t3.21}.

The family $(H,\psi)$ is called {\it good} if some $r\in\R$
and some $N\in\N$ exist which satisfy
\begin{gather}\label{3.39}
\OO(H|_{\C\times\{x\}})_0\supset V^{r}(x)
\qquad\text{for any}\quad x\in X,
\\
\dim_\C \OO(H|_{\C\times\{x\}})_0/V^{r}(x)=N
\qquad\text{for any}\quad x\in X.\label{3.40}
\end{gather}
\end{Definition}

\begin{Remarks}
\qquad
\begin{enumerate}\itemsep=0pt
\item[$(i)$] The notion of a family of marked $(TE)$-structures
is too weak. For example, it contains the following
pathological family of logarithmic $(TE)$-structures
of rank 1 over $X:=\C$ (with coordinate $t$)
and with trivial monodromy. Write $s_0\in C^0$ for a
generating flat section. Define $H$ by
\begin{gather*}
\OO(H)=\OO_{\C\times X}\cdot \big(t+z^l\big)s_0\qquad \text{for some}\quad l\in\N.
\end{gather*}
The marked $(TE)$-structures over all points $t\in\C^*\subset X=\C$
are isomorphic and even equal, the one over $t=0$ is
different. The dimension $\OO\big(H|_{\C\times\{t\}}\big)/V^l(t)$
is equal to $l$ for $t\in \C^*$ and equal to $0$ for $t=0$.
Therefore this family is not good in the sense of
Definition~\ref{t3.26}$(b)$.
Also, $z\nnn_{\paa_z}\big(t+z^l\big)s_0=lz^ls_0$ is not a
section in $\OO(H)$, although for each fixed $t\in X$,
the restriction to $\C\times\{t\}$ is a section in
$\OO(H|_{\C\times \{t\}})$.

\item[$(ii)$] Theorem~\ref{t3.29} gives evidence that the notion
of a good family of marked regular singular $(TE)$-structures
is useful. However, it is not clear a priori whether any
regular singular $(TE)$-structure $(H\to\C\times M,\nnn)$
over a simply connected manifold $M$ is a good family
of marked regular singular $(TE)$-structures over $X=M$.
A marking can be imposed as~$M$ is simply connected.
Though the condition~\eqref{3.40} is not clear a priori.
Theorem~\ref{t8.5} will show this for regular singular
rank $2$ $(TE)$-structures. It~builds on Theorems~\ref{t6.3} and~\ref{t6.7} which show this for
regular singular rank $2$ $(TE)$-structures over $M=\C$.

\item[$(iii)$] For not regular singular $(TE)$-structures,
we do not see an easy replacement of condition~\eqref{3.40}.
Is the condition $z^2\nnn_{\paa_z}\OO(H)\subset\OO(H)$
useful?
\end{enumerate}
\end{Remarks}

\begin{Definition}\label{t3.28}
Fix a reference pair $\big(H^{{\rm ref},\infty},M^{\rm ref}\big)$.
\begin{enumerate}\itemsep=0pt
\item[$(a)$] Define the functor $\MM^{(H^{{\rm ref},\infty},M^{\rm ref}),{\rm reg}}$
from the category of complex spaces to the category of sets by
\begin{gather*}
\MM^{(H^{{\rm ref},\infty},M^{\rm ref}),{\rm reg}}(X)
:= \{(H,\psi)\,|\, (H,\psi) \text{ is a good family of marked regular}
\\ \hphantom{\MM^{(H^{{\rm ref},\infty},M^{\rm ref}),{\rm reg}}(X):= \{}
\text{singular }(TE)\text{-structures over }X\},
\end{gather*}
{\sloppy
and, for any morphism $f\colon Y\to X$ of complex spaces and any element
$(H,\psi)$ of $\MM^{(H^{{\rm ref},\infty},M^{\rm ref}),{\rm reg}}(X)$, define
$\MM^{(H^{{\rm ref},\infty},M^{\rm ref}),{\rm reg}}(f)(H,\psi):=f^*(H,\psi)$.

}

\item[$(b)$] Choose $r\in\R$ and $N\in \N$.
Define the functor $\MM^{(H^{{\rm ref},\infty},M^{\rm ref}),r,N}$
from the category of~comp\-lex spaces to the category of sets by
\begin{gather*}
\MM^{(H^{{\rm ref},\infty},M^{\rm ref}),r,N}(X)
:=\{(H,\psi)\,|\, (H,\psi) \text{ is a good family of marked regular }
\\ \hphantom{\MM^{(H^{{\rm ref},\infty},M^{\rm ref}),r,N}(X):=\{}
\text{singular} (TE)\text{-structures over }X
\text{ which satisfies~\eqref{3.39} }
\\ \hphantom{\MM^{(H^{{\rm ref},\infty},M^{\rm ref}),r,N}(X):=\{}
\text{and~\eqref{3.40} for the given }r\text{ and }N\},
\end{gather*}
{\sloppy
and, for any morphism $f\colon Y\to X$ of complex spaces and any element
$(H,\psi)$ of $\MM^{(H^{{\rm ref},\infty},M^{\rm ref}),r,N}(X)$, define
$\MM^{(H^{{\rm ref},\infty},M^{\rm ref}),r,N}(f)(H,\psi):=f^*(H,\psi)$.

}
\end{enumerate}
\end{Definition}

{\sloppy\begin{Theorem}\label{t3.29}
The functors $\MM^{(H^{{\rm ref},\infty},M^{\rm ref}),{\rm reg}}$ and
$\MM^{(H^{{\rm ref},\infty},M^{\rm ref}),r,N}$ are represented by
complex spaces, which are called
$M^{(H^{{\rm ref},\infty},M^{\rm ref}),{\rm reg}}$ and
$M^{(H^{{\rm ref},\infty},M^{\rm ref}),r,N}$. In~the case of $\MM^{(H^{{\rm ref},\infty},M^{\rm ref}),r,N}$, the complex space
has even the structure of a projective algebraic variety.
As sets $M^{(H^{{\rm ref},\infty},M^{\rm ref}),{\rm reg}}
={\rm Set}^{(H^{{\rm ref},\infty},M^{\rm ref}),{\rm reg}}$.
\end{Theorem}

}

\begin{proof} The proof for $\MM^{(H^{{\rm ref},\infty},M^{\rm ref}),r,N}$
can be copied from the proof of Theorem~7.3 in~\cite{HS10}.
Here it is relevant that $r$ and $N$ with~\eqref{3.39} and
\eqref{3.40} imply the existence of an $r_2\in \R$ with $r_2<r$
and
\begin{gather}\label{3.44}
V^{r_2}(x)\supset \OO(H|_{\C\times\{x\}})_0
\qquad\text{for any}\quad x\in X.
\end{gather}
In~\cite{HS10}, $(TERP)$-structures are considered.
\eqref{3.39} and~\eqref{3.44} are demanded there.~\eqref{3.40} is not
demanded there explicitly, but it follows from the properties
of the pairing there, and this is used in Lemma~7.2 in~\cite{HS10}.
The additional conditions of $(TERP)$-structures are not
essential for the arguments in the proof of Lemma~7.2 and
Theorem~7.3 in~\cite{HS10}. Therefore these proofs apply
also here and give the statements for
$\MM^{(H^{{\rm ref},\infty},M^{\rm ref}),r,N}$.

Let us call $(r,N)\in\R\times \N$ and
$\big(\www r,\www N\big)\in\R\times \N$ compatible if $n\in\Z$
with $\big(\www r,\www N\big)=\big(r+n,\allowbreak N+n\cdot\dim H^{{\rm ref},\infty}\big)$ exists. In~the case $n>0$,
$\MM^{(H^{{\rm ref},\infty},M^{\rm ref}),\www r,\www N}$ is a union of
$\MM^{(H^{{\rm ref},\infty},M^{\rm ref}),r,N}$ and additional irreducible
components. Thus for fixed $(r,N)$ the union
\[
\bigcup_{n\in\N}M^{(H^{{\rm ref},\infty},M^{\rm ref}),r+n,N+n\cdot\dim H^{{\rm ref},\infty}}
\]
is a complex space with in general countably many irreducible
(and compact) components.
And~$M^{(H^{{\rm ref},\infty},M^{\rm ref}),{\rm reg}}$ is the union
of these unions for all possible $(r,N)$
(as $\Eig(M^{\rm mon})$ is finite, in each interval of length
1, only finitely many $r$ are relevant).
\end{proof}

\begin{Remarks}
\quad
\begin{enumerate}\itemsep=0pt
\item[$(i)$] For each reference pair $\big(H^{{\rm ref},\infty},M^{\rm ref}\big)$
with $\dim H^{{\rm ref},\infty}=2$,
the representing complex space
$M^{(H^{{\rm ref},\infty},M^{\rm ref}),{\rm reg}}$
for the functor $\MM^{(H^{{\rm ref},\infty},M^{\rm ref}),{\rm reg}}$
is given in Theorem~\ref{t7.4}.
There the topological components are unions
$\bigcup_{n\in\N}M^{(H^{{\rm ref},\infty},M^{\rm ref}),r+n,N+n\cdot\dim H^{{\rm ref},\infty}}$
and have countably many irreducible components which are either isomorphic
to $\P^1$ or to the Hirzebruch surface $\F_2$ or to the variety $\www\F_2$
obtained by blowing down the $(-2)$-curve in $\F_2$.
The space
$M^{(H^{{\rm ref},\infty},M^{\rm ref}),{\rm reg}}$ is a union of
countably many copies of one topological component.

\item[$(ii)$] Corollary~\ref{t7.3} says that any marked rank $2$
regular singular $(TE)$-structure $(H\to\C\times M,\nnn,\psi)$
with reference pair $\big(H^{{\rm ref},\infty},M^{\rm ref}\big)$ is
a good family of marked regular singular $(TE)$-structures.
Therefore and because of Theorem~\ref{t3.29},
such a $(TE)$-structure is induced by a morphism
$\varphi\colon M\to M^{(H^{{\rm ref},\infty},M^{\rm ref}),{\rm reg}}$.
This is crucial for the usefulness of the
space $M^{(H^{{\rm ref},\infty},M^{\rm ref}),{\rm reg}}$.
We hope that Corollary~\ref{t7.3} and this implication
are also true for higher rank regular singular
$(TE)$-structures.
\end{enumerate}
\end{Remarks}

\section[Rank 2 $(TE)$-structures over a point]{Rank 2 $\boldsymbol{(TE)}$-structures over a point}\label{c4}

Here we will classify the rank $2$ $(TE)$-structures over
a point.

\subsection{Separation into 4 cases}
They separate naturally into 4 cases.

\begin{Definition}
Let $(H\to\C,\nnn)$ be a rank $2$ $(TE)$-structure
over a point $t^0=0$. Its formal invariants
$\delta^{(0)}$, $\rho^{(0)}$, $\delta^{(1)}$, $\rho^{(1)}$
from Lemma~\ref{t3.9} are complex numbers.
The eigenvalues of $-\UU$ are called $u_1,u_2\in\C$.
They are given by
$(x-u_1)(x-u_2)=x^2+2\rho^{(0)}x+\delta^{(0)}$.
We separate four cases:
\begin{enumerate}\itemsep=0pt\setlength{\leftskip}{0.3cm}
\item[(Sem)] $\UU$ has two different eigenvalues
$-u_1$ and $-u_2\in\C$, i.e.,
$0\neq \delta^{(0)}-\big(\rho^{(0)}\big)^2$.
\item[(Bra)] $\UU$ has only one eigenvalue
\big(which is then $\rho^{(0)}$\big)
and one $2\times 2$ Jordan block,
and $\delta^{(1)}-2\rho^{(0)}\rho^{(1)}\neq 0$.
\item[(Reg)] $\UU$ has only one eigenvalue
\big(which is then $\rho^{(0)}$\big)
and one $2\times 2$ Jordan block,
and $\delta^{(1)}-2\rho^{(0)}\rho^{(1)}= 0$.
\item[(Log)] $\UU=\rho^{(0)}\cdot\id$.
\end{enumerate}
Here (Sem) stands for {\it semisimple},
(Bra) for {\it branched}, (Reg) for {\it regular singular}
and (Log) for {\it logarithmic}.
\end{Definition}

\begin{Remark}\label{t4.2}
Rank 2 $(TE)$-structures over a point are richer than the
germs of mermorphic rank $2$ vector bundles with a pole of
order 2. Though the four types above are closely related
to the formal classification of the latter ones by their
slopes (the notion of slopes is developed for example in
\cite[Section~5]{Sa93}). In~rank $2$, three slopes are possible,
slope 1, slope $\frac{1}{2}$ and slope~0.
Slope 1 corresponds to the type (Sem), slope $\frac{1}{2}$
to type (Bra), and slope 0 to the types (Reg) and (Log).
\end{Remark}

First we will treat the semisimple case (Sem).
Then the cases (Bra), (Reg) and (Log)
will be considered together.
Lemma~\ref{t4.9} will justify the names (Bra) and (Reg).
Finally, the three cases (Bra), (Reg) and (Log) will be
treated one after the other.
The following lemma gives some first information.
Its proof is straightforward.

\begin{Lemma}\label{t4.3}
Let $(H\to\C,\nnn)$ be a rank $2$ $(TE)$-structure over a point.
Denote by $\big(\www H\to\C,\www\nnn\big)$ the $(TE)$-structure
with trace free pole part with
$\big(\OO\big(\www H\big),\www\nnn\big)=(\OO(H),\nnn)\otimes\EE^{\rho^{(0)}/z}$
from Lemma~$\ref{t3.10}(b)$
$\big($called $\big(H^{[2]}\to\C,\nnn^{[2]}\big)$ there$\big)$,
and denote its invariants from Lemma~$\ref{t3.9}$
by~$\www\UU$,~$\www\delta^{(0)}$, $\www\rho^{(0)},
\www\delta^{(1)}$, $\www\rho^{(1)}$. Then
\begin{gather*}
\www\UU= \UU-\rho^{(0)}\id,
\\
\www\delta^{(0)}= \delta^{(0)}-\big(\rho^{(0)}\big)^2,\qquad
\www\rho^{(0)}=0,\nonumber
\\
\www\delta^{(1)}= \delta^{(1)}-2\rho^{(0)}\rho^{(1)},\qquad
\www\rho^{(1)}=\rho^{(1)}.
\end{gather*}
$\big(\www H\to\C,\www\nnn\big)$ is of the same type $($Sem$)$ or $($Bra$)$
or $($Reg$)$ or $($Nil$)$ as $(H\to\C,\nnn)$.
The following table characterizes of which type the
$(TE)$-structures $(H\to\C,\nnn)$ and
$\big(\www H\to\C,\www\nnn\big)$ are
\[
\def\arraystretch{1.5}
\begin{tabular}{c|c|c|c}
\hline
$($Sem$)$ & $($Bra$)$ &$($Reg$)$ &$($Log$)$
\\ \hline
$\www\delta^{(0)}\neq 0$ & $\www\delta^{(0)}=0$, $\www\delta^{(1)}\neq 0$
&$\www\delta^{(0)}=\www \delta^{(1)}=0$, $\www\UU\neq 0$ & $\www\UU=0$
\\
\hline
\end{tabular}
\]
Especially, $\www\UU=0$ implies $\www\delta^{(0)}=
\www\delta^{(1)}=0$.
\end{Lemma}

\subsection{The case (Sem)}

A $(TE)$-structure over a point with a semisimple
endomorphism $\UU$ with pairwise different eigenvalues
is formally isomorphic to a socalled {\it elementary model},
and its holomorphic isomorphism class is determined
by its Stokes structure. These two facts are well known.
A good reference is~\cite[Chapter~II, Sections~5 and~6]{Sa02}.
The older reference~\cite{Ma83a} considers only
the underlying meromorphic bundle, so
$\big(\OO(H)_0\otimes_{\C\{z\}}\C\{z\}\big[z^{-1}\big],\nnn\big)$.

In order to formulate the result for rank $2$ $(TE)$-structures
more precisely, we need some notation.

\begin{Definition}\label{t4.4}
Choose numbers $u_1,u_2,\alpha_1,\alpha_2\in\C$.
Consider the flat bundle $H'\to\C^*$ with flat connection
$\nnn$ and a basis $\uuuu{f}=(f_1,f_2)$
of global flat multivalued sections $f_1$ and $f_2$
with the monodromy
\begin{gather*}
\uuuu{f}\big(z\cdot {\rm e}^{2\pi {\rm i}}\big)= \uuuu{f}(z)
\begin{pmatrix}{\rm e}^{-2\pi {\rm i} \alpha_1} & 0 \\
0 & {\rm e}^{-2\pi {\rm i} \alpha_2} \end{pmatrix}\!.
\end{gather*}
The new basis $\uuuu{v}=(v_1,v_2)$ which is defined by
\begin{gather*}
\uuuu{v}(z)= \uuuu{f}(z)
\begin{pmatrix}{\rm e}^{u_1/z}z^{\alpha_1} & 0 \\
0 & {\rm e}^{u_2/z}z^{\alpha_2}\end{pmatrix}
\end{gather*}
(for some choice of $\log(z)$) is univalued. It~defines a $(TE)$-structure with
\begin{gather*}
z^2\nnn_{\paa_z}\uuuu{v}=\uuuu{v}\cdot B
\qquad\text{and}\qquad
B=\begin{pmatrix}-u_1+ z\alpha_1 & 0 \\
0 & -u_2+z\alpha_2\end{pmatrix}\!.
\end{gather*}
This $(TE)$-structure is called an {\it elementary model}.
The numbers $\alpha_1$ and $\alpha_2$ are called
the {\it regular singular exponents}.
The formal invariants
$\delta^{(0)},\rho^{(0)},\delta^{(1)},\rho^{(1)}\in\C$
of the $(TE)$-structure and the tuple
$(u_1,u_2,\alpha_1,\alpha_2)$ (up to joint exchange of the indices 1 and 2) are equivalent because of
\begin{gather}\label{4.6}
\delta^{(0)}-\big(\rho^{(0)}\big)^2= -\frac{1}{4}(u_1-u_2)^2,\qquad
\rho^{(0)}=-\frac{u_1+u_2}{2},
\\
\delta^{(1)}-2\rho^{(0)}\rho^{(1)}=\frac{u_1-u_2}{2}(\alpha_1-\alpha_2),\qquad
\rho^{(1)}=\frac{\alpha_1+\alpha_2}{2}.\label{4.7}
\end{gather}
Therefore also the tuple $(u_1,u_2,\alpha_1,\alpha_2)$
(up to joint exchange of the indices 1 and 2)
is a formal invariant of the $(TE)$-structure.
\end{Definition}

\begin{Theorem}\label{t4.5}\qquad

\begin{enumerate}\itemsep=0pt
\item[$(a)$] Any rank $2$ $(TE)$-structure over a point with
endomorphism $\UU$ with two different
eigen\-va\-lues is formally isomorphic to a unique
elementary model in Definition~$\ref{t4.4}$.
Here $-u_1$ and~$-u_2$ are the eigenvalues of $\UU$.

\item[$(b)$] The $(TE)$-structure in $(a)$ is up to holomorphic
isomorphism determined by the numbers~$u_1$, $u_2$, $\alpha_1$, $\alpha_2$ and two more numbers
$s_1,s_2\in\C$, the Stokes parameters. It~is holomorphically isomorphic to the elementary
model to which it is formally isomorphic if and only if~$s_1=s_2=0$.

\item[$(c)$] Any such tuple $(u_1,u_2,\alpha_1,\alpha_2,s_1,s_2)
\in \big(\C^2\setminus \{(u_1,u_1)\,|\, u_1\in\C\}\big)\times\C^4$
determines such a~$(TE)$-structure.
A second tuple $\big(\www u_1,\www u_2,\www\alpha_1,
\www \alpha_2,\www s_1,\www s_2\big)
\neq (u_1,u_2,\alpha_1,\alpha_2,s_1,s_2)$
determines an isomorphic $(TE)$-structure if and only
if $\big(\www u_1,\www u_2,\www\alpha_1,\www \alpha_2,\www s_1,\www s_2\big)
=(u_2,u_1,\alpha_2,\alpha_1,s_2,s_1)$.
\end{enumerate}
\end{Theorem}

Part $(a)$ follows for example from
\cite[Chapter~II, Theorem~5.7]{Sa02} together with
\cite[Chapter~II, Remark 5.8]{Sa02} (Theorem~5.7 considers
only the underlying meromorphic bundle;
Remark 5.8 takes care of the holomorphic bundle).
For the parts $(b)$ and $(c)$, one needs to deal in detail
with the Stokes structure. We will not do it here,
as the semisimple case is not central in this paper.
We~refer to~\cite[Chapter~II, Sections~5 and~6]{Sa02}
or to~\cite{HS11}.

\begin{Remarks}
\quad
\begin{enumerate}\itemsep=0pt
\item[$(i)$] Malgrange's unfolding result, Theorem~\ref{t3.16}$(c)$,
applies to these $(TE)$-structures.
Such a $(TE)$-structure has a unique universal unfolding.
The parameters $(\alpha_1,\alpha_2,s_1,s_2)$ are constant,
the parameters $(u_1,u_2)$ are local coordinates
on the base space. The base space is an $F$-manifold
of type $A_1^2$ with Euler field $E=u_1e_1+u_2e_2$.
See Remark~\ref{t5.3}$(iii)$.

\item[$(ii)$] We do not offer normal forms for the $(TE)$-structures
in Theorem~\ref{t4.5} for three reasons: (1)~As said in $(i)$,
the $(TE)$-structures above unfold uniquely to
$(TE)$-structures over germs of $F$-manifolds. In~that sense they are easy to deal with.
(2)~It looks difficult to write down normal forms.
(3)~Normal forms should be considered
together with the Stokes parameters, and the corresponding
Riemann--Hilbert map from the space of {\it monodromy data}
$(\alpha_1,\alpha_2,s_1,s_2)$ to a space of
parameters for normal forms should be studied.
This is a nontrivial project,
which does not fit into the main aims of this paper.
\end{enumerate}
\end{Remarks}

\subsection{Joint considerations on the cases (Bra), (Reg) and (Log)}

\begin{Notation}
We shall use the following matrices,
\begin{gather*}
C_1:={\bf 1}_2,\qquad
C_2:=\begin{pmatrix}0&0\\1&0\end{pmatrix}\!,\qquad
D:=\begin{pmatrix}1&0\\0&-1\end{pmatrix}\!,\qquad
E:=\begin{pmatrix}0&1\\0&0\end{pmatrix}\!,
\end{gather*}
and the relations between them,
\begin{gather*}
C_2^2=0,\qquad D^2=C_1,\qquad E^2=0,
\\[.5ex]
C_2D=C_2=-DC_2,\qquad [C_2,D]=2C_2,
\\
C_2E=\frac{1}{2}(C_1-D),\qquad
EC_2=\frac{1}{2}(C_1+D),\qquad
[C_2,E]=-D,
\\
DE=E=-ED,\qquad [D,E]=2E.
\end{gather*}
\end{Notation}

Consider a $(TE)$-structure $(H\to\C,\nnn)$
over a point with $\UU$ of type (Bra), (Reg) or (Log).
Then $\UU$ has only one eigenvalue, which is $\rho^{(0)}\in\C$.
We can and will restrict to $\C\{z\}$-bases $\uuuu{v}$
of~$\OO(H)_0$ such that the matrix
$B\in M_{2\times 2}(\C\{z\})$ with
$z^2\nnn_{\paa_z}\uuuu{v}=\uuuu{v}\cdot B$ has the shape
\begin{gather}\label{4.13}
B= b_1C_1+b_2C_2+zb_3D+zb_4E\qquad\text{with}\quad
b_1,b_2,b_3,b_4\in\C\{z\}.
\end{gather}
Write as in Remark~\ref{t3.2} $B=\sum_{k\geq 0}B^{(k)}z^k$
with $B^{(k)}\in M_{2\times 2}(\C)$, and write
\begin{gather}\label{4.14}
b_j=\sum_{k\geq 0}b_j^{(k)}z^k\qquad\text{with}\quad
b_j^{(k)}\in\C\qquad\text{for}\quad j\in\{1,2\},
\\
zb_j=\sum_{k\geq 1}b_j^{(k)}z^{k}\qquad\text{with}\quad
b_j^{(k)}\in\C\qquad\text{for}\quad j\in\{3,4\}.\label{4.15}
\end{gather}
Then the formal invariants $\delta^{(0)}$, $\rho^{(0)}$,
$\delta^{(1)}$ and $\rho^{(1)}$ of Lemma~\ref{t3.9} are
given by
\begin{gather*}
\rho^{(0)}=b_1^{(0)},\qquad \rho^{(1)}=b_1^{(1)},
\\
\delta^{(0)}-\big(\rho^{(0)}\big)^2=0,\qquad
\delta^{(1)}-2\rho^{(0)}\rho^{(1)}= -b_2^{(0)}b_4^{(1)}.
\end{gather*}
We are in the case (Bra) if $b_2^{(0)}\neq 0$ and
$b_4^{(1)}\neq 0$, in the case (Reg) if $b_2^{(0)}\neq 0$
and $b_4^{(1)}=0$,
and~in the case (Log) if $b_2^{(0)}=0$.

Consider $T\in {\rm GL}_2(\C\{z\})$ and the new basis
$\uuuu{\www v}=\uuuu{v}\cdot T$ and its matrix
$\www B =\sum_{k\geq 0}\www B^{(k)}z^k$ with~$z\nnn_{z\paa_z}\uuuu{\www v}=\uuuu{\www v}\cdot\www B$.
Write
\begin{gather*}
T=\tau_1C_1+\tau_2C_2+\tau_3D+\tau_4E\qquad\text{with}\quad
\tau_j=\sum_{k\geq 0}\tau_j^{(k)}z^k,\quad
\tau_j^{(k)}\in\C.
\end{gather*}
Then $\www B$ is determined by~\eqref{3.13}, which is
\begin{align}
0&=z^2\paa_zT+B\cdot T-T\cdot\www B\nonumber
\\
&= C_1\bigg(z^2\paa_z\tau_1+\big(b_1-\www b_1\big)\tau_1
+z\big(b_4-\www b_4\big)\frac{\tau_2}{2}+z\big(b_3-\www b_3\big)\tau_3
+\big(b_2-\www b_2\big)\frac{\tau_4}{2}\bigg)\nonumber
\\
&\phantom{=}+ C_2\big(z^2\paa_z \tau_2+\big(b_2-\www b_2\big)\tau_1
+\big(b_1-\www b_1\big)\tau_2 + z\big({-}b_3-\www b_3\big)\tau_2
+ \big(b_2+\www b_2\big)\tau_3\big)\nonumber
\\
&\phantom{=}+ D\bigg(z^2\paa_z\tau_3+z\big(b_3-\www b_3\big)\tau_1
+z\big(b_4+\www b_4\big)\frac{\tau_2}{2}+\big(b_1-\www b_1\big)\tau_3
+ \big({-}b_2-\www b_2\big)\frac{\tau_4}{2}\bigg)\nonumber
\\
&\phantom{=}+ E\big(z^2\paa_z\tau_4 +z\big(b_4-\www b_4\big)\tau_1
+z\big({-}b_4-\www b_4\big)\tau_3+\big(b_1-\www b_1\big)\tau_4
+z\big(b_3+\www b_3\big)\tau_4\big).\label{4.19}
\end{align}
We will use this quite often in order to construct
or compare normal forms. The following immediate corollary
of the proof of
Lemma~\ref{t3.11} provides a reduction of $b_1$.

\begin{Corollary}\label{t4.8}
The base change matrix $T={\rm e}^g\cdot C_1$ with $g$ as in
\eqref{3.20} leads to $\www b_j$ with
\begin{gather*}
\www b_1=b_1^{(0)}+zb_1^{(1)}=\rho^{(0)}+z \rho^{(1)},\qquad
\www b_2=b_2,\qquad \www b_3=b_3,\qquad \www b_4=b_4,
\end{gather*}
\end{Corollary}

From now on we will work in this section only with bases
$\uuuu{v}$ with $b_1=\rho^{(0)}+z \rho^{(1)}$. This is justified
by Corollary~\ref{t4.8}.

Furthermore, we will consider from now on in this section
mainly $(TE)$-structures with trace free pole part
\big(Definition~\ref{t3.8}, $\rho^{(0)}=\frac{1}{2}\tr\UU=0$\big).
See Lemmata~\ref{t3.10} and~\ref{t3.11} for the
relation to the general case.

The next lemma separates the cases (Bra) and (Reg).

\begin{Lemma}\label{t4.9}
Consider a $(TE)$-structure over a point with
$\UU$ of type $($Bra$)$ or type $($Reg$)$ and with trace free pole part
$($so $\UU$ is nilpotent but not~$0)$.

The $(TE)$-structure is regular singular if and only
if it is of type $($Reg$)$. If it is of type $($Bra$)$,
then the pullback of
$\OO(H)_0\otimes_{\C\{z\}}\C\{z\}\big[z^{-1}\big]$
by the map $\C\to\C$, $x\mapsto x^4=z$,
is the space of germs at~$0$ of sections of a meromorphic bundle
on $\C$ with a meromorphic connection
with an order~$3$ pole at~$0$ with semisimple pole part with
eigenvalues $\kappa_1$ and $\kappa_2=-\kappa_1$ with
$-\frac{1}{4}\kappa_1^2=\delta^{(1)}\in\C^*$.
Thus $\kappa_1^2$ is a formal invariant of the
$(TE)$-structure of type $($Bra$)$.
\end{Lemma}

\begin{proof}
Consider a $\C\{z\}$-basis $\uuuu{v}$ of $\OO(H)_0$
such that its matrix $B$ is as in~\eqref{4.13}
and such that $b_1=z\rho^{(1)}$. This is possible
by Corollary~\ref{t4.8} and the assumption $\rho^{(0)}=0$.
As $\UU$ is nilpotent, but not 0, $b_2^{(0)}\neq 0$.
Now $\delta^{(1)}=-b_2^{(0)}b_4^{(1)}$, so
$\delta^{(1)}\neq 0\iff b_4^{(1)}\neq 0$.

Consider the case $b_4^{(1)}\neq 0$, and consider
the pullback of the $(TE)$-structure by the map
$\C\to\C$, $x\mapsto x^4=z$.
Then $\frac{\ddd z}{z}=4\frac{\ddd x}{x}$ and
$z\paa_z=\frac{1}{4}x\paa_x$ and
\begin{gather}
\nnn_{x\paa_x}\uuuu{v}=
\uuuu{v}\cdot 4\sum_{k\geq 0}B^{(k)}x^{4k-4},\nonumber
\\
\nnn_{x\paa_x}\big(\uuuu{v}\cdot x^D\big)=
\big(\uuuu{v}\cdot x^D\big)4\nonumber
\\ \hphantom{\nnn_{x\paa_x}\big(\uuuu{v}\cdot x^D\big)=}
{}\times\!\bigg(x^{-2}\!\sum_{k\geq 0}\big(b_2^{(k)}C_2+b_4^{(k+1)}E\big)x^{4k}
\!+\rho^{(1)}C_1+\bigg(\frac{1}{4}
+\!\sum_{k\geq 0}b_3^{(k+1)}x^{4k}\!\bigg)D\bigg).\!\label{4.22}
\end{gather}
One sees a pole of order 3 with matrix
$4\big(b_2^{(0)}C_2+b_4^{(1)}E\big)$ of the pole part. It~is tracefree and has the eigenvalues
$\kappa_1$ and $\kappa_2=-\kappa_1$ with
$\kappa_1^2=4b_2^{(0)}b_4^{(1)}\in\C^*$.
This shows the claims in the case $b_4^{(1)}\neq 0$.

Consider the case $b_4^{(1)}= 0$, and consider
the pullback of the $(TE)$-structure by the map
$\C\to\C$, $x\mapsto x^2=z$.
Then $\frac{\ddd z}{z}=2\frac{\ddd x}{x}$ and
$z\paa_z=\frac{1}{2}x\paa_x$ and
\begin{gather}
\nnn_{x\paa_x}\uuuu{v}=
\uuuu{v}\cdot 2\sum_{k\geq 0}B^{(k)}x^{2k-2},\nonumber
\\
\nnn_{x\paa_x}\big(\uuuu{v}\cdot x^D\big)=\nonumber
\big(\uuuu{v}\cdot x^D\big)2
\\ \hphantom{\nnn_{x\paa_x}\big(\uuuu{v}\cdot x^D\big)=}
{}\times\bigg(\rho^{(1)}C_1 + \frac{1}{2}D
+\sum_{k\geq 0}
\big(b_2^{(k)}C_2+b_4^{(k+2)}E+b_3^{(k+1)}D\big)x^{2k}\bigg).\label{4.24}
\end{gather}
One sees a logarithmic pole. Therefore also the
sections $v_1$ and $v_2$ have moderate growth,
and~the $(TE)$-structure is regular singular.
\end{proof}

\begin{Remark}
The two transformations in~\eqref{4.22}
(for the case (Bra)) and~\eqref{4.24} (for the case (Reg))
are special cases of a systematic development of such
ramified gauge transformations in~\cite{BV83}
(a short description is given in~\cite[p.~17]{Va96}).
The basic idea goes back to the shearing transformations
of Fabry (see~\cite{Fa85} and \cite[p.~4]{Va96}).
\end{Remark}

\subsection{The case (Bra)}

The following theorem gives complete control on the
$(TE)$-structures over a point of the type (Bra).
Here $\Eig(M^{\rm mon})\subset\C$ is the set of eigenvalues
of the monodromy of such a $(TE)$-structure
(it has 1 or 2 elements).

\begin{Theorem}\label{t4.11}\quad

\begin{enumerate}\itemsep=0pt
\item[$(a)$] Consider a $(TE)$-structure over a point of the type $($Bra$)$.
The formal invariants $\rho^{(0)}$, $\rho^{(1)}$ and
$\delta^{(1)}\in\C$ from Lemma~$\ref{t3.9}$ and the
set $\Eig(M^{\rm mon})$
are invariants of its isomorphism class.
Together they form a complete set of invariants.
That means, the isomorphism class of the $(TE)$-structure
is determined by these invariants.

\item[$(b)$] Any such $(TE)$-structure has a $\C\{z\}$-basis
$\uuuu{v}$ of $\OO(H)_0$ such that its matrix is
in Birkhoff normal form, and more precisely, the matrix
$B$ has the shape
\begin{gather}\label{4.25}
B= \big(\rho^{(0)}+z\rho^{(1)}\big)C_1 + b_2^{(0)}C_2 + zb_3^{(1)}D+ zb_4^{(1)}E,
\end{gather}
where $b_2^{(0)},b_4^{(1)}\in\C^*$ and $b_3^{(1)}\in\C$ satisfy
$-b_2^{(0)}b_4^{(1)}=\delta^{(1)}-2\rho^{(0)}\rho^{(1)}$ and
$\Eig(M^{\rm mon})=\big\{{\rm e}^{-2\pi {\rm i} (\rho^{(1)}\pm b_3^{(1)})}\big\}$.
\end{enumerate}
\end{Theorem}

\begin{Remarks}\label{t4.12}\quad
\begin{enumerate}\itemsep=0pt
\item[$(i)$] Because of part $(a)$, two Birkhoff normal forms as in~\eqref{4.25}
with data $\big(\rho^{(0)},\rho^{(1)},b_2^{(0)},\allowbreak b_3^{(1)},b_4^{(1)}\big)$
and $\big(\www\rho^{(0)},\www\rho^{(1)},\www b_2^{(0)},\www b_3^{(1)}, \www b_4^{(1)}\big)$
give isomorphic $(TE)$-structures
if and only if $\www \rho^{(0)}=\rho^{(0)}$,
$\www\rho^{(1)}=\rho^{(1)}$,
$\www b_2^{(0)}\www b_4^{(1)}=b_2^{(0)} b_4^{(1)}$
and $\www b_3^{(1)}\in \big\{{\pm} b_3^{(1)}+k\,|\, k\in\Z\big\}$.
However, the pure $(TLE)$-structures which they
define, are isomorphic only if additionally
$\www b_3^{(1)}\in\big\{{\pm} b_3^{(1)}\big\}$.

\item[$(ii)$] We could restrict to Birkhoff normal forms with
$b_2^{(0)}=1$ or with $b_4^{(1)}=1$. Though in view of the
$(TE)$-structures in the 4th case in Theorem~\ref{t6.3}
we prefer not to do that.
\end{enumerate}
\end{Remarks}

\begin{proof}[Proof of Theorem~\ref{t4.11}] The proof has 3 steps.

\medskip\noindent
{\it Step 1:} Birkhoff normal forms exist. In~order to show this, it is sufficient to prove
the hypothesis in Theorem~\ref{t3.20}$(b)$.
The hypothesis says that the germ of the meromorphic
bundle underlying a $(TE)$-structure of type (Bra)
is irreducible.
A proof by calculation is given in the proof of~Lemma~28
in~\cite{DH20-3}.
Though this is also well known as the rank is two
and the slope is $\frac{1}{2}$ (see Remark~\ref{t4.2}).

\medskip\noindent
{\it Step 2:} Analysis of the Birkhoff normal forms.
The matrix $B$ of a Birkhoff normal form can be chosen
with $b_1=\rho^{(0)}+z\rho^{(1)}$ because of Corollary~\ref{t4.8}. Then it has the shape
\begin{gather*}
B= \big(\rho^{(0)}+z\rho^{(1)}\big)C_1+\big(b_2^{(0)}+zb_2^{(1)}\big)C_2
+zb_3^{(1)}D+zb_4^{(1)}E
\end{gather*}
with $b_2^{(0)}\neq 0$ and $b_4^{(1)}\neq 0$.

Consider the new basis $\uuuu{\www v}=\uuuu{v}\cdot T$
and its matrix $\www B$, where
\begin{gather}\label{4.29}
T=C_1+\tau_2^{(0)}C_2 \qquad\text{for some}\quad \tau_2^{(0)}\in\C.
\end{gather}
Equation~\eqref{4.19} gives
\begin{gather*}
0= \big(b_1-\www b_1\big)+z\big(b_4^{(1)}-\www b_4\big)\frac{\tau_2^{(0)}}{2},
\qquad
0= \big(b_2-\www b_2\big)+\big(b_1-\www b_1\big)\tau_2^{(0)}+z\big({-}b_3^{(1)}-\www b_3\big)\tau_2^{(0)},
\\
0= \big(b_3^{(1)}-\www b_3\big)+\big(b_4^{(1)}+\www b_4\big)\frac{\tau_2^{(0)}}{2},
\qquad
0= \big(b_4^{(1)}-\www b_4\big),
\end{gather*}
so
\begin{gather}
\www b_4= \www b_4^{(1)}=b_4^{(1)},\qquad
\www b_1=b_1,\qquad
\www b_3= \www b_3^{(1)}=b_3^{(1)}+b_4^{(1)}\tau_2^{(0)},
\nonumber
\\
\www b_2^{(0)}= b_2^{(0)},\qquad
\www b_2^{(1)}= b_2^{(1)}-2b_3^{(1)}\tau_2^{(0)}-b_4^{(1)}\big(\tau_2^{(0)}\big)^2.
\label{4.30}
\end{gather}
$\tau_2^{(0)}$ can be chosen such that $\www b_2^{(1)}=0$.
Then the Birkhoff normal form $\www B$ has the shape in~\eqref{4.25}.

Suppose now that $B$ has this shape, so $b_2=b_2^{(0)}$.
The choice $\tau_2^{(0)}:=-2b_3^{(1)}/b_4^{(1)}$ in~\eqref{4.29} leads to
\begin{gather}\label{4.31}
\www b_1=b_1,\qquad
\www b_2=b_2,\qquad
\www b_4=b_4\qquad\text{and}\qquad
\www b_3=-b_3.
\end{gather}

Consider the new basis $\uuuu{\www v}=\uuuu{v}\cdot T$
and its matrix $\www B$, where
\begin{gather*}
T=C_1+\tau_3^{(0)}D \qquad\text{for some}\quad \tau_3^{(0)}\in\C\setminus\{\pm 1\}.
\end{gather*}
Equation~\eqref{4.19} gives
\begin{gather*}
0= \big(b_1-\www b_1\big)+z\big(b_3^{(1)}-\www b_3\big)\tau_3^{(0)},
\\
0= \big(b_2^{(0)}-\www b_2\big)+\big(b_2^{(0)}+\www b_2\big)\tau_3^{(0)},
\\
0= z\big(b_3^{(1)}-\www b_3\big)+\big(b_1-\www b_1\big)\tau_3^{(0)},
\\
0= \big(b_4^{(1)}-\www b_4\big) + \big({-}b_4^{(1)}-\www b_4\big)\tau_3^{(0)},
\end{gather*}
so
\begin{gather}\label{4.33}
\www b_1= b_1,\qquad
\www b_3=b_3^{(1)},\qquad
\www b_2= b_2^{(0)}\frac{1+\tau_3^{(0)}}{1-\tau_3^{(0)}},\qquad
 \www b_4= b_4^{(1)}\frac{1-\tau_3^{(0)}}{1+\tau_3^{(0)}}.
\end{gather}
So, in a Birkhoff normal form in~\eqref{4.25},
one can change $b_2^{(0)}$ and $b_4^{(1)}$ arbitrarily
with constant product $b_2^{(0)}b_4^{(1)}$ and without
changing $b_1=\rho^{(0)}+z\rho^{(1)}$ and $b_3^{(1)}$.

Consider the new basis $\uuuu{\www v}=\uuuu{v}\cdot T$
and its matrix $\www B$, where
\begin{gather*}
T=\big(1+z\tau_1^{(1)}\big)C_1+\tau_2^{(0)}C_2+z\tau_3^{(1)}D
+z\tau_4^{(1)}E,\qquad \text{for some}\quad
\tau_1^{(1)},\tau_2^{(0)},
\tau_3^{(1)},\tau_4^{(1)}\in\C.
\end{gather*}
We are searching for coefficients
$\tau_1^{(1)},\tau_2^{(0)},\tau_3^{(1)},\tau_4^{(1)}\in\C$
such that
\begin{gather}\label{4.35}
\www b_1=b_1,\qquad
\www b_2=b_2,\qquad
\www b_4=b_4,\qquad
\www b_3=b_3+\varepsilon\qquad \text{with}\quad \varepsilon=\pm 1.
\end{gather}
Under these constraints,
\eqref{4.19} gives
\begin{gather*}
0= \tau_1^{(1)}-\varepsilon \tau_3^{(1)},
\\
0= \big({-}2b_3^{(1)}-\varepsilon\big)\tau_2^{(0)}+ 2b_2^{(0)}\tau_3^{(1)},
\\
0= z\tau_3^{(1)}-\varepsilon\big(1+z\tau_1^{(1)}\big)+b_4^{(1)}\tau_2^{(0)} -b_2^{(0)}\tau_4^{(1)},
\\
0=\tau_4^{(1)} -2b_4^{(1)}\tau_3^{(1)}+\big(2b_3^{(1)}+\varepsilon\big)\tau_4^{(1)}.
\end{gather*}
With $\tau_1^{(1)}=\varepsilon\tau_3^{(1)}$, these
equations boil down to the inhomogeneous linear system
of equations
\begin{gather}\label{4.36}
\begin{pmatrix}0\\ \varepsilon\\ 0 \end{pmatrix} =
\begin{pmatrix} -2b_3^{(1)}-\varepsilon & 2b_2^{(0)} & 0 \\
b_4^{(1)} & 0 & -b_2^{(0)} \\
0 & -2b_4^{(1)} & 2b_3^{(1)}+\varepsilon+1 \end{pmatrix}
\begin{pmatrix} \tau_2^{(0)}\\ \tau_3^{(1)}\\ \tau_4^{(1)}
\end{pmatrix}\!.
\end{gather}
The determinant of the $3\times 3$ matrix is
$-2 b_2^{(0)}b_4^{(1)}\neq 0$. Therefore the system~\eqref{4.36}
has a unique solution $\big(\tau_2^{(0)},\tau_3^{(1)},
\tau_4^{(1)}\big)^t$.
Thus a new basis $\uuuu{\www v}=\uuuu{v}\cdot T$
with~\eqref{4.35} exists.

Iterating this construction, one finds that one can
change the matrix $B$ in~\eqref{4.25} by a holomorphic
base change to a matrix $\www B$ with
\begin{gather}\label{4.37}
\www b_1=b_1,\qquad
\www b_2=b_2,\qquad
\www b_4=b_4,\qquad
\www b_3=b_3+k,
\end{gather}
for any $k\in\Z$.

Putting together~\eqref{4.30},~\eqref{4.31},~\eqref{4.33}
and~\eqref{4.37}, one sees that two Birkhoff normal forms
as in~\eqref{4.25}
with data $\big(\rho^{(0)},\rho^{(1)},b_2^{(0)},b_3^{(1)},b_4^{(1)}\big)$
and $\big(\www \rho^{(0)},\www \rho^{(1)},\www b_2^{(0)},\www b_3^{(1)},
\www b_4^{(1)}\big)$ give isomorphic $(TE)$-structures
if $\www \rho^{(0)}=\rho^{(0)}$, $\www\rho^{(1)}=\rho^{(1)}$,
$\www b_2^{(0)}\www b_4^{(1)}=b_2^{(0)} b_4^{(1)}$
and $\www b_3^{(1)}\in \big\{{\pm}\, b_3^{(1)}+k\,|\, k\in\Z\big\}$.
This shows {\it if} in Remark~\ref{t4.12}$(i)$.

\medskip\noindent
{\it Step 3:} Discussion of the invariants.
By Lemma~\ref{t3.9}, $\rho^{(0)}$, $\rho^{(1)}$ and $\delta^{(1)}$
are even formal inva\-ri\-ants of the $(TE)$-structure.
The set $\Eig(M^{\rm mon})$ is obviously an invariant of the
isomorphism class of the $(TE)$-structure.

The Birkhoff normal form in~\eqref{4.25} gives a pure
$(TLE)$-structure with a logarithmic pole at~$\infty$.
From its pole part at~$\infty$
and Theorem~\ref{t3.23}$(c)$ one reads off
\begin{gather*}
\Eig(M^{\rm mon})=\big\{{\rm e}^{-2\pi {\rm i}(\rho^{(1)}\pm b_3^{(1)})}\big\}.
\end{gather*}
As $\rho^{(1)}$ is an invariant of the $(TE)$-structure,
also the set $\big\{{\pm}\, b_3^{(1)}+k\,|\, k\in\Z\big\}$
is an invariant of the $(TE)$-structure.

\looseness=1
Together with Step 2, this shows {\it only if} in
Remark~\ref{t4.12}$(i)$ and all statements in Theo\-rem~\ref{t4.11}.
\end{proof}

\begin{Corollary}
The monodromy of a $(TE)$-structure over a point
of the type $($Bra$)$ has a $2\times 2$ Jordan block if its
eigenvalues coincide $\big($equivalently, if
$b_3^{(1)}\in\frac{1}{2}\Z$ for some $($or any$)$
Birkhoff normal form in Theorem~$\ref{t4.11}(b)\big)$.
\end{Corollary}

\begin{proof}
Consider a $(TE)$-structure over a point
of the type (Bra) such that the eigenvalues of its monodromy
coincide. Then for any Birkhoff normal form in
Theorem~\ref{t4.11}$(b)$ $b_3^{(1)}\in\frac{1}{2}\Z$,
and one can choose a Birkhoff normal form with
$b_3^{(1)}\in\big\{0,-\frac{1}{2}\big\}$.
The induced pure $(TLE)$-structure
has at $\infty$ a logarithmic pole, and its residue
endomorphism $[\nnn_{\www z\paa_{\www z}}]$, where
$\www z=z^{-1}$, is given by the
matrix $-\big(\rho^{(1)}C_1+b_3^{(1)}D+b_4^{(1)}E\big)$.

In the case $b_3^{(1)}=0$, the nonresonance condition
in Theorem~\ref{t3.23}$(c)$ is satisfied,
so Theorem~\ref{t3.23}$(c)$ can be applied.
Because of $b_4^{(1)}\neq 0$,
the monodromy has a $2\times 2$ Jordan block.

In the case $b_3^{(1)}=-\frac{1}{2}$, the meromorphic
base change
\begin{gather*}
\uuuu{\www v}:=\uuuu{v}\cdot\begin{pmatrix}z&0\\0&1
\end{pmatrix}
\end{gather*}
gives the new connection matrix
\begin{gather*}
\www B=\bigg(\rho^{(0)}+z\bigg(\rho^{(1)}+\frac{1}{2}\bigg)\bigg)C_1+zb_2^{(0)}C_2+b_4^{(1)}E.
\end{gather*}
Again, the pole at $\infty$ is logarithmic.
Now the nonresonance condition in Theorem~\ref{t3.23}$(c)$
is satisfied. Because of $b_2^{(0)}\neq 0$,
the monodromy has a $2\times 2$ Jordan block.
\end{proof}

For $(TE)$-structures of the type (Bra), formal isomorphism
is coarser than holomorphic isomorphism.

\begin{Lemma}\label{t4.14}
Consider a $(TE)$-structure over a point of the type $($Bra$)$.
By Lemma~$\ref{t3.9}$, the numbers $\rho^{(0)}$, $\rho^{(1)}$
and $\delta^{(1)}$ are formal invariants of the
$(TE)$-structure.

\begin{enumerate}\itemsep=0pt
\item[$(a)$] The set $\Eig(M^{\rm mon})$ and the equivalent set
$\big\{{\pm}\, b_3^{(1)}+k\,|\, k\in\Z\big\}$ are holomorphic invariants,
but not formal invariants.

\item[$(b)$]
The $(TE)$-structure with Birkhoff normal form in~\eqref{4.25}
is formally isomorphic to the $(TE)$-structure
with Birkhoff normal form in~$\eqref{4.25}$
with the same values $\rho^{(0)}$, $\rho^{(1)}$, $b_2^{(0)}$ and $b_4^{(1)}$,
but with an arbitrary $\www b_3^{(1)}$.
\end{enumerate}
\end{Lemma}

\begin{proof}
Part $(a)$ follows from part $(b)$.
For the proof of part (b), we have to find
$T\in {\rm GL}_2(\C[[z]])$ such that $T$, $B$ in~\eqref{4.25}
and
\begin{gather*}
\www B=\big(\rho^{(0)}+z\rho^{(1)}\big)C_1+b_2^{(0)}C_2+z\www b_3^{(1)}D+zb_4^{(1)}E
\end{gather*}
with $\www b_3^{(1)}\in\C$ arbitrary satisfy~\eqref{4.19}.
Here~\eqref{4.19} says
\begin{gather*}
0= z\paa_z\tau_1 + \big(b_3^{(1)}-\www b_3^{(1)}\big)\tau_3,
\\
0= z^2\paa_z\tau_2 + z\big({-}b_3^{(1)}-\www b_3^{(1)}\big)\tau_2+2b_2^{(0)}\tau_3,
\\
0= z^2\paa_z\tau_3 + z\big(b_3^{(1)}-\www b_3^{(1)}\big)\tau_1+ zb_4^{(1)}\tau_2-b_2^{(0)}\tau_4,
\\
0= z\paa_z\tau_4 -2b_4^{(1)}\tau_3+ \big(b_3^{(1)}+\www b_3^{(1)}\big)\tau_4.
\end{gather*}
This is equivalent to
\begin{gather*}
0=\tau_3^{(0)}=\tau_4^{(0)},
\\
0= k\tau_1^{(k)} + \big(b_3^{(1)}-\www b_3^{(1)}\big)\tau_3^{(k)}\qquad
\text{for}\quad k\geq 1,
\\
0= \big(k-1-b_3^{(1)}-\www b_3^{(1)}\big)\tau_2^{(k-1)}+2b_2^{(0)}\tau_3^{(k)}\qquad
\text{for}\quad k\geq 1,
\\
0= (k-1)\tau_3^{(k-1)}+ \big(b_3^{(1)}-\www b_3^{(1)}\big)\tau_1^{(k-1)}
+ b_4^{(1)}\tau_2^{(k-1)}-b_2^{(0)}\tau_4^{(k)}\qquad
\text{for}\quad k\geq 1,
\\
0= -2b_4^{(1)}\tau_3^{(k)}+ \big(k+b_3^{(1)}+\www b_3^{(1)}\big)\tau_4^{(k)}\qquad
\text{for}\quad k\geq 1.
\end{gather*}
This is equivalent to
\begin{gather}
\tau_3^{(0)}=\tau_4^{(0)}=0,\nonumber
\\
\tau_1^{(k)}= \frac{-1}{k}\big(b_3^{(1)}-\www b_3^{(1)}\big)\tau_3^{(k)}\qquad
\text{for}\quad k\geq 1,\nonumber
\\
2b_2^{(0)}\tau_3^{(k)}= \big(b_3^{(1)}+\www b_3^{(1)}+1-k\big)\tau_2^{(k-1)}\qquad
\text{for}\quad k\geq 1,\nonumber
\\
b_2^{(0)}\tau_4^{(1)}=b_4^{(1)}\tau_2^{(0)}+\big(b_3^{(1)}-\www b_3^{(1)}\big)\tau_1^{(0)},\nonumber
\\
b_2^{(0)}\tau_4^{(k)}=b_4^{(1)}\tau_2^{(k-1)}
+\bigg(k-1+\frac{-1}{k-1}\big(b_3^{(1)}-\www b_3^{(1)}\big)^2\bigg)
\big(2b_2^{(0)}\big)^{-1}\big(b_3^{(1)}+\www b_3^{(1)}+2-k\big)\tau_2^{(k-2)}\nonumber
\\
\qquad\text{for}\quad k\geq 2,\nonumber
\\
0= b_4^{(1)}\tau_2^{(0)}+\big(1+b_3^{(1)}+\www b_3^{(1)}\big)
\big(b_3^{(1)}-\www b_3^{(1)}\big)\tau_1^{(0)},\nonumber
\\
0= b_4^{(1)}(2k+1)\tau_2^{(k)}+
\big(k+1+b_3^{(1)}+\www b_3^{(1)}\big)\bigg(k+\frac{-1}{k}\big(b_3^{(1)}-\www b_3^{(1)}\big)^2\bigg)\big(2b_2^{(0)}\big)^{-1}
\nonumber
\\ \hphantom{0=}
{}\times\big(b_3^{(1)}+\www b_3^{(1)}+1-k\big)\tau_2^{(k-1)}
\qquad\text{for}\quad k\geq 1.
\label{4.41}
\end{gather}
One can choose $\tau_1^{(0)}\in\C^*$ freely.
Then the equations~\eqref{4.41} have unique
solutions $\tau_1-\tau_1^{(0)},\tau_2$, $\tau_3,\tau_4\in\C[[z]]$.
Therefore $T\in {\rm GL}_2(\C[[z]])$ exists such that $T$,
$B$ as in~\eqref{4.25} and $\www B$ as above satisfy~\eqref{4.19}. This shows part $(b)$.
\end{proof}

\begin{Remarks}\qquad
\begin{enumerate}\itemsep=0pt
\item[$(i)$] Because of Lemma~\ref{t4.14}, the set $\Eig(M^{\rm mon})$
takes here the role of the Stokes structure: It~distinguishes the
holomorphic isomorphism classes within one formal
isomorphism class.

\item[$(ii)$] The following $(TE)$-structures
have trivial Stokes structure.
The proof of Lemma~\ref{t4.9} leads to these $(TE)$ structures. They are the $(TE)$-structures
with $\Eig(M^{\rm mon})=\{\lambda_1,\lambda_2\}$
with $\lambda_2=-\lambda_1$,
respectively with $b_3^{(1)}\in\big({\pm}\,\frac{1}{4}+\Z\big)$
for any Birkhoff normal form in~\eqref{4.25}.

Choose a number $\alpha^{(1)}\in\C$.
Consider the rank $2$ bundle $H'\to\C^*$ with flat
connection $\nnn$ and flat multivalued basis
$\uuuu{f}=(f_1,f_2)$ with monodromy given by
\begin{gather*}
\uuuu{f}(z{\rm e}^{2\pi {\rm i}})=\uuuu{f}(z)\cdot {\rm i}{\rm e}^{-2\pi {\rm i}\alpha^{(1)}}
\cdot (C_2+E).
\end{gather*}
The eigenvalues are $\pm {\rm i}{\rm e}^{-2\pi {\rm i}\alpha^{(1)}}$.
Choose numbers $t_1\in\C$ and $t_2\in\C^*$.
The following basis of $H'$ is univalued.
\begin{eqnarray}\label{4.43}
\uuuu{v}:= \uuuu{f}\cdot {\rm e}^{t_1/z}z^{\alpha^{(1)}}
\begin{pmatrix}z^{-1/4}{\rm e}^{t_2z^{-1/2}}
& z^{1/4}{\rm e}^{t_2z^{-1/2}}\\
z^{-1/4}{\rm e}^{-t_2z^{-1/2}}
&-z^{1/4}{\rm e}^{-t_2z^{-1/2}}
\end{pmatrix}\!.
\end{eqnarray}
The matrix $B$ with
$z^2\nnn_{\paa_z}\uuuu{v}=\uuuu{v}\cdot B$ is
\begin{gather}\label{4.44}
B= \big({-}t_1+z\alpha^{(1)}\big)C_1-\frac{t_2}{2} C_2
-z\frac{1}{4}D-z\frac{t_2}{2}E.
\end{gather}
So here $\rho^{(1)}=\alpha^{(1)}$, $\rho^{(0)}=-t_1$,
$\delta^{(1)}-2\rho^{(0)}\rho^{(1)}=-\frac{1}{4}t_2^2$.

\item[$(iii)$] Part $(ii)$ generalizes to a $(TE)$-structure over
$M=\C^2$ with coordinates $t=(t_1,t_2)$.
Consider the rank $2$ bundle $H'\to\C^*\times M$ with
flat connection and flat multivalued basis $\uuuu{f}=(f_1,f_2)$
with monodromy given by
\begin{gather*}
\uuuu{f}\big(z{\rm e}^{2\pi {\rm i}},t\big)=\uuuu{f}(z,t)\cdot {\rm i}{\rm e}^{-2\pi {\rm i}\alpha^{(1)}}
\cdot (C_2+E).
\end{gather*}
The basis $\uuuu{v}$ in~\eqref{4.43} is univalued.
The matrices $A_1$, $A_2$ and $B$ in its connection 1-form
$\Omega$ as in~\eqref{3.4}--\eqref{3.6} are given
by~\eqref{4.44} and
\begin{gather*}
A_1=C_1,\qquad
A_2=C_2+zE.
\end{gather*}
The restriction to a point $t\in \C\times\C^*$ is a
$(TE)$-structure of type (Bra) with trivial Stokes
structure.
The restriction to a point $t\in\C\times\{0\}$
is a $(TE)$-structure of type (Log).
\end{enumerate}
\end{Remarks}

\subsection[The case (Reg) with $\tr\UU=0$]
{The case (Reg) with $\boldsymbol{\tr\UU=0}$}\label{c4.5}

The $(TE)$-structures over a point of the type (Reg) with
$\tr\UU=0$ are the regular singular $(TE)$-structures
over a point which are not logarithmic.
They can be easily classified using elementary sections.
Theorem~\ref{t4.17} splits them into three cases
(one in part $(a)$, two in part $(b)$: $\alpha_1=\alpha_2$
and $\alpha_1-\alpha_2\in\N$).

\begin{Notation}\label{t4.16}
Start with a $(TE)$-structure $(H\to\C,\nnn)$ of rank $2$
over a point. Recall the notions from Definition~\ref{t3.21}:
$H':=H|_{\C^*}$, $M^{\rm mon}$, $M^{\rm mon}_s$, $M^{\rm mon}_u$, $N^{\rm mon}$,
$\Eig(M^{\rm mon})$, $H^\infty$, $H^\infty_{\lambda}$, $C^\alpha$
for $\alpha\in\C$ with ${\rm e}^{-2\pi {\rm i}\alpha}\in \Eig(M^{\rm mon})$,
$s(A,\alpha)\in C^\alpha$ for
$A\in H^{\infty}_{{\rm e}^{-2\pi {\rm i}\alpha}}$,
$\operatorname{es}(\sigma,\alpha)\in C^\alpha$ for~$\sigma$
a~holomorphic section on $H|_{U_1\setminus\{0\}}$ for
$U_1\subset\C$ a neighborhood of 0.
Now the eigenvalues of $M^{\rm mon}$ are called $\lambda_1$
and $\lambda_2$ ($\lambda_1=\lambda_2$ is allowed).
The sheaf $\VV^{>-\infty}$ simplifies here to a~$\C\{z\}\big[z^{-1}\big]$-vector space of dimension $2$,
\begin{gather*}
V^{>-\infty}:= \begin{cases}
\C\{z\}\big[z^{-1}\big]\cdot C^{\alpha_1}
\oplus \C\{z\}\big[z^{-1}\big]\cdot C^{\alpha_2}&
\text{if}\quad\lambda_1\neq\lambda_2,
\\[1ex]
\C\{z\}\big[z^{-1}\big]\cdot C^{\alpha_1}&
\text{if}\quad\lambda_1=\lambda_2,\end{cases}
\end{gather*}
where $\alpha_1,\alpha_2\in\C$ with
${\rm e}^{-2\pi {\rm i}\alpha_j}=\lambda_j$.
$V^{>-\infty}$ is the space of sections of moderate growth.
\end{Notation}

\begin{Theorem}\label{t4.17}
Consider a regular singular, but not logarithmic,
rank $2$ $(TE)$-structure $(H\to\C,\nnn)$ over a point.
Associate to it the data in the Notation~$\ref{t4.16}$.
\begin{enumerate}\itemsep=0pt
\item[$(a)$] The case $N^{\rm mon}=0$:
There exist unique numbers $\alpha_1$, $\alpha_2$
with ${\rm e}^{-2\pi {\rm i}\alpha_j}=\lambda_j$
and $\alpha_1\neq\alpha_2$ and the following properties:
There exist elementary sections $s_1\in C^{\alpha_1}\setminus\{0\}$
and $s_2\in C^{\alpha_2}\setminus\{0\}$ and a number $t_2\in\C^*$
such that
\begin{align}
\label{4.48}
\OO(H)_0&=\C\{z\}(s_1+t_2s_2)\oplus \C\{z\}(zs_2)
\\
&= \C\{z\}\big(s_2+t_2^{-1}s_1\big)\oplus \C\{z\}(zs_1).
\label{4.49}
\end{align}
The isomorphism class of the $(TE)$-structure is
uniquely determined by the information $N^{\rm mon}=0$
and the set $\{\alpha_1,\alpha_2\}$.
The numbers $\alpha_1$ and $\alpha_2$ are called
leading exponents.

\item[$(b)$] The case $N^{\rm mon}\neq 0$ $($thus $\lambda_1=\lambda_2)$:
There exist unique numbers $\alpha_1$, $\alpha_2$
with ${\rm e}^{-2\pi {\rm i}\alpha_j}=\lambda_1$ and
$\alpha_1-\alpha_2\in\N_0$
and the following properties:
Choose any elementary section
$s_1\in C^{\alpha_1}\setminus \ker(z\nnn_{\paa_z}-\alpha_1\colon
C^{\alpha_1}\to C^{\alpha_1})$.
The elementary section $s_2\in C^{\alpha_2}$ with
\begin{gather}\label{4.50}
(z\nnn_{\paa_z}-\alpha_1)(s_1)=z^{\alpha_1-\alpha_2}s_2.
\end{gather}
is a generator of
$\ker(z\nnn_{\paa_z}-\alpha_2\colon C^{\alpha_2}\to C^{\alpha_2})$.
Then
\begin{gather}\label{4.51}
\OO(H)_0=\C\{z\}(s_1+t_2s_2)\oplus \C\{z\}(zs_2)
\end{gather}
for some $t_2\in\C$. If $\alpha_1>\alpha_2$ then
$t_2$ is in $\C^*$ and is independent of the choice of
$s_1$. If~\mbox{$\alpha_1=\alpha_2$}, then one can replace
$s_1$ by $s_1^{\rm[new]}:=s_1+t_2s_2$, and then
$t_2^{\rm[new]}=0$.
The~iso\-morphism class of the $(TE)$-structure is
uniquely determined by the information $N^{\rm mon}\neq 0$
and the pair $(\alpha_1,\alpha_2)$ if $\alpha_1=\alpha_2$
and the triple $(\alpha_1,\alpha_2,t_2)$ if $\alpha_1>\alpha_2$.
The numbers $\alpha_1$ and $\alpha_2$ are called
leading exponents.
\end{enumerate}
\end{Theorem}

\begin{proof} First, $(a)$ and $(b)$ are considered together.
Let $\beta_1,\beta_2\in\C$ be the unique numbers
with ${\rm e}^{-2\pi {\rm i}\beta_j}=\lambda_j$ and
$-1<\Ree(\beta_j)\leq 0$. Choose elementary sections
$\www s_1\in C^{\beta_1}$ and
$\www s_2\in C^{\beta_2}$ which
form a global basis of $H'$. In~the case $N^{\rm mon}\neq 0$ (then $\beta_1=\beta_2$)
choose them such that
$\www s_1\notin \ker\big(z\nnn_{\paa_z}-\beta_1\colon C^{\beta_1}\to
C^{\beta_1}\big)$ and~$\www s_2\in\ker\big(z\nnn_{\paa_z}-\beta_2\colon
C^{\beta_2}\to C^{\beta_2}\big)$.

Let $\sigma_1^{[1]},\sigma_2^{[1]}\in\OO(H)_0$
be a $\C\{z\}$-basis of $\OO(H)_0$. Write
\begin{gather*}
\big(\sigma_1^{[1]},\sigma_2^{[1]}\big)=\big(\www s_1,\www s_2\big)
\begin{pmatrix}
b_{11}& b_{12}\\b_{21}&b_{22}\end{pmatrix}
\qquad\text{with}\quad
b_{ij}\in\C\{z\}\big[z^{-1}\big].
\end{gather*}
Recall that the degree $\deg_z g$ of a Laurent series
$g=\sum_{j\in\Z}g^{(j)}z^j\in\C\{z\}\big[z^{-1}\big]$ is the
minimal $j$ with $g^{(j)}\neq 0$ if $g\neq 0$,
and $\deg_z0:=+\infty$.

In the case $N^{\rm mon}=0$ and $\lambda_1=\lambda_2$
(then $\beta_1=\beta_2$),
we suppose $\min(\deg_z b_{11},\deg_z b_{12})
\leq \min(\deg_z b_{21},\deg_z b_{22})$.
If it does not hold a priori, we can exchange
$\www s_1$ and $\www s_2$.

In any case, we suppose $\deg_z b_{11}\leq \deg_z b_{12}$.
If it does not hold a priori, we can exchange $\sigma_1^{[1]}$
and $\sigma_2^{[1]}$.

Again in the case $N^{\rm mon}=0$ and $\lambda_1=\lambda_2$,
we suppose $\deg_zb_{11}<\deg_zb_{21}$. If it does not hold
a priori, we can replace $\www s_2$ by a certain linear
combination of $\www s_2$ and $\www s_1$.

Now $\www b_{11}:=z^{-\deg_z b_{11}}b_{11}\in\C\{z\}^*$ is
a unit. Consider $\alpha_1:=\beta_1+\deg_z b_{11}$
and $s_1:=z^{\deg_z b_{11}}\www s_1\allowbreak\in C^{\alpha_1}$
and the new basis $\big(\sigma_1^{[2]},\sigma_2^{[2]}\big)$
of $\OO(H)_0$ with
\begin{gather*}
\big(\sigma_1^{[2]},\sigma_2^{[2]}\big)
:=\big(\sigma_1^{[1]},\sigma_2^{[1]}\big)
\begin{pmatrix} \www b_{11}^{-1} &-b_{11}^{-1}b_{12} \\
0 & 1 \end{pmatrix}
= \big(s_1,\www s_2\big)
\begin{pmatrix} 1 & 0 \\ \www b_{11}^{-1} b_{21} &
b_{22}-b_{11}^{-1}b_{12}b_{21}\end{pmatrix}\!.
\end{gather*}

Consider $m:=\deg_z\big(b_{22}-b_{11}^{-1}b_{12}b_{21}\big)\in\Z$
($+\infty$ is impossible) and
$\alpha_2:=\beta_2+(m-1)$ and
$s_2:=z^{m-1}\www s_2\in C^{\alpha_2}$.
Write $z^{-m+1}\www b_{11}^{-1}b_{21}=c_1+c_2$
with $c_1\in\C\big[z^{-1}\big]$ and $c_2\in z\C\{z\}$.
We can replace $\sigma_2^{[2]}$ by
$\sigma_2^{[3]}:=zs_2$
and $\sigma_1^{[2]}=s_1+(c_1+c_2)s_2$ by
$\sigma_1^{[3]}=s_1+c_1s_2$.

$(a)$ Consider the case $N^{\rm mon}=0$.
If $\lambda_1=\lambda_2$ then $\deg_z b_{21}\geq \deg_z b_{11}+1$
and thus
\begin{align}
(c_1+c_2)s_2&=\www{b}_{11}^{-1}b_{21}\www{s}_2\in
\C\{z\}\cdot z^{\deg_zb_{21}}\cdot C^{\beta_2}\nonumber
\\
&{}\subset \C\{z\}\cdot z^{\deg_zb_{11}+1}\cdot C^{\beta_2}
= \C\{z\}\cdot C^{\alpha_1+1}.\label{4.54}
\end{align}
In any case (whether $\lambda_1=\lambda_2$ or
$\lambda_1\neq \lambda_2$), we must have $c_1\neq 0$.
Else the $(TE)$-structure is logarithmic.
As the pole has precisely order 2, $c_1$ is a constant
$\neq 0$ (if $\lambda_1=\lambda_2$, here we need~\eqref{4.54}),
which is now called $t_2$. This implies $m-1=\deg_z b_{21}$. In~the case $N^{\rm mon}=0$ and $\lambda_1=\lambda_2$
we have $\beta_1=\beta_2$ and
\begin{gather*}
\alpha_2-\alpha_1=m-1-\deg_z b_{11}
=\deg_z b_{21}-\deg_z b_{11}> 0,
\end{gather*}
so especially $\alpha_2\neq \alpha_1$.

$(b)$ Consider the case $N^{\rm mon}\neq 0$.
Then $s_2$ is generator of
$\ker(z\nnn_{\paa_z}-\alpha_2\colon C^{\alpha_2}\to C^{\alpha_2})$,
and we can rescale it such that~\eqref{4.50} holds.
First consider the case $c_1=0$.
As the pole has precisely order 2, we must have
$\alpha_2=\alpha_1$. Then~\eqref{4.51} holds with $t_2=0$.
Now consider the case $c_1\neq 0$. Then
$\uuuu{\sigma}^{[3]}=\big(\sigma_1^{[3]},\sigma_2^{[3]}\big)$ satisfies
\begin{gather*}
z\nnn_{\paa_z}\uuuu{\sigma}^{[3]}=\uuuu{\sigma}^{[3]}
\begin{pmatrix}\alpha_1 & 0 \\
z^{-1}(z\paa_z-\alpha_1+\alpha_2)(c_1)
+z^{\alpha_1-\alpha_2-1}& \alpha_2+1
\end{pmatrix}\!.
\end{gather*}
First case, $\alpha_1-\alpha_2\in\Z_{<0}$:
The coefficient of $z^{\alpha_1-\alpha_2-1}$
in $z^{-1}(z\paa_z-\alpha_1+\alpha_2)(c_1)
+z^{\alpha_1-\alpha_2-1}$ is~$1$.
Therefore the pole order is $>2$, a contradiction.

Second case, $\alpha_1\geq \alpha_2$:
As the pole has precisely order 2, $c_1$ is a constant $\neq 0$,
which is now called $t_2$. Then~\eqref{4.51} holds,
and $t_2\in\C^*$. In~the case $\alpha_1-\alpha_2\in\N$,
$t_2$ is obviously independent of the choice of $s_1$.
\end{proof}

Corollary~\ref{t4.18} is an immediate consequence of
Theorem~\ref{t4.17}.

\begin{Corollary}\label{t4.18}
The set of regular singular, but not logarithmic,
rank~$2$ $(TE)$-structures over a point is in bijection
with the set
\begin{align*}
\{(0,\{\alpha_1,\alpha_2\})\,|\, \alpha_1,\alpha_2\in\C,\alpha_1\neq\alpha_2\}
&\cup\{(1,\alpha_1,\alpha_2)\,|\, \alpha_1=\alpha_2\in\C\}
\\
&\cup\{(1,\alpha_1,\alpha_2,t_2)\,|\, \alpha_1,\alpha_2\in\C,
\alpha_1-\alpha_2\in \N,t_2\in\C^*\}.
\end{align*}
The first set parametrizes the cases with $N^{\rm mon}=0$,
the second and third set parametrize the cases with
$N^{\rm mon}\neq 0$. Theorem~$\ref{t4.17}$ describes the
corresponding $(TE)$-structures.
\end{Corollary}

\begin{Remark}\label{t4.19}
The connection matrices for the special bases in
Theorem~\ref{t4.17} can be written down easily.

The basis in~\eqref{4.48}:
\begin{gather}\label{4.57}
\nnn_{z\paa_z}(s_1+t_2s_2,zs_2)=
(s_1+t_2s_2,zs_2)\begin{pmatrix}\alpha_1 & 0 \\
z^{-1}(\alpha_2-\alpha_1)t_2 & \alpha_2+1\end{pmatrix}\!.
\end{gather}

The basis in~\eqref{4.49} with $\www t_2:=t_2^{-1}$:
\begin{gather}\label{4.58}
\nnn_{z\paa_z}\big(s_2+\www t_2s_1,zs_1\big)=
\big(s_2+\www t_2s_1,zs_1\big)\begin{pmatrix}\alpha_2 & 0 \\
z^{-1}(\alpha_1-\alpha_2)\www t_2 & \alpha_1+1\end{pmatrix}\!.
\end{gather}

The basis in~\eqref{4.51} with~\eqref{4.50}:
\begin{gather}\label{4.59}
\nnn_{z\paa_z}(s_1+t_2s_2,zs_2)=
(s_1+t_2s_2,zs_2)
\begin{pmatrix}\alpha_1 & 0 \\
z^{-1}(\alpha_2-\alpha_1)t_2 +z^{\alpha_1-\alpha_2-1}
& \alpha_2+1\end{pmatrix}\!.
\end{gather}

Finally, in the case $N^{\rm mon}\neq 0$ and $t_2\in\C^*$,
we consider with $\www t_2:=t_2^{-1}$ also the basis
$\big(s_2+\www t_2 s_1,zs_1\big)$. Again~\eqref{4.50} is assumed:
\begin{gather}
\nnn_{z\paa_z}\big(s_2+\www t_2s_1,zs_1\big)=\big(s_2+\www t_2s_1,zs_1\big)\nonumber
\\ \hphantom{\nnn_{z\paa_z}(s_2+\www t_2s_1,zs_1)=}
{}\times\begin{pmatrix}\alpha_2
+z^{\alpha_1-\alpha_2}\www t_2 &
z^{\alpha_1-\alpha_2+1} \\[.5ex]
z^{-1}(\alpha_1-\alpha_2)\www t_2 -z^{\alpha_1-\alpha_2-1}
\www t_2^2 & \alpha_1+1-z^{\alpha_1-\alpha_2}\www t_2
\end{pmatrix}\!.
\end{gather}\label{4.60}
\end{Remark}

\subsection[The case (Log) with $\tr\UU=0$]{The case (Log) with $\boldsymbol{\tr\UU=0}$}\label{c4.6}

The $(TE)$-structures over a point of the type (Log) with
$\tr\UU=0$ are the logarithmic $(TE)$-structures
over a point. Just as the regular singular $(TE)$-structures,
they can easily be classified using
elementary sections. Theorem~\ref{t4.20}
splits them into two cases.
We use again the Notation~\ref{t4.16}.

\begin{Theorem}\label{t4.20}
Consider a logarithmic rank $2$ $(TE)$-structure $(H\to\C,\nnn)$
over a point. Asso\-ci\-ate to it the data in the Notation~$\ref{t4.16}$.
\begin{enumerate}\itemsep=0pt
\item[$(a)$] The case $N^{\rm mon}=0$: There exist unique numbers
$\alpha_1$, $\alpha_2$ with ${\rm e}^{-2\pi {\rm i}\alpha_j}=\lambda_j$
and the following property:
There exist elementary sections $s_1\in C^{\alpha_1}\setminus\{0\}$
and $s_2\in C^{\alpha_2}\setminus\{0\}$ such that
\begin{gather}\label{4.61}
\OO(H)_0 = \C\{z\}\, s_1\oplus \C\{z\}\, s_2.
\end{gather}
The isomorphism class of the $(TE)$-structure is uniquely
determined by the information $N^{\rm mon}=0$ and
the set $\{\alpha_1,\alpha_2\}$.
The numbers $\alpha_1$ and $\alpha_2$ are called
leading exponents.

\item[$(b)$] The case $N^{\rm mon}\neq 0$ $($thus $\lambda_1=\lambda_2)$:
There exist unique numbers $\alpha_1$, $\alpha_2$
with ${\rm e}^{-2\pi {\rm i}\alpha_j}=\lambda_1$ and
$\alpha_1-\alpha_2\in\N_0$ and
the following properties:
Choose any elementary section
$s_1\in C^{\alpha_1}-\ker(\nnn_{z\paa_z}-\alpha_1\colon
C^{\alpha_1}\to C^{\alpha_1})$.
The elementary section $s_2\in C^{\alpha_2}$ with
\begin{gather}\label{4.62}
(z\nnn_{\paa_z}-\alpha_1)(s_1)=z^{\alpha_1-\alpha_2}s_2.
\end{gather}
is a generator of
$\ker(z\nnn_{\paa_z}-\alpha_2\colon C^{\alpha_2}\to C^{\alpha_2})$.
Then
\begin{gather}\label{4.63}
\OO(H)_0=\C\{z\}\, s_1\oplus \C\{z\}\, s_2.
\end{gather}
The isomorphism class of the $(TE)$-structure is
uniquely determined by the information $N^{\rm mon}\neq 0$
and the set $\{\alpha_1,\alpha_2\}$.
The numbers $\alpha_1$ and $\alpha_2$ are called
leading exponents.
\end{enumerate}
\end{Theorem}

\begin{proof}
First, $(a)$ and $(b)$ are considered together.
By Theorem~\ref{t3.23}$(a)$, $\OO(H)_0$ is
generated by two elementary sections
$s_1\in C^{\alpha_1}$ and $s_2\in C^{\alpha_2}$
for some numbers $\alpha_1$ and $\alpha_2$.
The numbers~$\alpha_1$ and~$\alpha_2$ are the
eigenvalues of the residue endomorphism.
So, they are unique.
This finishes already the proof of part $(a)$.

$(b)$ Consider the case $N^{\rm mon}\neq 0$.
We can renumber $s_1$ and $s_2$ if necessary, so that
afterwards $\alpha_1-\alpha_2\in\N_0$.
If $\alpha_1=\alpha_2$, then $\OO(H)_0=\C\{z\}C^{\alpha_1}$,
and $s_1$ and $s_2$ can be changed so that
$s_1\in C^{\alpha_1}\setminus \ker(\nnn_{z\paa_z}-\alpha_1)$
and $s_2\in \ker(\nnn_{z\paa_z}-\alpha_1\colon
C^{\alpha_1}\to C^{\alpha_1})\setminus\{0\}$ satisfy~\eqref{4.62}.
Then nothing more has to be shown.
Consider the case $\alpha_1-\alpha_2\in\N$.
If $s_2\in C^{\alpha_2}\setminus \ker(\nnn_{z\paa_z}-\alpha_2)$,
then $(\nnn_{z\paa_z}-\alpha_2)(s_2)$ is not in
$\OO(H)_0$, and thus the pole is not logarithmic,
a contradiction. Therefore
$s_2\in \ker(\nnn_{z\paa_z}-\alpha_2\colon
C^{\alpha_2}\to C^{\alpha_2})$. Then necessarily
$s_1\in C^{\alpha_1}\setminus \ker(\nnn_{z\paa_z}-\alpha_1\colon
C^{\alpha_1} \to C^{\alpha_1})$.
We~can rescale $s_2$ so that~\eqref{4.62} holds.
Nothing more has to be shown.
\end{proof}

Corollary~\ref{t4.21} is an immediate consequence of
Theorem~\ref{t4.20}.

\begin{Corollary}\label{t4.21}
The set of logarithmic rank $2$ $(TE)$-structures over a point
is in bijection with the set
\begin{gather*}
\{(0,\{\alpha_1,\alpha_2\})\,|\, \alpha_1,\alpha_2\in\C,\}
\cup\{(1,\alpha_1,\alpha_2)\,|\, \alpha_1,\alpha_2\in\C,\alpha_1-\alpha_2\in \N_0\}.
\end{gather*}
The first set parametrizes the cases with $N^{\rm mon}=0$,
the second set parametrizes the cases with $N^{\rm mon}\neq 0$.
Theorem~$\ref{t4.20}$ describes the corresponding
$(TE)$-structures.
\end{Corollary}

\begin{Remark}\label{t4.22}
The connection matrices for the special bases in
Theorem~\ref{t4.20} can be written down easily.

The basis in~\eqref{4.61}:
\begin{gather*}
\nnn_{z\paa_z}(s_1,s_2)=(s_1,s_2)
\begin{pmatrix}\alpha_1 & 0 \\ 0 & \alpha_2\end{pmatrix}\!.
\end{gather*}

The basis in~\eqref{4.63}:
\begin{gather*}
\nnn_{z\paa_z}(s_1,s_2)=(s_1,s_2)
\begin{pmatrix}\alpha_1 & 0 \\ z^{\alpha_1-\alpha_2} & \alpha_2\end{pmatrix}\!.
\end{gather*}

The basis $(s_1,s_2)$ gives a Birkhoff normal form
in the cases $N^{\rm mon}=0$ and in the cases
$(N^{\rm mon}\neq 0$ and $\alpha_1=\alpha_2)$. In~the cases $(N^{\rm mon}\neq 0$ and $\alpha_1-\alpha_2\in\N)$,
a Birkhoff normal form does not exist.
\end{Remark}

\section[Rank 2 $(TE)$-structures over germs of regular $F$-manifolds]{Rank 2 $\boldsymbol{(TE)}$-structures over germs of regular $\boldsymbol{F}$-manifolds}\label{c5}

This section discusses unfoldings of
$(TE)$-structures over a point $t^0$
of type (Sem) or (Bra) or~(Reg).
Here Malgrange's unfolding result
Theorem~\ref{t3.16}$(c)$ applies. It~provides a universal unfolding for the $(TE)$-structure
over $t^0$. Any unfolding is induced by the universal
unfolding.
The universal unfoldings turn out to be precisely
the $(TE)$-structures with primitive Higgs fields
over germs of regular $F$-manifolds.

Sections~\ref{c6} and~\ref{c8} discuss
unfoldings of $(TE)$-structures over a point of type (Log).
Section~\ref{c8} treats arbitrary such unfoldings.
Section~\ref{c6} prepares this. It~treats 1-parameter unfoldings with trace free pole parts
of logarithmic $(TE)$-structures over a point.

If one starts with a $(TE)$-structure with primitive Higgs field
over a germ $\big(M,t^0\big)$ of a regular $F$-manifold, then
the endomorphism $\UU|_{t^0}\colon K_{t^0}\to K_{t^0}$ is regular.

Vice versa, if one starts with a $(TE)$-structure
over a point $t^0$ with a regular endomorphism $\UU\colon
K_{t^0}\to K_{t^0}$,
then it unfolds uniquely to a $(TE)$-structure
with primitive Higgs field over a~germ of a regular $F$-manifold
by Malgrange's result Theorem~\ref{t3.16}$(c)$.
The germ of the regular $F$-manifold is uniquely determined by
the isomorphism class of $\UU\colon K_{t^0}\to K_{t^0}$
(i.e., its Jordan block structure).
And the $(TE)$-structure is uniquely determined by its
restriction to $t^0$.

The following statement on the rank $2$ cases
is an immediate consequence of Malgrange's unfolding result
Theorem~\ref{t3.16}$(c)$, the classification of germs
of regular 2-dimensional $F$-manifolds in Remark~\ref{t2.6}$(ii)$
(building on Theorems~\ref{t2.2} and~\ref{t2.3},
see also Remark~\ref{t3.17}$(iii)$) and the classification
of the rank $2$ $(TE)$-structures into the cases
(Sem), (Bra), (Reg) and (Log) in Definition~\ref{t4.4}.

\begin{Corollary}\label{t5.1} \qquad

\begin{enumerate}\itemsep=0pt
\item[$(a)$] For any rank $2$ $(TE)$-structure over a point $t^0$
except those of type $($Log$)$, the endomorphism
$\UU\colon K_{t^0}\to K_{t^0}$ is regular.
The $(TE)$-structure has a unique universal unfolding.
This unfolding has a primitive Higgs field.
Its base space is a germ $\big(M,t^0\big)=\big(\C^2,0\big)$
of an $F$-manifold with Euler field and is as follows:
\[
\def\arraystretch{1.5}
\begin{tabular}{c|c|c}
\hline
Type & $F$-manifold & Euler field
\\
\hline
$($Sem$)$ & $A_1^2$ & $\sum_{i=1}^2(u_i+c_i)e_i$ with $c_1\neq c_2$
\\
$($Bra$)$ or $($Reg$)$ & $\NN_2$ & $t_1\paa_1+g(t_2)\paa_2$ with $g(0)\neq 0$
\\
\hline
\end{tabular}
\]
In the case of $($Bra$)$ or $($Reg$)$, a coordinate change brings
$E$ to the form $t_1\paa_1+\paa_2$.

\item[$(b)$] Any unfolding of a rank $2$ $(TE)$-structure over $t^0$
with regular endomorphism $\UU\colon K_{t^0}\to K_{t^0}$
is induced by the universal unfolding in $(a)$.
\end{enumerate}
\end{Corollary}

Because of the existence and uniqueness of the universal
unfolding, it is not really necessary
to give it explicitly. On the other hand, in rank $2$,
it is easy to give it explicitly.
The following lemma offers one way.

\begin{Lemma}\label{t5.2}
Let $(H\to\C,\nnn)$ by a $(TE)$-structure over
a point with monodromy $M^{\rm mon}$ of some rank $r\in\N$. It~has an unfolding which is a $(TE)$-structure
$\big(H^{\rm (unf)}\to \C\times M,\nnn\big)$, where
$M=\C\times\C^*$ with coordinates $t=(t_1,t_2)$
$($on $\C^2\supset M)$, with the following properties.
\begin{enumerate}\itemsep=0pt
\item[$(a)$] The monodromy around $t_2=0$ is $(M^{\rm mon})^{-1}$.

\item[$(b)$] The original $(TE)$-structure is isomorphic to the
one over $t^0=(0,1)$.

\item[$(c)$] If $\uuuu{v}^0$ is a $\C\{z\}$-basis of $\OO(H)_0$
with $z^2\nnn_{\paa_z}\uuuu{v}^0=\uuuu{v}^0\, B^0$, then
$H^{\rm (unf)}$ has over $(\C,0)\times M$ a~basis
$\uuuu{v}$ such that the matrices $A_1$, $A_2$ and $B$
in~\eqref{3.4}--\eqref{3.6} are as follows
\begin{gather}
A_1= C_1,\label{5.2}
\\
A_2= -B^0\bigg(\frac{z}{t_2}\bigg),\label{5.3}
\\
B = -t_1C_1+t_2B^0\bigg(\frac{z}{t_2}\bigg)=-t_1A_1-t_2 A_2.\label{5.4}
\end{gather}

\item[$(d)$] If $\UU|_{t^0}$ is regular and $\rank H=2$, then
the Higgs field of the $(TE)$-structure $H^{\rm (unf)}$
is everywhere primitive. Therefore then $M$ is an $F$-manifold
with Euler field. The Euler field is $E=t_1\paa_1+t_2\paa_2$.

\item[$(e)$] If $\UU|_{t^0}$ is regular and $\rank H=2$,
the $(TE)$-structure over the germ $\big(M,t^0\big)$
is the universal unfolding of the one over $t^0$.
\end{enumerate}
\end{Lemma}

\begin{proof}
Let $\uuuu{f}^0=\big(f_1^0,\dots ,f_r^0\big)$ be a flat multivalued basis
of $H':=H|_{\C^*}$. Let $M^{\rm mat}\in {\rm GL}_r(\C)$
be the matrix of its monodromy, so
$\uuuu{f}^0\big(z {\rm e}^{2\pi {\rm i}}\big)=\uuuu{f}^0\cdot M^{\rm mat}$.
Let $\uuuu{v}^0=\big(v_1^0,\dots ,v_r^0\big)$ be a $\C\{z\}$-basis
of~$\OO(H)_0$. Let $B^0\in {\rm GL}_r(\C\{z\})$ be the matrix
with $z^2\nnn_{\paa_z}\uuuu{v}^0=\uuuu{v}^0\, B^0$.
Consider the mat\-rix~$\Psi(z,t)$ with multivalued entries with
\begin{gather*}
\uuuu{v}^0=\uuuu{f}^0\cdot \Psi(z).
\end{gather*}
Then
\begin{gather*}
\Psi\big(z{\rm e}^{2\pi {\rm i}}\big)= \big(M^{\rm mat}\big)^{-1}\cdot\Psi(z),
\\
\Psi^{-1}\paa_z \Psi = z^{-2}B^0(z).
\end{gather*}

Embed the flat bundle $H':=H|_{\C^*}$ as the bundle
over $t^0=(0,1)$ into a flat bundle
${H^{(mf)}}'\to\C^*\times M$ with monodromy $M^{\rm mon}$
around $z=0$ and monodromy $(M^{\rm mon})^{-1}$ around
$t_2=0$. The~flat multivalued basis $\uuuu{f}^0$ of $H'$ extends
to a flat multivalued basis $\uuuu{f}$ of ${H^{(mf)}}'$ with
\begin{gather*}
\uuuu{f}\big(z{\rm e}^{2\pi {\rm i}},t\big)=\uuuu{f}(z,t)M^{\rm mat},
\\
\uuuu{f}\big(z,t_1,t_2{\rm e}^{2\pi {\rm i}}\big)=\uuuu{f}(z,t)\big(M^{\rm mat}\big)^{-1}.
\label{5.9}
\end{gather*}
The tuple of sections $\uuuu{v}$ with
\begin{gather*}
\uuuu{v}= \uuuu{f}\cdot {\rm e}^{t_1/z}\Psi\bigg(\frac{z}{t_2}\bigg)
\end{gather*}
is univalued, it is a basis of ${H^{(mf)}}'$ in a neighbourhood
of $\{0\}\times M$, and it has the connection matrices
in~\eqref{5.2}--\eqref{5.4}: The calculations for $A_2$
and $B$ are
\begin{align*}
\nnn_{\paa_2}\uuuu{v}&= \uuuu{f}\, {\rm e}^{t_1/z}\,
\bigg({-}\frac{z}{t_2^2}\bigg)(\paa_z\Psi)\bigg(\frac{z}{t_2}\bigg) = \uuuu{f}\, {\rm e}^{t_1/z}\, \bigg({-}\frac{z}{t_2^2}\bigg)
\Psi\bigg(\frac{z}{t_2}\bigg)\bigg(\frac{z}{t_2}\bigg)^{-2}B^0\bigg(\frac{z}{t_2}\bigg)\\
&=\uuuu{v}\, \bigg({-}\frac{1}{z}\bigg)B^0\bigg(\frac{z}{t_2}\bigg),
\\
\nnn_{\paa_z}\uuuu{v}
&= \uuuu{f}\, {\rm e}^{t_1/z}\,
\bigg(\bigg({-}\frac{t_1}{z^2}\bigg)\Psi\bigg(\frac{z}{t_2}\bigg)+ \bigg(\frac1{t_2}\bigg)(\paa_z\Psi)\bigg(\frac{z}{t_2}\bigg)\bigg)\\
&= \uuuu{f}\, {\rm e}^{t_1/z}\, \bigg(\bigg({-}\frac{t_1}{z^2}\bigg)\Psi\bigg(\frac{z}{t_2}\bigg)+ \bigg(\frac1{t_2}\bigg)
\Psi\bigg(\frac{z}{t_2}\bigg)\bigg(\frac{z}{t_2}\bigg)^{-2}B^0\bigg(\frac{z}{t_2}\bigg)\bigg)
\\
&=\uuuu{v}\, \bigg(\bigg({-}\frac{t_1}{z^2}\bigg)C_1+\bigg(\frac{t_2}{z^2}\bigg) B^0\bigg(\frac{z}{t_2}\bigg)\bigg).
\end{align*}

Therefore $\uuuu{v}$ defines a $(TE)$-structure,
which we call $(H^{\rm (unf)}\to\C\times M,\nnn)$. It~unfolds the one over $t^0=(0,1)$,
and that one is isomorphic to $(H\to\C,\nnn)$.

It rests to show $(d)$ and $(e)$.
Suppose $\rank H=2$.
Then $\UU|_{t^0}$ is regular if and only if
$\big(B^0\big)^{(0)}\notin \C\cdot C_1$.
Then also $A_2^{(0)}(t)=-\big(B^0\big)^{(0)}\notin\C\cdot C_1$,
so then the Higgs field of the $(TE)$-structure $H^{\rm (unf)}$
is everywhere primitive.
Because of $B^{(0)}=-t_1A_1^{(0)}-t_2A_2^{(0)}$,
the Euler field is $E=t_1\paa_1+t_2\paa_2$.
$(e)$ follows from $(d)$ and Malgrange's result
Theorem~\ref{t3.16}$(c)$.
\end{proof}

\begin{Remarks}\label{t5.3}\qquad
\begin{enumerate}\itemsep=0pt
\item[$(i)$] In the cases (Reg) we will see the universal unfoldings
again in Section~\ref{c7}, in Remarks~\ref{t7.2}. In~a first step in Remarks~\ref{t7.1}, the value $t_2$
in the normal form in Remarks~\ref{t4.19} is
turned into a parameter in $\P^1$.
Remarks~\ref{t7.2} add another parameter~$t_1$ in~$\C$.
Then the Higgs field becomes primitive and
the base space $\C\times\P^1$
becomes a~2-dimensional $F$-manifold with Euler field.
For each $t^0\in\C\times\C^*$, the $(TE)$-structure over~$t^0$
is of type (Reg), and the $(TE)$-structure over the
germ $\big(M,t^0\big)$ is a universal unfolding of the one over~$t^0$.

\item[$(ii)$] In the cases (Bra), the following formulas give
a universal unfolding over $\big(\C^2,0\big)$ of any $(TE)$-structure
of type (Bra) over the point 0 (see Theorem~\ref{t4.11} for
their classification), such that the Euler field
is $E=(t_1+c_1)\paa_1+\paa_2$. Here $\rho^{(1)}\in\C$,
$b_3^{(0)}\in\C$, $b_2^{(0)},b_4^{(1)}\in\C^*$,
\begin{gather*}
A_1= C_1,
\\
A_2= -b_2^{(0)}C_2 -z\bigg(\frac{1}{2}+b_3^{(1)}\bigg)D
- z b_4^{(1)}{\rm e}^{t_2}E,
\\
B= (-t_1-c_1)C_1+b_2^{(0)}C_2 + z\big(\rho^{(1)}C_1+b_3^{(1)}D+
b_4^{(1)}{\rm e}^{t_2}E\big)
\\ \hphantom{B}
{}= (-t_1-c_1)A_1 - A_2 + z\rho^{(1)}C_1-z\frac{1}{2}D.
\end{gather*}

\item[$(iii)$] In the cases (Sem), a $(TE)$-structure over a point
extends uniquely to a $(TE)$-structure over the
universal covering $M$ of the manifold
$\big\{(u_1,u_2)\in\C^2\,|\, u_1\neq u_2\big\}$ (see~\cite{Ma83b} and~\cite[Chapter~III, Theorem~2.10]{Sa02}).
For each $t^0\in M$ the $(TE)$-structure over $t^0$
is of type (Sem), and the $(TE)$-structure over the germ
$\big(M,t^0\big)$ is the universal unfolding of the $(TE)$-structure
over~$t^0$.
\end{enumerate}
\end{Remarks}

\section[1-parameter unfoldings of logarithmic $(TE)$-structuresover a point]
{1-parameter unfoldings of logarithmic $\boldsymbol{(TE)}$-structures \\over a point}\label{c6}

This section classifies unfoldings over $\big(M,t^0\big)=(\C,0)$
with trace free pole part
of logarithmic $(TE)$-structures over the point~$t^0$.

It is a preparation for Section~\ref{c8},
which treats arbitrary unfoldings of $(TE)$-structures
of type (Log) over a point.

Section~\ref{c6.1}: An unfolding with trace free pole part
over $\big(M,t^0\big)=(\C,0)$ of a logarithmic rank $2$ $(TE)$-structure
over $t^0$ will be considered. Invariants of it will
be defined. Theorem~\ref{t6.2} gives constraints on these
invariants and shows that the monodromy is
semisimple if the generic type is (Sem) or (Bra).

By Theorem~\ref{t3.20}$(a)$ (which is trivial in our case
because of the logarithmic pole at $z=0$ of the $(TE)$-structure
over $t^0$) and Remark~\ref{t3.19}$(iii)$,
the $(TE)$-structure has a Birkhoff normal form,
i.e., an extension to a pure $(TLE)$-structure,
if its monodromy is semisimple.

Section~\ref{c6.2}: All pure $(TLE)$-structures over
$\big(M,t^0\big)=(\C,0)$ with trace free pole part and with
logarithmic restriction to $t^0$ are classified
in Theorem~\ref{t6.3}. These comprise all with
semisimple monodromy and thus all with generic
types (Sem) or (Bra).

Section~\ref{c6.3}: All $(TE)$-structures
over $\big(M,t^0\big)=(\C,0)$ with trace free pole part and
with logarithmic restriction over $t^0$ whose
monodromies have a $2\times 2$ Jordan block
are classified in Theorem~\ref{t6.7}.
Their generic types are (Reg) or (Log) because of
Theorem~\ref{t6.2}.
Most of them have no Birkhoff normal forms.
The intersection with Theorem~\ref{t6.3} is small
and consists of those which have Birkhoff normal forms.

Theorems~\ref{t6.3} and~\ref{t6.7} together give
all unfoldings with trace free pole parts
over $\big(M,t^0\big)=(\C,0)$ of logarithmic rank $2$
$(TE)$-structures over $t^0$.

\subsection{Numerical invariants for such (\emph{TE})-structures}
\label{c6.1}

The next definition gives some numerical invariants
for such $(TE)$-structures.
Recall the invariants $\delta^{(0)}$ and $\delta^{(1)}$
in Lemma~\ref{t3.9}.

\begin{Definition}\label{t6.1}
Let $\big(H\to\C\times \big(M,t^0\big),\nnn\big)$ be a $(TE)$-structure
with trace free pole part over $\big(M,t^0\big)=(\C,0)$
(with coordinate $t$) whose restriction over $t^0=0$ is
logarithmic. Let $M\subset \C$ be a neighborhood of 0
on which the $(TE)$-structure is defined.
On $M\setminus\{0\}$ it has a fixed type, (Sem) or (Bra) or
(Reg) or (Log), which is called the {\it generic type}
of the $(TE)$-structure. Lemma~\ref{t4.3} characterizes
the generic type in terms of (non)vanishing of
$\delta^{(0)},\delta^{(1)}\in t\C\{t\}$ and $\UU$:
\[
\def\arraystretch{1.5}
\begin{tabular}{c|c|c|c}
\hline
(Sem) &(Bra) &(Reg) &(Log)
\\ \hline
$\delta^{(0)}\neq 0$ & $\delta^{(0)}=0$,\ $\delta^{(1)}\neq 0$ &
$\delta^{(0)}=\delta^{(1)}=0$, $\UU\neq 0$ & $\UU=0$
\\
\hline
\end{tabular}
\]
For the generic types (Sem), (Bra) and (Reg), define
$k_1\in\N$ by
\begin{eqnarray}\label{6.1}
k_1 :=\max(k\in\N\,|\, \UU(\OO(H)_0)\subset t^k\OO(H)_0.
\end{eqnarray}
For the generic types (Sem) and (Bra) define $k_2\in\Z$ by
\begin{gather*}
k_2:= \begin{cases}
\deg_t\delta^{(0)}-k_1 & \text{for the generic type (Sem)},\\
\deg_t\delta^{(1)}-k_1 & \text{for the generic type (Bra)}.
\end{cases}
\end{gather*}
\end{Definition}

The following theorem gives for the generic type (Bra) and
part of the generic type (Sem) restrictions on the eigenvalues
of the residue endomorphism of the logarithmic pole at $z=0$
of the $(TE)$-structure over $t^0=0$.
And it shows that the monodromy is semisimple if the generic
type is (Sem) or (Bra).

\begin{Theorem}\label{t6.2}
Let $\big(H\to\C\times \big(M,t^0\big),\nnn\big)$ be a rank $2$ $(TE)$-structure
with trace free pole part over $\big(M,t^0\big)=(\C,0)$
whose restriction over $t^0=0$ is logarithmic.
Recall the invariant $\rho^{(1)}\in\C$ from Lemma~$\ref{t3.9}(b)$, and recall the invariants $k_1\in\N$ and $k_2\in\Z$
from Definition~$\ref{t6.1}$
if the generic type is $($Sem$)$ or $($Bra$)$.
\begin{enumerate}\itemsep=0pt
\item[$(a)$] Suppose that the generic type is $($Sem$)$.
\begin{enumerate}\itemsep=0pt
\item[$(i)$] Then $k_2\geq k_1$.

\item[$(ii)$] If $k_2>k_1$ then the eigenvalues of the
residue endomorphism of the logarithmic pole at~$z=0$
of the $(TE)$-structure over $t^0$ are
$\rho^{(1)}\pm \frac{k_1-k_2}{2(k_1+k_2)}$. Their difference
is smaller than~$1$. Especially, the eigenvalues of the
monodromy are different, and the mono\-dromy is semisimple.

\item[$(iii)$] Also if $k_1=k_2$, the monodromy is semisimple.
\end{enumerate}

\item[$(b)$] Suppose that the generic type is $($Bra$)$.
\begin{enumerate}\itemsep=0pt
\item[$(i)$] Then $k_2\in\N$.

\item[$(ii)$] The eigenvalues of the residue endomorphism of the
logarithmic pole at $z=0$ of the $(TE)$-structure over $t^0$
are $\rho^{(1)}\pm \frac{-k_2}{2(k_1+k_2)}$. Their difference
is smaller than $1$. Especially, the eigenvalues of the
monodromy are different, and the monodromy is semisimple.
\end{enumerate}
\end{enumerate}
\end{Theorem}

\begin{proof}
By Lemma~\ref{t3.11}, a $\C\{t,z\}$-basis $\uuuu{v}$
of the germ $\OO(H)_{(0,0)}$
can be chosen such that the matrices
$A$ and $B\in M_{2\times 2}(\C\{t,z\})$ with
$z\nnn_{\paa_t}\uuuu{v}=\uuuu{v} A$ and
$z^2\nnn_{\paa_z}\uuuu{v}=\uuuu{v} B$ satisfy~\eqref{3.19},
$0=\tr A=\tr\big(B-z\rho^{(1)} C_1\big)$, or, more explicitly,
\begin{gather*}
A= a_2C_2+a_3D+a_4E\qquad\text{with}\quad
a_2,a_3,a_4 \in \C\{t,z\},
\\
B= z\rho^{(1)}C_1+b_2C_2+b_3D+b_4E\qquad\text{with}\quad
b_2,b_3,b_4\in\C\{t,z\}.
\end{gather*}
Write $a_j=\sum_{k\geq 0}a_j^{(k)}z^k$ and
$a_j^{(k)}=\sum_{l\geq 0}a_{j,l}^{(k)}t^l\in\C\{t\}$,
and analogously for $b_j$.
Condition~\eqref{t3.8} says here
\begin{gather}
0= z\paa_tB-z^2\paa_z A + zA +[A,B]\nonumber
\\ \hphantom{0}
{}= C_2\bigg[ z\paa_t b_2 + za_2^{(0)} -
\sum_{k\geq 2}(k-1)a_2^{(k)}z^{k+1}+2a_2b_3-2a_3b_2\bigg]
\label{6.5}
\\ \hphantom{0=}
{}+ D\bigg[z\paa_tb_3 + za_3^{(0)} -
\sum_{k\geq 2}(k-1)a_3^{(k)}z^{k+1}-a_2b_4+a_4b_2\bigg]
\label{6.6}
\\ \hphantom{0=}
{}+E\bigg[ z\paa_t b_4 + za_4^{(0)} -
\sum_{k\geq 2}(k-1)a_4^{(k)}z^{k+1}-2a_4b_3+2a_3b_4\bigg].
\label{6.7}
\end{gather}

\medskip
$(a)$ Suppose that the generic type is (Sem).

\medskip
$(i)$ By definition of $k_1$ and $k_2$,
\begin{gather}\label{6.8}
k_1 = \min\big(\deg_tb_2^{(0)},\deg_tb_3^{(0)},\deg_tb_4^{(0)}\big),
\\
k_1+k_2= \deg_t\big(\big(b_3^{(0)}\big)^2+b_2^{(0)}b_4^{(0)}\big)\geq 2k_1,
\label{6.9}
\end{gather}
thus $k_2\geq k_1$.

\medskip
$(ii)$ Suppose $k_2>k_1$.
By a linear change of the basis $\uuuu{v}$, we can arrange
that $k_1=\deg_tb_2^{(0)}$. The base change matrix
$T=C_1+b_3^{(0)}/b_2^{(0)}\cdot E\in {\rm GL}_2(\C\{t\})$
gives the new basis $\uuuu{\www v}=\uuuu{v}\cdot T$
with matrix
\begin{gather*}
\www B^{(0)}=T^{-1}B^{(0)}T=b_2^{(0)}C_2+
\bigg(b_4^{(0)}+\frac{\big(b_3^{(0)}\big)^2}{b_2^{(0)}}\bigg)E.
\end{gather*}
We can make a coordinate change in $t$ such that afterwards
\begin{gather*}
b_2^{(0)}b_4^{(0)}+\big(b_3^{(0)}\big)^2=\gamma^2 t^{k_1+k_2}
\end{gather*}
for an arbitrarily chosen $\gamma\in\C^*$.
Then a diagonal base change leads to a basis
which is again called $\uuuu{v}$ with matrices
which are again called $A$ and $B$ with
\begin{gather*}
b_3^{(0)}=0,\qquad
b_2^{(0)}=\gamma t^{k_1},\qquad
b_4^{(0)}=\gamma t^{k_2}.
\end{gather*}
Now the vanishing of the coefficients in $\C\{t\}$ of
$C_2\cdot z^0$, $C_2\cdot z^1$, $D\cdot z^0$, $D\cdot z^1$
and $E\cdot z^1$ in~\eqref{6.5}--\eqref{6.7}
tells the following:
\begin{align*}
C_2\cdot z^0\colon \quad &a_3^{(0)}=0,
\\
C_2\cdot z^1\colon \quad &0=k_1\gamma t^{k_1-1}+a_2^{(0)}\big(1+2b_3^{(1)}\big)
-2a_3^{(1)}\gamma t^{k_1},
\\
&\text{so}\quad \deg_t a_2^{(0)}=k_1-1,\quad
0=k_1\gamma +a_{2,k_1-1}^{(0)}\big(1+2b_{3,0}^{(1)}\big),
\\
D\cdot z^0\colon\quad &a_2^{(0)}\gamma t^{k_2}=a_4^{(0)}\gamma t^{k_1},
\quad\text{so}\quad a_4^{(0)}=a_2^{(0)}t^{k_2-k_1},
\\
&\text{so}\quad \deg a_4^{(0)}=k_2-1,\quad\text{and}\quad
a_{4,k_2-1}^{(0)}=a_{2,k_1-1}^{(0)}.
\\
D\cdot z^1\colon\quad &a_2^{(0)}b_4^{(1)}+a_2^{(1)}\gamma t^{k_2}
=a_4^{(0)}b_2^{(1)}+a_4^{(1)}\gamma t^{k_1},
\\
&\text{so}\quad b_{4,0}^{(1)}=0 \quad \text{(here }
k_2>k_1\text{ is used)},
\\
E\cdot z^1\colon\quad & 0=k_2\gamma t^{k_2-1}+a_4^{(0)}\big(1-2b_3^{(1)}\big)
+2a_3^{(1)}\gamma t^{k_2},
\\
&\text{so}\quad 0=k_2\gamma +a_{4,k_2-1}^{(0)}\big(1-2b_{3,0}^{(1)}\big).
\end{align*}
This shows
\begin{gather*}
b_{4,0}^{(1)}=0,\qquad
b_{3,0}^{(1)}=\frac{k_1-k_2}{2(k_1+k_2)}
\in \bigg({-}\frac{1}{2},0\bigg)\cap\Q.
\end{gather*}
With respect to the basis $\uuuu{v}|_{(0,0)}$ of $K_{(0,0)}$,
the matrix of the residue endomorphism of the
logarithmic pole at $z=0$ of the $(TE)$-structure over $t^0=0$
is
\begin{eqnarray*}
B^{(1)}(0)&=& \rho^{(1)}C_1 + b_{3,0}^{(1)}D+b_{2,0}^{(1)}C_2.
\end{eqnarray*}
It is semisimple with the eigenvalues
$\rho^{(1)}\pm b_{3,0}^{(1)}$, whose difference is smaller
than 1. The monodromy is semisimple with the two different
eigenvalues $\exp\big({-}2\pi {\rm i}\big(\rho^{(1)}\pm b_{3,0}^{(1)}\big)\big)$.

\medskip
$(iii)$ Suppose $k_2=k_1$. As in the proof of $(ii)$, we
can make a coordinate change in $t$ and then obtain a
$\C\{t,z\}$-basis $\uuuu{\www v}$ of $\OO(H)_{(0,0)}$ with
\begin{gather*}
\www b_3^{(0)}=0 ,\qquad \www b_2^{(0)}=\www b_4^{(0)}=\gamma t^{k_1}
\end{gather*}
for an arbitrarily chosen $\gamma\in\C^*$.
Now the constant base change matrix
$T=\left(\begin{smallmatrix}1&\hphantom{-}1\\1&-1\end{smallmatrix}\right)$ gives the
basis $\uuuu{v}=\uuuu{\www v}\cdot T$ with
\begin{gather*}
b_2^{(0)}=b_4^{(0)}=0,\qquad b_3^{(0)}=\gamma t^{k_1}.
\end{gather*}
The vanishing of the coefficients in $\C\{t\}$ of
$C_2\cdot z^0$, $E\cdot z^0$, $D\cdot z^1$,
$C_2\cdot z^1$ and $E\cdot z^1$ in~\eqref{6.5}--\eqref{6.7}
tells the following:
\begin{align*}
&C_2\cdot z^0\colon&&\hspace{-55mm} a_2^{(0)}=0,
\\
&E\cdot z^0\colon&&\hspace{-55mm} a_4^{(0)}=0,
\\
&D\cdot z^1\colon&&\hspace{-55mm} 0=k_1\gamma t^{k_1-1}+a_3^{(0)},\quad
\text{so}\quad a_3^{(0)}=-k_1\gamma t^{k_1-1},
\\
&C_2\cdot z^1\colon&&\hspace{-55mm} b_2^{(1)}=\frac{b_3^{(0)}}{a_3^{(0)}}a_2^{(1)}
= \frac{-1}{k_1} \cdot t \cdot a_2^{(1)},\quad
\text{so}\quad b_{2,0}^{(1)}=0,
\\
&E\cdot z^1\colon&&\hspace{-55mm} b_4^{(1)}=\frac{b_3^{(0)}}{a_3^{(0)}}a_4^{(1)}
= \frac{-1}{k_1} \cdot t \cdot a_4^{(1)},\quad
\text{so}\quad b_{4,0}^{(1)}=0.
\end{align*}
With respect to the basis $\uuuu{v}|_{(0,0)}$ of $K_{(0,0)}$,
the matrix of the residue endomorphism of the logarithmic
pole at $z=0$ of the $(TE)$-structure over $t^0=0$ is
\begin{gather*}
B^{(1)}(0)= \rho^{(1)}C_1+b_{3,0}^{(1)}D.
\end{gather*}
It is diagonal with the eigenvalues $\rho^{(1)}\pm b_{3,0}^{(1)}$.
Therefore the monodromy has the eigenvalues
$\exp\big({-}2\pi {\rm i}\big(\rho^{(1)}\pm b_{3,0}^{(1)}\big)\big)$.

If $b_{3,0}^{(1)}\in\C\setminus \big(\frac{1}{2}\Z\setminus\{0\}\big)$, the eigenvalues
of the residue endomorphism do not differ by a~nonzero integer.
Because of Theorem~\ref{t3.23}$(c)$, then the monodromy is
semisimple.

We will show that the monodromy is also in the cases
$b_{3,0}^{(1)}\in\frac{1}{2}\Z\setminus\{0\}$ semisimple,
by reducing these cases to the case $b_{3,0}^{(1)}=0$.

Suppose $b_{3,0}^{(1)}\in\frac{1}{2}\N$. The case
$b_{3,0}^{(1)}\in \frac{1}{2}\Z_{<0}$ can be reduced to this
case by exchanging $v_1$ and~$v_2$.
We will construct a new $(TE)$-structure over $\big(M,t^0\big)=(\C,0)$
with the same monodromy and again with trace free pole part and
of generic type (Sem)
with logarithmic restriction over~$t^0$, but where
$B^{(1)}(0)$ is replaced by
\begin{gather*}
\www B^{(1)}(0)=\bigg(\rho^{(1)}+\frac{1}{2}\bigg)+
\bigg(b_{3,0}^{(1)}-\frac{1}{2}\bigg)D.
\end{gather*}
Applying this sufficiently often, we arrive at the case
$b_{3,0}^{(1)}=0$, which has semisimple monodromy.

The basis $\uuuu{\www v}:= \uuuu{v}\cdot
\left(\begin{smallmatrix}1&0\\0&z\end{smallmatrix}\right)$ of
$H':=H|_{\C^*\times (M,t^0)}$ in a neighborhood of $(0,0)$
defines a new $(TE)$-structure over $(M,0)$ because of
\begin{gather*}
z\nnn_{\paa_t}\uuuu{\www v}= \uuuu{\www v}
\bigl(z^{-1}a_2C_2+a_3D+za_4E\bigr)\qquad\text{and}\qquad a_2^{(0)}=0,
\\[1ex]
z^2\nnn_{\paa_z}\uuuu{\www v}= \uuuu{\www v}
\bigg(z\bigg(\rho^{(1)}+\frac{1}{2}\bigg)C_1+z^{-1}b_2C_2+\bigg(b_3-z\frac{1}{2}\bigg)D
+zb_4E\bigg)\qquad \text{and}\qquad b_2^{(0)}=0.
\end{gather*}
Of course, it has the same monodromy.
The restriction over $t^0=0$ has a logarithmic pole at $z=0$
because $b_2^{(1)}=\frac{-1}{k_1}ta_2^{(1)}$
and $b_3^{(0)}=\gamma t^{k_1}$ with $k_1\in\N$.
Its generic type is still (Sem).
Its numbers $\www k_1$ and $\www k_2$ satisfy
$\www k_1+\www k_2=\deg_t\det\www\UU=\deg_t\big(b_3^{(0)}\big)^2=2k_1$.
The assumption $\www k_1<\www k_2$ would lead together with
part $(ii)$ to two different eigenvalues of the monodromy,
a contradiction. Therefore $\www k_1=\www k_2=k_1$.
Thus we are in the same situation as before, with
$b_{3,0}^{(1)}$ diminuished by~$\frac{1}{2}$.\looseness=-1

\medskip
$(b)$ Suppose that the generic type is (Bra).

$(i)$ and $(ii)$ $\UU$ is nilpotent, but not 0.
We can choose a $\C\{t,z\}$-basis $\uuuu{v}$ of
$\OO(H)_{(0,0)}$ such that
\begin{gather*}
B^{(0)}=b_2^{(0)}C_2,\qquad \text{so}\quad
b_3^{(0)}=b_4^{(0)}=0.
\end{gather*}
Then $\delta^{(1)}=-b_2^{(0)}b_4^{(1)}$.
Here $\deg_tb_2^{(0)}=k_1$ and $\deg_t\delta^{(1)}=k_1+k_2$,
so $k_2=\deg_tb_4^{(1)} \geq 0$. We can make a
coordinate change in $t$ such that afterwards
\begin{gather*}
b_2^{(0)}b_4^{(1)}=\gamma^2t^{k_1+k_2},
\end{gather*}
for an arbitrarily chosen $\gamma\in\C^*$.
Then a diagonal base change leads to a basis which is
again called $\uuuu{v}$ with matrices which are again
called $A$ and $B$ with
\begin{gather*}
b_2^{(0)}=\gamma t^{k_1},\qquad
b_3^{(0)}=b_4^{(0)}=0,\qquad
b_4^{(1)}=\gamma t^{k_2}.
\end{gather*}
The vanishing of the coefficients in $\C\{t\}$ of
$C_2\cdot z^0$, $D\cdot z^0$, $C_2\cdot z^1$, $D\cdot z^1$
and $E\cdot z^2$ in~\eqref{6.5}--\eqref{6.7}
tells the following
\begin{align*}
&C_2\cdot z^0\colon &&\hspace{-40mm} a_3^{(0)}=0,
\\
&D\cdot z^0\colon &&\hspace{-40mm}a_4^{(0)}=0,
\\
&C_2\cdot z^1\colon
&&\hspace{-40mm}0=k_1\gamma t^{k_1-1}+a_2^{(0)}\big(1+2b_3^{(1)}\big)-2a_3^{(1)}\gamma t^{k_1},
\\
&&&\hspace{-40mm}\text{so}\quad \deg_t a_2^{(0)}=k_1-1, \quad
0=k_1\gamma +a_{2,k_1-1}^{(0)}\big(1+2b_{3,0}^{(1)}\big),
\\
&D\cdot z^1\colon &&\hspace{-40mm}a_2^{(0)}\gamma t^{k_2}=a_4^{(1)}\gamma t^{k_1},\quad
\text{so}\quad t^{k_2}=a_4^{(1)}\frac{t^{k_1}}{a_2^{(0)}},
\\
&&&\hspace{-40mm}\text{so}\quad k_2=1+\deg a_4^{(1)}\geq 1,\quad\text{and}\quad
a_4^{(1)}=a_2^{(0)}t^{k_2-k_1},
\\
&E\cdot z^2\colon &&\hspace{-40mm}
0=k_2\gamma t^{k_2-1}+2a_3^{(1)}\gamma t^{k_2}-2a_4^{(1)}b_3^{(1)},
\\
&&&\hspace{-40mm} \text{so}\quad 0=k_2\gamma -2a_{2,k_1-1}^{(0)}b_{3,0}^{(1)}.
\end{align*}
This shows
\begin{gather*}
k_2\geq 1,\qquad b_{4,0}^{(1)}=0,\qquad
b_{3,0}^{(1)}=\frac{-k_2}{2(k_1+k_2)}
\in \bigg({-}\frac{1}{2},0\bigg)\cap\Q.
\end{gather*}
With respect to the basis $\uuuu{v}|_{(0,0)}$ of $K_{(0,0)}$,
the matrix of the residue endomorphism of the
logarithmic pole at $z=0$ of the $(TE)$-structure over $t^0=0$
is
\begin{gather*}
B^{(1)}(0)= \rho^{(1)}C_1 + b_{3,0}^{(1)}D+b_{2,0}^{(1)}C_2.
\end{gather*}
It is semisimple with the eigenvalues
$\rho^{(1)}\pm b_{3,0}^{(1)}$, whose difference is smaller
than 1. The mono\-dromy is semisimple with the two different
eigenvalues $\exp\big({-}2\pi {\rm i}\big(\rho^{(1)}\pm b_{3,0}^{(1)}\big)\big)$.
\end{proof}

\subsection[1-parameter unfoldings with trace free pole part
of logarithmic pure $(TLE)$-structures over a point]
{1-parameter unfoldings with trace free pole part
of logarithmic \\pure $\boldsymbol{(TLE)}$-structures over a point}\label{c6.2}

\looseness=1
Such unfoldings are themselves pure $(TLE)$-structures
over $(\C,0)$, see Remark~\ref{t3.19}$(iii)$ respectively
\cite[Chapter~VI, Theorem~2.1]{Sa02} or~\cite[Theorem~5.1(c)]{DH20-2}.
Their restrictions over $t^0=0$ have a logarithmic
pole at $z=0$.
Theorem~\ref{t6.3} classifies such pure $(TLE)$-structures.
The underlying $(TE)$-structures were subject of
Definition~\ref{t6.1} and Theorem~\ref{t6.2}.
They gave their generic type and
invariants $(k_1,k_2)\in\N^2$ (for the generic types
(Sem) and (Bra)) and $k_1\in\N$ (for the generic type (Reg)).
Theorem~\ref{t6.3} will give an invariant $k_1\in\N$ also
for the generic type (Log) with Higgs field $\neq 0$.
Lemma~\ref{t3.9}$(b)$ gave the invariant $\rho^{(1)}\in\C$.
The coordinate on $\C$ is again called~$t$.

\begin{Theorem}\label{t6.3}
Any pure rank $2$ $(TLE)$-structure over $\big(M,t^0\big)=(\C,0)$
with trace free pole part and with logarithmic
restriction over $t^0$ has after a suitable coordinate change
in $t$ a unique Birkhoff normal form in the following list.
Here the Birkhoff normal form
consists of two matrices $A$ and $B$ which are associated
to a global basis $\uuuu{v}$ of $H$ whose restriction to
$\{\infty\}\times \big(M,t^0\big)$ is flat with
respect to the residual connection along
$\{\infty\}\times \big(M,t^0\big)$, via
$z\nnn_{\paa_t}\uuuu{v}=\uuuu{v}A$ and
\mbox{$z^2\nnn_{\paa_z}\uuuu{v}=\uuuu{v}B$}.
The matrices have the shape
\begin{gather}
A=a_2^{(0)}C_2+a_3^{(0)}D+a_4^{(0)}E,\nonumber
\\
B=z\rho^{(1)}C_1-\gamma tA
+zb_2^{(1)}C_2+zb_3^{(1)}D,\label{6.19}
\end{gather}
with $a_2^{(0)},a_3^{(0)},a_4^{(0)}\in\C[t]$,
$\rho^{(1)},\gamma\in\C$, $b_2^{(1)},b_3^{(1)}\in\C$
$($so here $zb_4^{(1)}E$ does not turn up, resp.~$b_4^{(1)}\allowbreak=0)$.
The left column of the following list gives the
generic type of the
underlying $(TE)$-structure and, depending on the type,
the invariant $k_1\in\N$ or the invariants $k_1,k_2\in\N$
from Definition~$\ref{t6.1}$ of the underlying $(TE)$-structure.
The invariant $\rho^{(1)}\in \C$ is arbitrary and
is not listed in the table.
$\zeta\in\C$, $\alpha_3\in \R_{\geq 0}\cup\H$,
$\alpha_4\in\C\setminus \{-1\}$, $k_1\in\N$ and $k_2\in\N$
are invariants in some cases. In~the first $6$ cases, $a_i^{(0)}$ is determined by
$b_i^{(0)}=-\gamma t a_i^{(0)}$.
\[
\def\arraystretch{1.5}
\begin{tabular}{c|c|c|c|c|c|c}
\hline
\raisebox{1mm}[3.9ex][2.2ex]{\parbox[c]{25mm}{\centering Generic type and~invariants}} & $\gamma$
& $b_2^{(0)}$ & $b_3^{(0)}$ & $b_4^{(0)}$ &
$b_2^{(1)}$ & $b_3^{(1)}$
\\
\hline
$($Sem$)$ & & & & & &
\\
$k_2-k_1>0$\ odd & $\frac{2}{k_1+k_2}$ &
$t^{k_1}$ & $0$ & $t^{k_2}$ & $0$ & $\frac{k_1-k_2}{2(k_1+k_2)}$
\\
$k_2-k_1\in 2\N$& $\frac{2}{k_1+k_2}$ &
$t^{k_1}$ & $\zeta t^{(k_1+k_2)/2}$ & $\big(1-\zeta^2\big)t^{k_2}$ &
$0$ & $\frac{k_1-k_2}{2(k_1+k_2)}$
\\
$k_2=k_1$ & $\frac{1}{k_1}$ & $0$ & $t^{k_1}$ & $0$ &
$0$ & $\alpha_3$
\\
\hline
$($Bra$)$, $k_1$, $k_2$ & $\frac{1}{k_1+k_2}$ & $t^{k_1}$ &
$t^{k_1+k_2}$ & $-t^{k_1+2k_2}$ & $0$ & $\frac{-k_2}{2(k_1+k_2)}$
\\
\hline
$($Reg$)$, $k_1$ & $\frac{1+\alpha_4}{k_1}$ & $t^{k_1}$ & $0$ &
$0$ & $0$ & $\frac{1}{2}\alpha_4$
\\
$($Reg$)$, $k_1$ & $\frac{1}{k_1}$ & $t^{k_1}$ & $0$ & $0$ & $1$ & $0$
\\ \hline \hline
Generic type & $\gamma$
& $a_2^{(0)}$ & $a_3^{(0)}$ & $a_4^{(0)}$ &
$b_2^{(1)}$ & $b_3^{(1)}$
\\ \hline
$($Log$)$ & $0$ & $k_1t^{k_1-1}$ & $0$ & $0$ & $0$ & $-\frac{1}{2}$
\\
$($Log$)$ & $0$ & $0$ & $0$ & $0$ & $0$ & $\alpha_3$
\\
$($Log$)$ & $0$ & $0$ & $0$ & $0$ & $1$ & $0$
\\
\hline
\end{tabular}
\]
\end{Theorem}

Before the proof, several remarks on these Birkhoff
normal forms are made.
The proof is given after Remark~\ref{t6.6}.

\begin{Remarks}\label{t6.4}\qquad
\begin{enumerate}\itemsep=0pt
\item[$(i)$] The matrix $B(0)=zB^{(1)}(0)$ is the matrix of the
logarithmic pole at $z=0$ of the restriction over $t^0=0$
of the $(TE)$-structure. In~all cases except the 6th case and the 9th case, it is
$z\big(\rho^{(1)}C_1+b_3^{(1)}D\big)$, so it is diagonal. In~these cases the monodromy is semisimple with eigenvalues
$\exp\big({-}2\pi {\rm i} \big(\rho^{(1)}\pm b_3^{(1)}\big)\big)$. In~the 6th case and the 9th case,
this matrix is $z\big(\rho^{(1)}C_1+C_2\big)$. Then the matrix and
the monodromy have a $2\times 2$ Jordan block,
and the monodromy has the eigenvalue $\exp\big({-}2\pi {\rm i}\rho^{(1)}\big)$. In~all cases, the leading exponents
(defined in Theorem~\ref{t4.20})
of the logarithmic $(TE)$-structure over $t^0$ are
called $\alpha_1^0$ and $\alpha_2^0$, and they~are
\begin{gather*}
\alpha_{1/2}^0=\rho^{(1)}\pm b_3^{(1)},\qquad
\text{i.e.},\qquad \frac{\alpha_1^0+\alpha_2^0}{2}=\rho^{(1)},\qquad
\alpha_1^0-\alpha_2^0=2b_3^{(1)}.
\end{gather*}
The 6th and 9th cases turn up again in Theorem~\ref{t6.7}.
See Remarks~\ref{t6.8}$(iv)$--$(vi)$.

\item[$(ii)$] In the generic types (Sem), the critical values satisfy
$u_2=-u_1$ because the pole part is trace free,
$-\frac{u_1+u_2}{2}=\rho^{(0)}=0$.
They and the regular singular exponents $\alpha_1$ and
$\alpha_2$ can be calculated with the formulas~\eqref{4.6} and~\eqref{4.7}:
\begin{gather}\label{6.22}
\delta^{(0)}= -b_2^{(0)}b_4^{(0)}-\big(b_3^{(0)}\big)^2=-t^{k_1+k_2},
\\
u_{1/2}=\pm\sqrt{\frac{1}{4}(u_1-u_2)^2}=\pm\sqrt{-\delta^{(0)}}
=\pm t^{(k_1+k_2)/2},\label{6.23}
\\
\frac{\alpha_1+\alpha_2}{2}= \rho^{(1)}, \label{6.24}
\\
\alpha_1-\alpha_2= u_1^{-1}\delta^{(1)}
= \begin{cases}
0,& \text{gen. type (Sem) with }k_2-k_1>0\text{ odd,}\\
\frac{k_2-k_1}{k_1+k_2}\zeta,
& \text{gen. type (Sem) with }k_2-k_1\in 2\N,\\
-2\alpha_3,& \text{gen. type (Sem) with }k_2=k_1.
\end{cases} \label{6.25}
\end{gather}
If $k_2=k_1$ then $\{\alpha_1,\alpha_2\}
=\big\{\alpha_1^0,\alpha_2^0\big\}$, but if
$k_2>k_1$ then $\{\alpha_1,\alpha_2\}\neq
\big\{\alpha_1^0,\alpha_2^0\big\}$, except if $\zeta\in\{\pm 1\}$.

\item[$(iii)$] In the generic type (Bra), $\rho^{(1)}\in\C$ is arbitrary,
$b_3^{(1)}=\frac{-k_2}{2(k_1+k_2)}$, and $\delta^{(1)}$
varies as follows,
\begin{eqnarray}\label{6.26}
\delta^{(1)}= \frac{k_2}{k_1+k_2}t^{k_1+k_2}.
\end{eqnarray}

\item[$(iv)$] In the 5th, 7th and 8th cases in Theorem~\ref{t6.3},
the monodromy is semisimple and the $(TE)$-structure is regular singular.
Associate to it the data in Definition~\ref{t3.18}:
$H':=H|_{\C\times (M,t^0)}$, $M^{\rm mon}$, $N^{\rm mon}$,
$\Eig(M^{\rm mon})=\{\lambda_1,\lambda_2\}$, $H^\infty$, $C^{\alpha}$ for
$\alpha\in\C$ with ${\rm e}^{-2\pi {\rm i} \alpha_j}\in\{\lambda_1,\lambda_2\}$.
The~leading exponents of the logarithmic $(TE)$-structure
over $t^0$ are called $\alpha_1^0$ and $\alpha_2^0$ as in $(i)$.
The leading exponents of the $(TE)$-structure over
$t\in \C\setminus\{0\}$ are now called $\alpha_1$ and $\alpha_2$.
Possibly after renumbering $\lambda_1$ and $\lambda_2$,
$\alpha_1^0$ and $\alpha_2^0$, and $\alpha_1$ and $\alpha_2$,
we have ${\rm e}^{-2\pi {\rm i}\alpha_j^0}={\rm e}^{-2\pi {\rm i}\alpha_j}=\lambda_j$
and the relations in the following table:
\begin{gather}\label{6.27}
\def\arraystretch{1.3}
\begin{tabular}{c|c|c|c|c}
\hline
In Theorem~\ref{t6.3} & $\alpha_1^0$ & $\alpha_2^0$&$\alpha_1$ &$\alpha_2$
\\ \hline
5th case &$\rho^{(1)}+\frac{1}{2}\alpha_4$ &$\rho^{(1)}-\frac{1}{2}\alpha_4$
&$\alpha_1^0$ & $\alpha_2^0-1$
\\
7th case & $\rho^{(1)}-\frac{1}{2}$ &
$\rho^{(1)}+\frac{1}{2}$ & $\alpha_1^0$ & $\alpha_2^0$
\\
8th case &$\rho^{(1)}+\alpha_3$ &$\rho^{(1)}-\alpha_3$ & $\alpha_1^0$ & $\alpha_2^0$
\\
\hline
\end{tabular}
\end{gather}
And there exist sections $s_j\in C^{\alpha_j}\setminus\{0\}$ with
\begin{gather}
\OO(H)_0= \C\{t,z\}\bigg(s_1+\frac{-1}{1+\alpha_4}t^{k_1}s_2\bigg)
\oplus \C\{t,z\}(zs_2)
\qquad \text{in the 5th case,} \label{6.28}\\
\OO(H)_0= \C\{t,z\}\big(s_1+t^{k_1}z^{-1}s_2\big)
\oplus \C\{t,z\}s_2
\qquad \text{in the 7th case,} \label{6.29}\\[.5ex]
\OO(H)_0= \C\{t,z\}s_1
\oplus \C\{t,z\}s_2
\qquad \text{in the 8th case.}\label{6.30}
\end{gather}
{\sloppy
One confirms~\eqref{6.28}--\eqref{6.30}
immediately by calculating the matrices
$A$ and $B$ with \mbox{$z\nnn_{\paa_t}\uuuu{v}=\uuuu{v}A$}
and $z^2\nnn_{\paa_z}\uuuu{v}=\uuuu{v}B$ for
$\uuuu{v}$ the basis in~\eqref{6.28}--\eqref{6.30}.

}

\item[$(v)$] Theorem~\ref{t6.7} contains for the 6th and 9th cases
in Theorem~\ref{t6.3} a description similar to part $(iv)$.
See Remarks~\ref{t6.8}$(iv)$--$(vi)$.
\end{enumerate}
\end{Remarks}

\begin{Remarks}\label{t6.5}
These remarks study the behaviour of the $(TE)$-structures
in Theorem~\ref{t6.3} under pull back via maps
$\varphi\colon (\C,0)\to(\C,0)$.
The normal forms in Theorem~\ref{t6.3}
are chosen such that the pull backs
by maps $\varphi$ with
$\varphi(s)=s^n$ for some $n\in\N$
are again normal forms in Theorem~\ref{t6.3}.

\begin{enumerate}\itemsep=0pt
\item[$(i)$] A general observation:
Let $\big(H\to\C\times \big(M,t^0\big),\nnn\big)$ be a $(TE)$-structure
over $\big(M,t^0\big)=(\C,0)$ of rank $r\in\N$.
Let $\uuuu{v}$ be $\C\{t,z\}$-basis
of $\OO(H)_0$ with $z\nnn_{\paa_t}\uuuu{v}=\uuuu{v}A$
and $z^2\nnn_{\paa_z}\uuuu{v}=\uuuu{v}B$ and~$A,B\in M_{r\times r}(\C\{t,z\})$.
Choose $n\in\N$ and consider a map
$\varphi\colon (\C,0)\to(\C,0)$, $s\mapsto \varphi(s)=t$.
Then the pull back $(TE)$-structure $\varphi^*(H,\nnn)$
has the basis $\uuuu{\www v}(z,s)=\uuuu{v}(z,\varphi(s))$.
The~matrices $\www A,\www B\in M_{r\times r}(\C\{s,z\})$
with $z\nnn_{\paa_s}\uuuu{\www v}=\uuuu{\www v}\www A$,
$z^2\nnn_{\paa_z}\uuuu{\www v}=\uuuu{\www v}\www B$ are
\begin{gather}\label{6.31}
\www A=\paa_s\varphi(s)\cdot A(z,\varphi(s)),\qquad
\www B=B(z,\varphi(s)).
\end{gather}

\item[$(ii)$] These formulas~\eqref{6.31} show for the 1st to 7th
cases in the list in Theorem~\ref{t6.3} the following:
The pull back via $\varphi\colon (\C,0)\to(\C,0)$ with
$\varphi(s)=s^n$ for some $n\in\N$
of such a $(TE)$-structure with invariants $(k_1,k_2)$
or $k_1$
is a $(TE)$-structure in the same case, where the invariants
$(k_1,k_2)$ or $k_1$ are replaced by $\big(\www k_1,\www k_2\big)=(nk_1,nk_2)$
or $\www{k}_1=nk_1$,
and where all other invariants coincide with the old invariants.

\item[$(iii)$] The following table says
which of the $(TE)$-structures in the 1st to 7th cases
in the list in Theorem~\ref{t6.3}
are not induced by other such $(TE)$-structures:
\[
\def\arraystretch{1.3}
\begin{tabular}{l|c}
\hline
Generic type and invariants & Not induced if
\\ \hline
(Sem)$\colon\ k_2-k_1>0$ odd & $\gcd(k_1,k_2)=1$
\\
(Sem)$\colon\ k_2-k_1\in 2\N$, $\zeta= 0$ & $\gcd(k_1,k_2)=1$
\\
(Sem)$\colon\ k_2-k_1\in 2\N$, $\zeta\neq 0$ &
$\gcd\big(k_1,\frac{k_1+k_2}{2}\big)=1$
\\
(Sem)$\colon\ k_2=k_1$ &$ k_2=k_1=1 $
\\ \hline
(Bra) & $\gcd(k_1,k_2)=1$
\\ \hline
(Reg)$\colon\ N^{\rm mon}=0$ &$ k_1=1$
\\
(Reg)$\colon\ N^{\rm mon}\neq$ 0 & $k_1=1$
\\ \hline
(Log) & $k_1=1$
\\
\hline
\end{tabular}
\]

\item[$(iv)$] In the 8th and 9th cases, the $(TE)$-structure is
induced by its restriction over $t^0$ via the map
$\varphi\colon \big(M,t^0\big)\to\big\{t^0\big\}$, so it is constant
along $M$.

\item[$(v)$] The formulas~\eqref{6.28} and~\eqref{6.29}
confirm part $(ii)$ for the 5th and 7th cases in Theorem~\ref{t6.3}.
Formula~\eqref{6.30} confirms part $(iv)$ in the
8th case in Theorem~\ref{t6.3}.
Analogous statements to part $(ii)$ and part $(iv)$
hold for the cases in Theorem~\ref{t6.7}.
They follow from the formulas~\eqref{6.49},~\eqref{6.50}
and~\eqref{6.51} there,
which are analogous to~\eqref{6.28},~\eqref{6.29}
and~\eqref{6.30}.
See Remarks~\ref{t6.8}$(ii)$ and $(iii)$.
\end{enumerate}
\end{Remarks}

\begin{Remark}\label{t6.6}
In the 2nd and 4th cases in the list in Theorem~\ref{t6.3},
another $\C\{t,z\}$-basis $\uuuu{\www v}$
of~$\OO(H)_0$ with nice matrices $\www A$ and $\www B$
is
\begin{gather*}
\uuuu{\www v}=\uuuu{v}\cdot T\quad\text{with}\quad
T=C_1+\frac{a_3^{(0)}}{a_2^{(0)}}E
=\begin{cases}
C_1+\zeta t^{(k_2-k_1)/2}E & \text{in the 2nd case,}\\
C_1+t^{k_2}E & \text{in the 4th case.}
\end{cases}
\end{gather*}
In the 2nd case
\begin{gather*}
\www A= -\gamma^{-1}\big(t^{k_1-1}C_2 + t^{k_2-1}E\big) + z\frac{k_2-k_1}{2}\zeta t^{(k_2-k_1-2)/2}E,
\\
\www B= z\rho^{(1)} C_1 -\gamma t\www A +
z b_3^{(1)} D.\nonumber
\end{gather*}
In the 4th case
\begin{gather*}
\www A= -\gamma^{-1}t^{k_1-1}C_2 + zk_2t^{k_2-1}E,\\
\www B= z\rho^{(1)} C_1 -\gamma t\www A +z b_3^{(1)} D.\nonumber
\end{gather*}
These matrices are not in Birkhoff normal form.
The basis $\uuuu{\www v}$ is still a global basis
of the pure $(TLE)$-structure, but
the sections $\www v_j|_{\{\infty\}\times M}$ are not
flat with respect to the residual connection along
$\{\infty\}\times M$.
\end{Remark}

\begin{proof}[Proof of Theorem~\ref{t6.3}]
Consider any pure $(TLE)$-structure over $\big(M,t^0\big)=(\C,0)$
with trace free pole part and with logarithmic restriction
to $t^0$. Choose a global basis $\uuuu{v}$ of $H$
whose restriction to
$\{\infty\}\times \big(M,t^0\big)$ is flat with
respect to the residual connection along
$\{\infty\}\times \big(M,t^0\big)$.
Its matrices $A$ and $B$ with
$z\nnn_{\paa_t}\uuuu{v}=\uuuu{v}A$ and
$z^2\nnn_{\paa_z}\uuuu{v}=\uuuu{v}B$
have because of~\eqref{3.18} (in Lemma~\ref{t3.11})
the shape~\eqref{6.19} and
\begin{gather*}
B=z\rho^{(1)}C_1+\big(b_2^{(0)}+zb_2^{(1)}\big)C_2 +
\big(b_3^{(0)}+zb_3^{(1)}\big)D+\big(b_4^{(0)}+zb_4^{(1)}\big)E
\end{gather*}
with $a_j^{(0)}\in\C\{t\}$, $b_j^{(0)}\in t\C\{t\}$,
$b_j^{(1)}\in\C$. They satisfy the relations~\eqref{3.28}
(and, equivalently,~\eqref{6.5}--\eqref{6.7}),
so, more explicitly,
\begin{gather}\label{6.37}
a_i^{(0)}b_j^{(0)}=a_j^{(0)}b_i^{(0)}\qquad\text{for}\quad
(i,j)\in\{(2,3),(2,4),(3,4)\},
\\[1ex]
\begin{pmatrix}-\paa_tb_2^{(0)} \\ -\paa_tb_3^{(0)} \\
-\paa_tb_4^{(0)} \end{pmatrix}
= \begin{pmatrix}1+2b_3^{(1)} & -2b_2^{(1)} & 0 \\
-b_4^{(1)} & 1 & b_2^{(1)} \\
0 & 2b_4^{(1)} & 1-2b_3^{(1)} \end{pmatrix}
\begin{pmatrix} a_2^{(0)} \\ a_3^{(0)} \\ a_4^{(0)}
\end{pmatrix}\!. \label{6.38}
\end{gather}

First we consider the cases when all $a_j^{(0)}$ are 0.
Then also all $b_j^{(0)}$ are 0, because of
$b_j^{(0)}\in t\C\{t\}$ and because of the differential
equations~\eqref{6.38}.
Then $B=zB^{(1)}$, and it is clear that this matrix
can be brought to the form $B=z\rho^{(1)}C_1+ z\alpha_3 D$
or $B=z\rho^{(1)}C_1+zC_2$
by a constant base change.
The number $\alpha_3\in\C$ can be replaced by $-\alpha_3$, so
$\alpha_3\in\R_{\geq 0}\cup\H$ is unique.
This gives the last two cases in the list.
There the generic type is (Log).

For the rest of the proof, we consider the cases
when at least one $a_j^{(0)}$ is not 0.
Then~\eqref{6.37} says
\begin{gather}\label{6.39}
\big(b_2^{(0)},b_3^{(0)},b_4^{(0)}\big)=\frac{b_j^{(0)}}{a_j^{(0)}}
\cdot \big(a_2^{(0)},a_3^{(0)},a_4^{(0)}\big),\qquad
\text{so} \quad B^{(0)} = \frac{b_j^{(0)}}{a_j^{(0)}}\cdot A^{(0)}.
\end{gather}
If $b_j^{(0)}=0$ then $b_2^{(0)}=b_3^{(0)}=b_4^{(0)}=0$,
and the generic type is (Log).
If $b_j^{(0)}\neq 0$, then the generic type is
(Sem) or (Bra) or (Reg).
Consider for a moment the cases when the residue endomorphism
of the logarithmic pole at $z=0$ of the $(TE)$-structure
over $t^0$ is semisimple. By Theorem~\ref{t6.1}, these cases
include the generic types (Sem) and (Bra).
Then a linear base change gives $b_2^{(1)}=b_4^{(1)}=0$,
so that the $3\times 3$-matrix in~\eqref{6.38} is diagonal.
Then denote $\www\beta_j:=\deg_t b_j^{(0)}\in\N$.
A coordinate change in $t$ leads to
$b_j^{(0)}=b_{j,\www\beta_j}^{(0)}\cdot t^{\www\beta_j}$.
The differential equation in~\eqref{6.38} leads to
$a_j^{(0)}=a_{j,\www\beta_j-1}^{(0)}\cdot t^{\www\beta_j-1}$,
and to $b_j^{(0)}/a_j^{(0)}=-\gamma t$ for some $\gamma\in\C^*$.
Define
\begin{gather}\label{6.40}
\beta_2=\frac{1+2b_3^{(1)}}{\gamma},\qquad
\beta_3=\frac{1}{\gamma},\qquad
\beta_4=\frac{1-2b_3^{(1)}}{\gamma}.
\end{gather}
Now~\eqref{6.39} and the differential equations in~\eqref{6.38}
show $\www\beta_j=\beta_j$ and
\begin{gather}
b_2^{(0)}=0\qquad\text{or}\qquad\big(\beta_2\in\N\text{ and }
b_2^{(0)}=b_{2,\beta_2}^{(0)}\cdot t^{\beta_2}\big)\neq 0,
\nonumber
\\
b_3^{(0)}=0\qquad\text{or}\qquad\big(\beta_3\in\N\text{ and }
b_3^{(0)}=b_{3,\beta_3}^{(0)}\cdot t^{\beta_3}\big)\neq 0,
\label{6.41}
\\
b_4^{(0)}=0\qquad\text{or}\qquad\big(\beta_4\in\N\text{ and }
b_4^{(0)}=b_{4,\beta_4}^{(0)}\cdot t^{\beta_4}\big)\neq 0.
\nonumber
\end{gather}

Now we discuss the generic types (Sem), (Bra), (Reg)
and (Log) separately.

\medskip\noindent
{\it Generic type $($Sem$)$}.
By Theorem~\ref{t6.2}, we can choose the basis $\uuuu{v}$
such that $b_2^{(1)}=b_4^{(1)}=0$. In~the cases $k_2>k_1$, by Theorem~\ref{t6.2},
$b_3^{(1)}$ is up to the sign unique, and we can choose it
to be
\begin{gather*}
b_3^{(1)}=\frac{k_1-k_2}{2(k_1+k_2)}\in\bigg({-}\frac{1}{2},0\bigg)\cap\Q
\end{gather*}
(possibly by exchanging $v_1$ and $v_2$). In~the cases $k_2=k_1$
we write $\alpha_3:=b_3^{(1)}\in\C$.
We can change its sign and get a unique
$\alpha_3\in\R_{\geq 0}\cup\H$.
We make a suitable coordinate change in $t$ and obtain
$b_2^{(0)}$, $b_3^{(0)}$, $b_4^{(0)}$ as in~\eqref{6.41}.
The relations~\eqref{6.8} and~\eqref{6.9} still hold.
Equation~\eqref{6.9} implies
\begin{gather*}
\big(b_2^{(0)}b_4^{(0)}\neq 0,\ \beta_2+\beta_4=k_1+k_2\big)
\qquad\text{or}\qquad
\big(b_3^{(0)}\neq 0,\ 2\beta_3=k_1+k_2\big)
\end{gather*}
(or both). In~both cases~\eqref{6.40} gives
\begin{gather*}
\gamma=\frac{2}{k_1+k_2}.
\end{gather*}
Thus
\begin{gather}\label{6.43}
(\beta_2,\beta_3,\beta_4)=
\begin{cases}
\big(k_1,\frac{k_1+k_2}{2},k_2\big)&\text{if}\quad k_2>k_1,
\\
(k_1(1+2\alpha_3),k_1,k_1(1-2\alpha_3)&\text{if}\quad k_2=k_1.\end{cases}
\end{gather}

In the cases $k_2>k_1$, we have $\beta_2<\beta_3<\beta_4$.
Then~\eqref{6.41} and the relation~\eqref{6.8} imply
$b_2^{(0)}\neq 0$, so $b_{2,\beta_2}^{(0)}\neq 0$.
The nonvanishing $\delta^{(0)}\neq 0 $ implies
$b_{2}^{(0)}b_{4}^{(0)}+\big(b_{3}^{(0)}\big)^2
\neq 0$.

In the case $k_2-k_1>0$ even, a linear coordinate change in $t$
and a diagonal base change allow to reduce the triple
$\big(b_{2,\beta_2}^{(0)},b_{3,\beta_3}^{(0)},b_{4,\beta_4}^{(0)}\big)
\in\C^3$ to a triple $\big(1,\zeta,1-\zeta^2\big)$
with $\zeta\in\C$ unique.

In the case $k_2-k_1>0$ odd, we have $\beta_3\notin\N$,
so $b_3^{(0)}=0$,
and a linear coordinate change in $t$ and a diagonal base
change allow to reduce the pair
$\big(b_{2,\beta_2}^{(0)},b_{4,\beta_4}^{(0)}\big)
\in(\C^*)^2$ to the pair $(1,1)$.

In the cases $k_2=k_1$ and $\alpha_3\neq 0$,~\eqref{6.8}
and~\eqref{6.43} imply $b_2^{(0)}=b_4^{(0)}=0$.
Then a linear coordinate change in $t$ allows to reduce
$b_{3,\beta_3}^{(0)}$ to the value $1$.

In the cases $k_2=k_1$ and $\alpha_3=0$,
as in the proof of Theorem~\ref{t6.2}$(a)$ $(iii)$,
a base change with constant coefficients leads to
$b_2^{(0)}=b_4^{(0)}=0$. Then a linear coordinate change
in $t$ allows to reduce $b_{3,\beta_3}^{(0)}$ to the value
$1$. In~all cases of generic type (Sem), we obtain the
normal forms in the list in Theorem~\ref{t6.3}.

\medskip\noindent
{\it Generic type $($Bra$)$.}
By Theorem~\ref{t6.2}, we can choose the basis $\uuuu{v}$
such that $b_2^{(1)}=b_4^{(1)}=0$,
and~$b_3^{(1)}$ is up to the sign unique.
We can choose it to be
\begin{gather*}
b_3^{(1)}=\frac{-k_2}{2(k_1+k_2)}\in
\Bigl(-\frac{1}{2},0\Bigr)\cap\Q
\end{gather*}
(possibly by exchanging $v_1$ and $v_2$).
We make a suitable coordinate change in $t$ and obtain
$b_2^{(0)}$, $b_3^{(0)}$, $b_4^{(0)}$ as in~\eqref{6.41}.
The nonvanishing $\delta^{(1)}\neq 0$
and $\deg_t\delta^{(1)}=k_1+k_2$ say
\begin{gather*}
0\neq \delta^{(1)}=-2b_3^{(1)}b_3^{(0)},
\end{gather*}
so $b_3^{(0)}\neq 0$ and
\begin{gather*}
\frac1\gamma=\beta_3=\deg b_3^{(0)}=\deg\delta^{(1)}=k_1+k_2,\qquad
\gamma=\frac{1}{k_1+k_2},
\\
(\beta_2,\beta_3,\beta_4)=(k_1,k_1+k_2,k_1+2k_2).
\end{gather*}
The relation~\eqref{6.8} still holds, and it implies
$b_2^{(0)}\neq 0$. The vanishing $\delta^{(0)}=0$ says
$b_{2,\beta_2}^{(0)}b_{4,\beta_4}^{(0)}+\big(b_{3,\beta_3}^{(0)}\big)^2
= 0$. Together with $b_{2,\beta_2}^{(0)}\neq 0$ and
$b_{3,\beta_3}^{(0)}\neq 0$ it implies
$b_{4,\beta_4}^{(0)}\neq 0$.
A linear coordinate change in $t$
and a diagonal base change allow to reduce the triple
$\big(b_{2,\beta_2}^{(0)},b_{3,\beta_3}^{(0)},b_{4,\beta_4}^{(0)}\big)
\in(\C^*)^3$ to the triple $(1,1,-1)$. We obtain the
normal form in the list in Theorem~\ref{t6.3}.

\medskip\noindent
{\it Generic type $($Reg$)$.}
First we consider the case when the residue endomorphism
of the logarithmic pole at $z=0$ of the $(TE)$-structure
over $t^0$ is semisimple. Then a linear base change
gives $b_2^{(1)}=b_4^{(1)}=0$. And a suitable coordinate
change in $t$ gives $b_2^{(0)}$, $b_3^{(0)}$, $b_4^{(0)}$
as in~\eqref{6.41}.

First consider the case $b_3^{(1)}\neq 0$. Then
the vanishing $0=\delta^{(1)}=-2b_3^{(1)}b_3^{(0)}$
says $b_3^{(0)}=0$. Now the vanishing
$0=\delta^{(0)}=-b_2^{(0)}b_4^{(0)}$ says that either
$b_2^{(0)}=0$ or $b_4^{(0)}=0$. Both together cannot be
0 as the generic type is (Reg) and not (Log).
After possibly exchanging $v_1$ and $v_2$, we suppose
$b_2^{(0)}\neq 0$, $b_4^{(0)}=0$. Now $k_1=\beta_2$.
Write $\alpha_4:=2b_3^{(1)}\in\C$.
By~\eqref{6.40},
\begin{eqnarray}\label{6.46}
\gamma=\frac{1+\alpha_4}{k_1}.
\end{eqnarray}
A diagonal base change allows to reduce $b_{2,\beta_2}^{(0)}$
to $1$.

Now consider the case $b_3^{(1)}=0$.
Then $\beta_2=\beta_3=\beta_4=1/\gamma$, and this is
equal to $k_1$, as $\beta_j\in\N$ for at least one $j$.
Write $\alpha_4:=b_3^{(1)}=0$. Then
\eqref{6.46} still holds.
By a base change with constant coefficients, we can obtain
$b_2^{(0)}=t^{k_1}$ and $b_3^{(0)}=0$.
The vanishing $0=\delta^{(0)}=-b_2^{(0)}b_4^{(0)}$
tells $b_4^{(0)}=0$.
For $\alpha_4\neq 0$ as well as for $\alpha_4=0$,
we obtain the normal form in the 5th case in the list
in Theorem~\ref{t6.3}.

Finally consider the case when the residue endomorphism
of the logarithmic pole at $z=0$ of the $(TE)$-structure
over $t^0$ has a $2\times 2$ Jordan block.
A base change with constant coefficients leads to
$b_3^{(1)}=b_4^{(1)}=0$ and $b_2^{(1)}=1$.
We will lead the assumption $b_4^{(0)}\neq 0$ as well as the
assumption $b_4^{(0)}=0,b_3^{(0)}\neq 0$ to a contradiction.

{\samepage
Assume $b_4^{(0)}\neq 0$.
Denote $\beta_4:=\deg_t b_4^{(0)}\in\N$.
A coordinate change in $t$ leads to
$b_4^{(0)}=\frac{-1}{\beta_4}t^{\beta_4}$.
The differential equation in~\eqref{6.38} for $b_4^{(0)}$
gives $a_4^{(0)}=t^{\beta_4-1}$.
Now~\eqref{6.39} gives $b_3^{(0)}=\frac{-1}{\beta_4}ta_3^{(0)}$.
The differential equation
in~\eqref{6.38} for $b_3^{(0)}$ becomes
\begin{gather*}
\paa_t \big(ta_3^{(0)}\big)= \beta_4 a_3^{(0)}+\beta_4 t^{\beta_4-1}.
\end{gather*}
This equation has no solution in $\C\{t\}$, a contradiction.

}

Assume $b_4^{(0)}=0$, $b_3^{(0)}\neq 0$.
The same arguments as for the case $b_4^{(0)}\neq 0$
give a contradiction if we replace
$\big(b_4^{(0)},a_4^{(0)},b_3^{(0)},a_3^{(0)}\big)$ by
$\big(b_3^{(0)},a_3^{(0)},b_2^{(0)},a_2^{(0)}\big)$.

Therefore $b_4^{(0)}=0$, $b_3^{(0)}=0$, $b_2^{(0)}\neq 0$.
Now $k_1=\deg_t b_2^{(0)}$.
A coordinate change in $t$ leads to
$b_2^{(0)}=t^{k_1}$. The differential equations~\eqref{6.38} gives
$a_4^{(0)}=a_3^{(0)}=0$, $a_2^{(0)}=-k_1t^{k_1-1}$.
We obtain the normal form in the 6th case in the list
in Theorem~\ref{t6.3}.

\medskip\noindent
{\it Generic type $($Log$)$.}
Now $b_2^{(0)}=b_3^{(0)}=b_4^{(0)}=0$.
The cases when all $a_i^{(0)}=0$, were consi\-de\-red above.
We assume now $a_j^{(0)}\neq 0$ for some $j\in\{2,3,4\}$.
The equations~\eqref{6.38} become a~homogeneous system
of linear equations with a nontrivial solution.
Therefore the determinant of the $3\times 3$-matrix in
\eqref{6.38} vanishes. It~is $1-4\big(b_3^{(1)}\big)^2-4b_2^{(1)}b_4^{(1)}$.
Its vanishing tells $\det \big(B^{(1)}-z\rho^{(1)}C_1\big)=\frac{-1}{4}$.
As $\tr\big(B^{(1)}-z\rho^{(1)}C_1\big)=0$, this matrix has the
eigenvalues $\pm\frac{1}{2}$.
Therefore a~linear base change gives
\begin{gather*}
b_2^{(1)}=b_4^{(1)}=0,\qquad b_3^{(1)}=-\frac{1}{2}.
\end{gather*}
Now the system of equations~\eqref{6.38} gives
$a_3^{(0)}=a_4^{(0)}=0$, whereas $a_2^{(0)}$ is
arbitrary in $\C\{t\}\setminus\{0\}$.
Denote $k_1:=1+\deg_t a_2^{(0)}\in\N$. A coordinate change
in $t$ leads to $a_2^{(0)}=k_1t^{k_1-1}$.
We obtain the normal form in the seventh case in the list
in Theorem~\ref{t6.3}.
\end{proof}

\subsection[Generically regular singular $(TE)$-structures
over $(\C,0)$ with logarithmic restriction over $t^0=0$
and not semisimple monodromy]
{Generically regular singular $\boldsymbol{(TE)}$-structures
over $\boldsymbol{(\C,0)}$ \\with logarithmic restriction over $\boldsymbol{t^0=0}$
and not semisimple monodromy}\label{c6.3}

The only 1-parameter unfoldings with trace free pole part
of logarithmic $(TE)$-structures over a point,
which are not covered by Theorem~\ref{t6.3}, have
generic type (Reg) or (Log) and not semisimple monodromy.
This follows from Theorem~\ref{t6.2}
and Theorem~\ref{t3.20}$(a)$.
These $(TE)$-structures are classified in Theorem~\ref{t6.7}. Some of them are in the 6th or 9th case
in Theorem~\ref{t6.3}, but most are not.

\begin{Theorem}\label{t6.7}
Consider a rank $2$ $(TE)$-structure $\big(H\to\C\times \big(M,t^0\big),\nnn\big)$
over $\big(M,t^0\big)=(\C,0)$ which is generically regular singular
$($so of generic type $($Reg$)$ or $($Log$))$, which has trace free
pole part, whose restriction over $t^0$ is logarithmic,
and whose monodromy has a $2\times 2$ Jordan block.
Associate to it the data in Definition~$\ref{t3.18}$:
$H':=H|_{\C\times (M,t^0)}$, $M^{\rm mon}$, $N^{\rm mon}$,
$\Eig(M^{\rm mon})=\{\lambda\}$, $H^\infty$,
$C^\alpha$ for $\alpha\in\C$ with ${\rm e}^{-2\pi {\rm i}\alpha}=\lambda$.

The leading exponents of the $(TE)$-structures over $t\neq 0$
come from Theorem~$\ref{4.15}(b)$ if the generic type is
$($Reg$)$ and from Theorem~$\ref{t4.20}(b)$ if the generic type
is $($Log$)$. In~both cases the leading exponents are independent
of $t$ and are still called $\alpha_1$ and $\alpha_2$.
Recall $\alpha_1-\alpha_2\in\N_0$.
The leading exponents of the logarithmic $(TE)$-structure
over $t^0=0$ from Theorem~$\ref{t4.20}(b)$ are now called
$\alpha_1^0$ and $\alpha_2^0$.
Recall $\alpha_1^0-\alpha_2^0\in\N_0$.

Precisely one of the three cases $(I)$, $(II)$ and $(III)$
in the following table holds:
\[
\def\arraystretch{1.5}
\begin{tabular}{l|l|l|l}
\hline
case $(I)$ & $\alpha_1^0=\alpha_1$ &$ \alpha_2^0=\alpha_2+1$ &thus $\alpha_1-\alpha_2\in\N$
\\
case $(II)$ & $\alpha_1^0=\alpha_1+1$ & $\alpha_2^0=\alpha_2$
\\
case $(III)$ & $\alpha_1^0=\alpha_1$ & $\alpha_2^0=\alpha_2$
\\
\hline
\end{tabular}
\]
Choose any section
$s_1\in C^{\alpha_1}\setminus \ker(\nnn_{z\paa_z}-\alpha_1\colon
C^{\alpha_1}\to C^{\alpha_1})$. It~determines uniquely
a section
$s_2\in \ker(\nnn_{z\paa_z}-\alpha_2\colon C^{\alpha_2}\to
C^{\alpha_2})\setminus\{0\}$
with
\begin{gather}\label{6.48}
(\nnn_{z\paa_z}-\alpha_1)(s_1)=z^{\alpha_1-\alpha_2}s_2.
\end{gather}
Then
\begin{gather}
\OO(H)_{(0,0)}=\C\{t,z\}(s_1+fs_2)\oplus \C\{t,z\}zs_2\ \
\text{for some } f\in t\C\{t\}\setminus\{0\}
\text{ in case $(I)$,}\label{6.49}
\\
\OO(H)_{(0,0)}= \C\{t,z\}(s_2+fs_1)\oplus \C\{t,z\}zs_1\ \
\text{for some }f\in t\C\{t\}\setminus\{0\}
\text{ in case $(II)$,}\!\!\label{6.50}
\\
\OO(H)_{(0,0)}=\C\{t,z\}s_1 \oplus \C\{t,z\}s_2\ \
\text{in case $(III)$.}\label{6.51}
\end{gather}
The function $f$ in the cases~\eqref{6.49} and~\eqref{6.50}
is independent of the choice of $s_1$, so it is an
invariant of the gauge equivalence class of the
$(TE)$-structure.
\end{Theorem}

Before the proof, some remarks are made.

\begin{Remarks}\label{t6.8}\quad
\begin{enumerate}\itemsep=0pt
\item[$(i)$] Equation~\eqref{6.48} gives
\begin{gather}\label{6.52}
\nnn_{z\paa_z}((s_1,s_2))=
(s_1,s_2)\begin{pmatrix}\alpha_1 & 0 \\ z^{\alpha_1-\alpha_2}
& \alpha_2\end{pmatrix}
=(s_1,s_2)\cdot z^{-1}B,
\end{gather}
with
\begin{gather*}
B=z\frac{\alpha_1+\alpha_2}{2}C_1
+z^{\alpha_1-\alpha_2+1}C_2+z\frac{\alpha_1-\alpha_2}{2}D.
\end{gather*}

\item[$(ii)$] The generic type is (Log) in the case (III).
This $(TE)$-structure is induced by its restriction
over $t^0=0$ via the projection $\varphi\colon \big(M,t^0\big)\to\big\{t^0\big\}$.
The matrices $A$ and $B$ for the basis $\uuuu{v}=(s_1,s_2)$
are $A=0$ and $B$ as in~\eqref{6.52}.

\item[$(iii)$] The generic type is (Reg) in the cases
(I) and (II). In~these cases
the $(TE)$-structure is induced by the special
cases of~\eqref{6.49} respectively~\eqref{6.50} with
$\www f=t$ via the map $\varphi=f\colon (\C,0)\to(\C,0)$.

\item[$(iv)$] The matrices $A$ and $B$ for the basis
$\uuuu{v}=(s_1+fs_2,zs_2)$
in~\eqref{6.49} ($\Rightarrow$ case (I),
$\Rightarrow \alpha_1-\alpha_2\in\N$) are
\begin{gather}\label{6.53}
A=\paa_tf\cdot C_2,\qquad
B=\begin{pmatrix}z\alpha_1 & 0 \\ (\alpha_2-\alpha_1)f
+z^{\alpha_1-\alpha_2} & z(\alpha_2+1)\end{pmatrix}\!.
\end{gather}
The matrices $A$ and $B$ for the basis
$\uuuu{v}=(s_2+fs_1,zs_1)$
in~\eqref{6.50} ($\Rightarrow$ case (II),
$\Rightarrow \alpha_1-\alpha_2\allowbreak\in\N_0$) are
\begin{gather}\label{6.54}
A=\paa_tf\cdot C_2,\qquad
B=\begin{pmatrix}z\alpha_2+z^{\alpha_1-\alpha_2+1}f &
z^{\alpha_1-\alpha_2+2} \\ (\alpha_1-\alpha_2)f
-z^{\alpha_1-\alpha_2}f^2 & z(\alpha_1+1)-z^{\alpha_1-\alpha_2+1}f\end{pmatrix}\!.
\end{gather}

\item[$(v)$] The invariant $k_1\in\N$ from~\eqref{6.1} is here
$k_1=\deg_t f\in\N$ in the case~\eqref{6.49}
and the case (\eqref{6.50} and $\alpha_1-\alpha_2\in\N$). It~is $k_1=2\deg_t f\in 2\N$ in the case
(\eqref{6.50} and $\alpha_1=\alpha_2$).
A~suitable coordinate change in $t$
reduces $f$ to $f=t^{k_1}$ respectively $f=t^{k_1/2}$.

\item[$(vi)$] The overlap of the $(TE)$-structures in Theorem~\ref{t6.3}
and in Theorem~\ref{t6.7} is as follows.

The case~\eqref{6.49} with $\alpha_1=\alpha_2+1$
and $f=-t^{k_1}$
is the 6th case in Theorem~\ref{t6.3}
with $\rho^{(1)}=\alpha_1$.

The case~\eqref{6.51} with $\alpha_1=\alpha_2$
is the 9th case in Theorem~\ref{t6.3}
with $\rho^{(1)}=\alpha_1$.
\end{enumerate}
\end{Remarks}

\begin{proof}[Proof of Theorem~\ref{t6.7}]
Choose any section
$s_1^0\in C^{\alpha_1^0}\setminus \ker\big(\nnn_{z\paa_z}-\alpha_1^0\colon C^{\alpha_1^0}
\to C^{\alpha_1^0}\big)\setminus\{0\}$. It~determines uniquely a section
$s_2^0\in \ker\big(\nnn_{z\paa_z}\colon C^{\alpha_2^0}\to C^{\alpha_2^0}\big)\setminus\{0\}$
with
\begin{gather*}
\big(\nnn_{z\paa_z}-\alpha_1^0\big)s_1^0
=z^{\alpha_1^0-\alpha_2^0}s_2^0.
\end{gather*}
Then
\begin{gather*}
\OO\big(H|_{\C\times\{t^0\}}\big)_0=\C\{z\}s_1^0|_{\C\times\{t^0\}}
\oplus \C\{z\}s_2^0|_{\C\times\{t^0\}}.
\end{gather*}

{\sloppy
Choose a $\C\{t,z\}$-basis $\uuuu{v}=(v_1,v_2)$ of
$\OO(H)_{(0,0)}$ which extends this $\C\{z\}$-basis of
$\OO\big(H|_{\C\times\{t^0\}}\big)_0$. It~has the shape
\begin{gather}
\uuuu{v}= \big(s_1^0,s_2^0\big)\cdot F\qquad \text{with}\quad
F=\begin{pmatrix} f_1 & f_2\\f_3&f_4\end{pmatrix}
\qquad\text{and}\nonumber
\\
f_1,f_4\in\C\{t,z\}\big[z^{-1}\big],\qquad
f_1(z,0)=f_4(z,0)=1,\qquad
f_2,f_3\in t\C\{t,z\}\big[z^{-1}\big].\label{6.57}
\end{gather}}\noindent
We write $f_j=\sum_{k\geq \deg_zf_j}f_j^{(k)}z^k$ with
$f_j^{(k)}=\sum_{l\geq 0}f_{j,l}^{(k)}t^l\in\C\{t\}$
and $f_{j,l}^{(k)}\in\C$.
Also we write
$\det F=\sum_{k\geq \deg_z\det F}(\det F)^{(k)}z^k$
with $(\det F)^{(k)}\in\C\{t\}$.

A meromorphic function $g\in\C\{t,z\}\big[z^{-1}\big]$ on a
neighborhood $U\subset \C\times M$ which is holomorphic
and not vanishing on $(U\setminus\{0\}\times M)\cup\{(0,0)\}$
is in $\C\{t,z\}^*$.
This and the facts that $\uuuu{v}$ and $\big(s_1^0,s_2^0\big)$ are bases
of $H|_{U\setminus\{0\}\times M}$ for some neighborhood
$U\subset\C\times M$ of $(0,0)$ and that
$\uuuu{v}|_{\C\times\{t^0\}}=\big(s_1^0,s_2^0\big)$ imply
\begin{gather*}
(\det F)\in\C\{t,z\}^*,\qquad \text{so especially}\quad
(\det F)^{(k)}=0\quad\text{ for}\quad k<0.\label{6.58}
\end{gather*}
Write $k_j:=\deg_z f_j\in\Z\cup\{\infty\}$
($\infty$ if $f_j=0$). Recall~\eqref{6.57}. It~implies
$f_1^{(0)},f_4^{(0)}\in\C\{t\}^*$ and
$f_1^{(k)},f_4^{(k)}\in t\C\{t\}$ for $k\neq 0$
and $f_2^{(k)},f_3^{(k)}\in t\C\{t\}$ for all $k$.
Especially $k_1\leq 0$ and $k_4\leq 0$.
We~distinguish four cases. Precisely one of them holds:
\begin{align*}
&\text{Case }\www{\rm(I)}\colon&&\hspace{-60mm}
0=k_1\leq k_2,\qquad\ 0>\min(k_3,k_4),
\\
&\text{Case }\www{\rm(II)}\colon&&\hspace{-60mm}
0=k_4\leq k_3,\qquad\ 0>\min(k_1,k_2),
\\
&\text{Case }\www{\rm(III)}\colon&&\hspace{-60mm}
0=k_1=k_4,\qquad\ 0\leq k_2,\quad 0\leq k_3,
\\
&\text{Case }\www{\rm(IV)}\colon&&\hspace{-60mm}
0>\min(k_1,k_2),\ \, \ 0>\min(k_3,k_4).
\end{align*}
We will show: Case $\www{\rm (I)}$ leads to~\eqref{6.49}
and case (I),
case $\www{\rm (II)}$ leads to~\eqref{6.50} and case (II),
case~$\www{\rm (III)}$ leads to~\eqref{6.51} and case~(III),
and case $\www{\rm (IV)}$ is impossible.

\medskip\noindent
{\it Case $\www{(III)}$}:
Then $F\in {\rm GL}_2(\C\{t,z\})$ and a base change leads to the
new basis $\uuuu{\www v}=\big(s_1^0,s_2^0\big)$. With
\begin{gather*}
(\alpha_1,\alpha_2,s_1,s_2)=\big(\alpha_1^0,\alpha_2^0,s_1^0,s_2^0\big),
\end{gather*}
this gives~\eqref{6.51} and case (III).

\medskip\noindent
{\it Case $\www{(I)}$}: Then $f_1\in\C\{t,z\}^*$, and
a base change leads to a new basis
$\uuuu{v}^{[1]}=\big(s_1^0,s_2^0\big)\cdot F^{[1]}$ with
\begin{gather*}
f_1^{[1]}=1,\qquad f_2^{[1]}=0,
\qquad f_4^{[1]}=\det F^{[1]}\in \C\{t,z\}^*,
\qquad f_3^{[1]}\in t\C\{t,z\}\big[z^{-1}\big].
\end{gather*}
As $k_4^{[1]}=0$, we have $k_3^{[1]}<0$.
A base change leads to a new basis
$\uuuu{v}^{[2]}=\big(s_1^0,s_2^0\big)\cdot F^{[2]}$ with
\begin{gather*}
F^{[2]}=C_1 + f_3^{[2]}C_2, \qquad\text{with}\quad
f_3^{[2]}\in tz^{-1}\C\{t\}\big[z^{-1}\big]\setminus\{0\}.
\end{gather*}
The covariant derivative
$z\nnn_{\paa_t}v_1^{[2]}=z\paa_t f_3^{[2]}\cdot v_2^{[2]}$
must be in $\OO(H)_0$. This shows
$f_3^{[2]}\in tz^{-1}\C\{t\}\allowbreak\setminus\{0\}$. With
{\samepage\begin{gather*}
(\alpha_1,\alpha_2,s_1,s_2,f)=\big(\alpha_1^0,\alpha_2^0-1,s_1^0,z^{-1}s_2^0,zf_3^{[2]}\big),
\end{gather*}
this gives~\eqref{6.49} and case (II).

}

\medskip\noindent
{\it Case $\www{(II)}$}: Then $f_4\in\C\{t,z\}^*$, and
a base change leads to a new basis
$\uuuu{v}^{[1]}=\big(s_1^0,s_2^0\big)\cdot F^{[1]}$ with
\begin{gather*}
f_4^{[1]}=1,\qquad f_3^{[1]}=0,
\qquad f_1^{[1]}=\det F^{[1]}\in \C\{t,z\}^*,
\qquad f_2^{[1]}\in t\C\{t,z\}\big[z^{-1}\big].
\end{gather*}
As $k_1^{[1]}=0$, we have $k_2^{[1]}<0$.
A base change leads to a new basis
$\uuuu{v}^{[2]}=\big(s_1^0,s_2^0\big)\cdot F^{[2]}$ with
\begin{gather*}
F^{[2]}=C_1 + f_2^{[2]}E, \qquad\text{with}\quad
f_2^{[2]}\in tz^{-1}\C\{t\}\big[z^{-1}\big]\setminus\{0\}.
\end{gather*}
The covariant derivative
$z\nnn_{\paa_t}v_2^{[2]}=z\paa_t f_2^{[2]}\cdot v_1^{[2]}$
must be in $\OO(H)_0$. This shows
$f_2^{[2]}\in tz^{-1}\C\{t\}\allowbreak\setminus\{0\}$. With
\begin{gather*}
(\alpha_1,\alpha_2,s_1,s_2)=\big(\alpha_1^0-1,\alpha_2^0,z^{-1}s_1^0,s_2^0\big),
\end{gather*}
this gives~\eqref{6.50} and almost case (II).
{}"{}Almost{}"{} because we still have to show
$\alpha_1-\alpha_2\in\N_0$.
This follows from the summand $-z^{\alpha_1-\alpha_2}f^2$ in
the left lower entry in the matrix $B$ in~\eqref{6.54}.

\medskip\noindent
{\it Case $\www{(IV)}$}: Exchange $v_1$ and $v_2$ if
$k_1>k_2$ or if $k_1=k_2$ and
$\deg_tf_1^{(k_1)}>\deg_tf_2^{(k_1)}$.
Keep the basis $\uuuu{v}$ if not. The new basis
$\uuuu{v}^{[1]}$ satisfies
$\min(k_1,k_2)=k_1^{[1]}\leq k_2^{[1]}$, and
in the case $k_1^{[1]}=k_2^{[1]}$ it satisfies
$\deg_t \big(f_1^{[1]}\big)^{(k_1^{[1]})}\leq \deg_t \big(f_2^{[1]}\big)^{(k_1^{[1]})}$.

By replacing $v_2^{[1]}$ by a suitable element in
$v_2^{[1]}+\C\{t,z\} v_1^{[1]}$ , we obtain a new basis
$\uuuu{v}^{[2]}$ either with $f_2^{[2]}=0$ or
with $k_1^{[2]}< k_2^{[2]}$ and
$\deg_t \big(f_1^{[2]}\big)^{(k_1^{[2]})}
>\deg_t \big(f_2^{[2]}\big)^{(k_2^{[2]})}$.

The case $f_2^{[2]}=0$ is impossible, as then
we would have
$f_1^{[2]}f_4^{[2]}=\det F^{[2]}\in\C\{t,z\}^*$, so
$f_1^{[2]}\in\C\{t,z\}^*$ and
$0= k_1^{[2]}$, but also $k_1^{[2]}=k_1^{[1]}=\min(k_1,k_2)<0$.
For the same reason, $f_3^{[2]}=0$ is impossible.

$f_4^{[2]}=0$ is impossible as then we would have
$-f_2^{[2]}f_3^{[2]}=\det F^{[2]}\in \C\{t,z\}^*$,
so $f_2^{[2]},f_3^{[2]}\in\C\{t,z\}^*$,
$0=k_2^{[2]}=k_3^{[2]}$ and $k_4^{[2]}=\infty$,
so $0=\min\big(k_3^{[2]},k_4^{[2]}\big)=\min(k_3,k_4)<0$,
a contradiction.

Write
\begin{gather*}
l_2:=\deg_t\big(f_2^{[2]}\big)^{(k_2^{[2]})}\in\N_0,\qquad
l_1:=\deg_t \big(f_1^{[2]}\big)^{(k_1^{[2]})}-l_2\in\N,
\\
l_3:=\deg_t\big(f_3^{[2]}\big)^{(k_3^{[2]})}\in\N_0,\qquad
l_4:=\deg_t \big(f_4^{[2]}\big)^{(k_4^{[2]})}\in\N_0.
\end{gather*}
Multiplying $v_1^{[2]}$ and $v_2^{[2]}$ by suitable units in
$\C\{t\}$, we obtain a basis $\uuuu{v}^{[3]}$ with
$k_j^{[3]}=k_j^{[2]}$ and
\begin{gather*}
\big(f_1^{[3]}\big)^{(k_1^{[3]})}=t^{l_1+l_2},\qquad
\big(f_2^{[3]}\big)^{(k_2^{[3]})}=t^{l_2},
\\
\big(f_3^{[3]}\big)^{(k_3^{[3]})}=t^{l_3}\cdot u_3,\qquad
\big(f_4^{[3]}\big)^{(k_4^{[3]})}=t^{l_4}\cdot u_4
\end{gather*}
for some units $u_3,u_4\in\C\{t\}^*$.
We still have $0>k_1^{[3]}<k_2^{[3]}$ and
$\min\big(k_3^{[3]},k_4^{[3]}\big)<0$. Consider
\begin{gather*}
z\nnn_{\paa_t}\big(v_1^{[3]}\big)=z\paa_t f_1^{[3]}\cdot s_1^0 +
z\paa_t f_3^{[3]}\cdot s_2^0 \in \OO(H)_{(0,0)}
=\C\{t,z\}v_1^{[3]} \oplus \C\{t,z\}v_2^{[3]}.
\end{gather*}
The leading nonvanishing monomial in $z\paa_t f_1^{[3]}$
is $z^{k_1^{[3]}+1}t^{l_1+l_2-1}$.
This implies $k_2^{[3]}=k_1^{[3]}+1\leq 0$.
Therefore $k_1^{[3]}+k_4^{[3]}<0$ or
$k_2^{[3]}+k_3^{[3]}<0$. Each part
$\big(\det F^{[3]}\big)^{(k)}$ for $k<0$ vanishes. This shows
\begin{gather*}
k_1^{[3]}+k_4^{[3]}=k_2^{[3]}+k_3^{[3]}<0, \qquad \text{so}\quad
k_4^{[3]}=k_3^{[3]}+1\leq 0,\quad k_3^{[3]}<0,
\\
0=\big(f_1^{[3]}\big)^{(k_1^{[3]})} \big(f_4^{[3]}\big)^{(k_3^{[3]}+1)}
-\big(f_2^{[3]}\big)^{(k_1^{[3]}+1)} \big(f_3^{[3]}\big)^{(k_3^{[3]})}
= t^{l_2}\bigl(t^{l_1+l_4}u_4 - t^{l_3}u_3\bigr),
\\
\text{so}\quad l_3=l_1+l_4, \quad u_3=u_4.
\end{gather*}
We can write
\begin{gather*}
\uuuu{v}^{[3]}= \big(s_1^0,s_2^0\big)\begin{pmatrix}
(t^{l_1+l_2}+zg_1)z^{k_1^{[3]}} & (t^{l_2}+zg_2)z^{k_1^{[3]}+1} \\
(t^{l_1+l_4}u_3 + zg_3)z^{k_3^{[3]}} & (t^{l_4}u_3 + zg_4)z^{k_3^{[3]}+1}
\end{pmatrix}
\end{gather*}
with some suitable $g_1,g_2,g_3,g_4\in\C\{t,z\}$.
This shows
\begin{gather}
 \OO(H)_{(0,0)}\cap\bigl( z^{k_1^{[3]}+2}\C\{t,z\}s_1^0 +
\C\{t,z\}\big[z^{-1}\big]s_2^0\bigr)\nonumber
\\ \qquad
{}= z^2\C\{t,z\}v_1^{[3]} + z\C\{t,z\}v_2^{[3]} +
\C\{t,z\}\big(zv_1^{[3]}-t^{l_1}v_2^{[3]}\big)\nonumber
\\ \qquad
{}\subset \OO(H)_{(0,0)}\cap \bigl(
\C\{t,z\}\big[z^{-1}\big]s_1^0 + z^{k_3^{[3]}+2}\C\{t,z\}s_2^0\bigr).\label{6.66}
\end{gather}
Now consider the element
\begin{gather*}
z\big(\nnn_{z\paa_z}-\big(\alpha_1^0+k_1^{[3]}\big)\big)\big(v_1^{[3]}\big)
\\ \qquad
{}= z^2\paa_z(zg_1)z^{k_1^{[3]}}s_1^0
+ \big(t^{l_1+l_2}+zg_1\big)z^{k_1^{[3]}+1+\alpha_1^0-\alpha_2^0}s_2^0 + z^2\paa_z(zg_3)z^{k_3^{[3]}}s_2^0
\\ \qquad \hphantom{=}
{}
+ \big(t^{l_1+l_4}u_3+zg_3\big)\big(k_3^{[3]}+\alpha_2^0
-\alpha_1^0-k_1^{[3]}\big)z^{k_3^{[3]}+1}s_2^0
\end{gather*}
of $\OO(H)_{(0,0)}$. It~is contained in the first line of~\eqref{6.66},
and therefore also in the third line of~\eqref{6.66}.
But this leads to a contradiction, when we compare the
coefficient of $s_2^0$. Here observe
\begin{gather*}
k_3^{[3]}+\alpha_2^0-\alpha_1^0-k_1^{[3]}
\left\{\!\!\begin{array}{l}<\\=\\>\end{array}\!\!\right\} 0
\iff
k_1^{[3]}+1+\alpha_1^0-\alpha_2^0
\left\{\!\!\begin{array}{l}<\\=\\>\end{array}\!\!\right\} k_3^{[3]}+1.
\end{gather*}
This contradiction shows that case $\www{\rm (IV)}$
is impossible.
\end{proof}

\section[Marked regular singular rank 2 $(TE)$-structures]{Marked regular singular rank 2 $\boldsymbol{(TE)}$-structures}\label{c7}

The regular singular rank $2$ $(TE)$-structures over points
were subject of Sections~\ref{c4.5} and~\ref{c4.6},
those over $(\C,0)$ were subject of Theorem~\ref{t6.3}
and Remark~\ref{t6.4}$(iv)$ and of Theorem~\ref{t6.7}
and~Remarks~\ref{t6.8}.

First we will consider in Remarks~\ref{t7.1}$(i){+}(ii)$
regular singular rank $2$ $(TE)$-structures over~$\P^1$,
which arise naturally
from Theorems~\ref{t4.17} and~\ref{t4.20}.
The $(TE)$-structures over the germs
$\big(\P^1,0\big)$ and $\big(\P^1,\infty\big)$ appeared already
in Remark~\ref{t6.4}$(iv)$ and in Theorem~\ref{t6.7}.

With the construction in Lemma~\ref{t3.10}$(d)$,
each of these $(TE)$-structures over $\P^1$
extends to a rank $2$ $(TE)$-structure
of generic type (Reg) or (Log)
over $\C\times\P^1$ with primitive Higgs field.
With Theorem~\ref{t3.14}, the base manifold
$\C\times\P^1$ obtains a canonical structure
as $F$-manifold with Euler field.
For each $t^0\in\C\times\C^*$, the $(TE)$-structure
over the germ $\big(\C\times\P^1,t^0\big)$ is a universal
unfolding of its restriction over $t^0$.
For each $t^0\in\C\times\{0,\infty\}$, the
$(TE)$-structure over the germ
$\big(\C\times \P^1,t^0\big)$ will reappear in Theorems~\ref{t8.1},~\ref{t8.5} and~\ref{t8.6}.
See Remarks~\ref{t7.2}$(i){+}(ii)$.

Then we will observe in Corollary~\ref{t7.3} that
any marked regular singular $(TE)$-structure
is a~{\it good} family of marked regular singular
$(TE)$-structures (over points) in the sense of Definition~\ref{t3.26}$(b)$.

In Theorem~\ref{t7.4} we will determine the moduli spaces
$M^{(H^{{\rm ref},\infty},M^{\rm ref}),{\rm reg}}$
for {\it marked} regular singular
rank $2$ $(TE)$-structures, which were subject of
Theorem~\ref{t3.29}. The parameter space~$\P^1$ of each
$(TE)$-structure over $\P^1$ in Remarks~\ref{t7.1}$(i){+}(ii)$
embeds into one of these moduli spaces, after the choice
of a marking. Also these embeddings will be described
in Theorem~\ref{t7.4}.

Because of Corollary~\ref{t7.3}, any marked
regular singular rank $2$ $(TE)$-structure
over a mani\-fold~$M$ is induced by a holomorphic map
$M\to M^{(H^{{\rm ref},\infty},M^{\rm ref}),{\rm reg}}$,
where $\big(H^{{\rm ref},\infty},M^{\rm ref}\big)$ is the reference
pair used in the marking of the $(TE)$-structure.
Remark~\ref{t7.5} says something about the
horizontal direction(s) in the moduli spaces.

\begin{Remarks}\label{t7.1}\qquad

\begin{enumerate}\itemsep=0pt
\item[$(i)$] Consider the manifold $M^{(3)}:=\P^1$ with coordinate
$t_2$ on $\C\subset\P^1$ and coordinate
$\www t_2:=t_2^{-1}$ on $\P^1\setminus\{0\}\subset\P^1$.

With the projection $M^{(3)}\to\{0\}$,
we pull back the flat bundle $H'\to\C^*$ in Theorem~\ref{t4.17} to a flat bundle ${{H^{(3)}}'}$ on
$\C^*\times M$. Recall the notations~\eqref{4.14}.

Now we read $\uuuu{v}:=(s_1+t_2s_2,zs_2)$
in~\eqref{4.57} and~\eqref{4.59}
as a basis of sections on ${{H^{(3)}}'}|_{\C^*\times\C}$,
and $\uuuu{\www v}:=\big(s_2+\www t_2s_1,zs_1\big)$
in~\eqref{4.58} and~\eqref{4.60}
as a basis of sections on ${{H^{(3)}}'}|_{\C^*\times(\P^1\setminus\{0\})}$.
One sees immediately
\begin{gather}\label{7.1}
z\nnn_{\paa_2}\uuuu{v}=\uuuu{v}\, C_2,\qquad
z\nnn_{\www \paa_2}\uuuu{\www v}=\uuuu{\www v}\, C_2.
\end{gather}
and again~\eqref{4.57} resp.~\eqref{4.59}
and~\eqref{4.58} resp.~\eqref{4.60}.
Therefore $\uuuu{v}$ and $\uuuu{\www v}$ are in any case bases
of a $(TE)$-structure $\big(H^{(3)}\to\C\times M^{(3)},\nnn^{(3)}\big)$
on $\C\times \C\subset\C\times M^{(3)}$ respectively
$\C\times \big(\P^1\setminus\{0\}\big)\subset\C\times M^{(3)}$.
The restricted $(TE)$-structures over $t_2\in\C^*$ are
those in~Theo\-rem~\ref{t4.17}.
They are regular singular, but not logarithmic.
Their leadings exponents $\alpha_1$ and~$\alpha_2$ are
independent of $t_2\in\C^*$.
The $(TE)$-structures over $t_2=0$ and over $\www t_2=0$
(so $t_2=\infty$) are logarithmic except for the case
($N^{\rm mon}\neq 0$ and~$\alpha_1=\alpha_2$), in which case
the one over $t_2=0$ is regular singular, but not logarithmic.
Their leading exponents are
called~$\alpha_1^0$ and~$\alpha_2^0$ and
$\alpha_1^\infty$ and~$\alpha_2^\infty$. Then
\[
\def\arraystretch{1.3}
\begin{tabular}{l|l}
\hline
\quad\ over $0$ & \quad\ over $\infty$
\\ \hline
$\alpha_1^0=\alpha_1$ & $\alpha_1^\infty=\alpha_1+1$
\\
$\alpha_2^0=\alpha_2+1$ & $\alpha_2^\infty=\alpha_2$
\\
\hline
\end{tabular}
\]
except that in the case
($N^{\rm mon}\neq 0$ and~$\alpha_1=\alpha_2$)
we have $\alpha_1^0=\alpha_1$, $\alpha_2^0=\alpha_2$.
For use in Theorem~\ref{t7.3}, we write the base space
for the $(TE)$-structure over $\P^1$ with leading exponents
$\alpha_1$ and~$\alpha_2$ as
$M^{(3),0,\alpha_1,\alpha_2}\cong\P^1$ in the case
$N^{\rm mon}=0$ and as $M^{(3),\neq 0,\alpha_1,\alpha_2}\cong\P^1$
in the case $N^{\rm mon}\neq 0$.

\item[$(ii)$] We extend the case $N^{\rm mon}=0$ from Theorem~\ref{t4.17}$(a)$ to the case $\alpha_1=\alpha_2$.
\eqref{4.57} and~\eqref{4.58} still hold,
but now the restricted $(TE)$-structures over
points in $M^{(3)}=\P^1$ are all logarithmic,
though the $(TE)$-structure over $M^{(3)}$ is not
logarithmic, but only regular singular.
\eqref{7.1} still holds. In~this case, the leading exponents are constant
and are $\alpha_1$ and $\alpha_1+1$
(so, not $\alpha_1$ and $ \alpha_2=\alpha_1$).
Similarly to $(i)$, the base space is called
$M^{(3),0,\alpha_1,{\rm log}}\cong\P^1$.
\end{enumerate}
\end{Remarks}

\begin{Remarks}\label{t7.2}\quad

\begin{enumerate}\itemsep=0pt
\item[$(i)$] The construction in Lemma~\ref{t3.10}$(d)$
extends a $(TE)$-structure
$\big(H^{(3)}\to\C\times M^{(3)},\nnn^{(3)}\big)$ in~Remark~\ref{t7.1}$(i)$ or $(ii)$
with $M^{(3)}=\P^1$ to a $(TE)$-structure
$\big(H^{(4)}\to\C\times M^{(4)},\nnn^{(4)}\big)$ with
$M^{(4)}=\C\times M^{(3)}=\C\times \P^1$, via
$\big(\OO\big(H^{(4)}\big),\nnn^{(4)}\big)=\big(\varphi^{(4)}\big)^*
\big(\OO\big(H^{(3)}\big),\nnn^{(3)}\big)\otimes \EE^{t_1/z}$,
where $t_1$ is the coordinate on the first factor $\C$
in $\C\times\P^1$, and where $\varphi^{(4)}\colon M^{(4)}\to
M^{(3)}$, $(t_1,t_2)\mapsto t_2,$ is the projection.
Define $\uuuu{v^{(4)}}:=\big(\varphi^{(4)}\big)^*
(\uuuu{v}\text{ in Remark~\ref{t7.1}$(i)$ or $(ii)$})$.

Then the matrices $A_2$ and $B$ with
$z\nnn_{\paa_i}\uuuu{v}^{(4)}=\uuuu{v}^{(4)}A_i$ and
$z^2\nnn_{\paa_z}\uuuu{v}^{(4)}=\uuuu{v}^{(4)}B$ are unchanged,
and $A_1=C_1$, so
\begin{gather*}
A_1=C_1,\quad A_2=C_2\quad \text{(as in~\eqref{7.1})},\qquad
B\quad\text{is as in~\eqref{4.57} or~\eqref{4.59}}.
\end{gather*}
The Higgs field is everywhere on $M^{(4)}$ primitive.
By Theorem~\ref{t3.14},
$M^{(4)}=\C\times\P^1$ is an~$F$-manifold with Euler field.
The unit field is $\paa_1$, the multiplication is given
by $\paa_2\circ\paa_2=0$ and $\www\paa_2\circ\www\paa_2=0$.
So, each germ $\big(M^{(4)},t^0\big)$ is the germ $\NN_2$.
The Euler field is
\begin{gather*}
\begin{split}
&E= t_1\paa_1 + (\alpha_1-\alpha_2)t_2\paa_2
=t_1\paa_1 + (\alpha_2-\alpha_1)\www t_2\www\paa_2
\\
&\qquad\begin{cases}
\text{in the case~\eqref{4.57} and~\eqref{4.58}
and in $(ii)$ above,}
\\
\text{and in the case~\eqref{4.59} and~\eqref{4.60}
 with }\alpha_1-\alpha_2\in\N,\end{cases}
 \\
&E= t_1\paa_1 -\paa_2
=t_1\paa_1 + \www t_2^2\www\paa_2\qquad
\text{in the case~\eqref{4.59} and~\eqref{4.60}
with }\alpha_1=\alpha_2.
\end{split}
\end{gather*}

\item[$(ii)$] For $t^{(4)}\in \C\times\C^*\subset M^{(4)}$,
the $(TE)$-structure
$\big(H^{(4)}\to \C\times \big(M^{(4)},t^{(4)}\big),\nnn\big)$
is a universal unfolding of the one over $t^0$,
because that one is of type (Reg) and because the
Higgs field is primitive. See Corollary~\ref{t5.1}.

\item[$(iii)$] Let $\big(H\to\C\times \big(M,t^0\big),\nnn\big)$ be a regular
singular unfolding of a regular singular, but not
logarithmic rank $2$ $(TE)$-structure over $t^0$.
Because of part $(ii)$, it is induced by the
$(TE)$-structure $\big(H^{(4)},\nnn^{(4)}\big)$ via a map
$\big(M,t^0\big)\to \big(M^{(4)},t^{(4)}\big)$ for some
$t^{(4)}\in\{0\}\times \C^*$.
Because it is regular singular, the image of the map
is in $\{0\}\times \C^*\subset \{0\}\times M^{(3)}$.
As there the leading exponents are constant,
they are also constant on the unfolding $(H,\nnn)$.
\end{enumerate}
\end{Remarks}

\begin{Theorem}\label{t7.3}
Any marked regular singular rank $2$ $(TE)$-structure
$($see Definition~$\ref{t3.15}(b)$, espe\-cially,
$M$ is simply connected$)$
is a {\it good} family of marked regular singular
$(TE)$-structures $($over points$)$ in the sense of Definition~$\ref{t3.26}(b)$.
\end{Theorem}

\begin{proof} Let $((H\to\C\times M,\nnn),\psi)$
be a regular singular rank $2$
$(TE)$-structure with a marking~$\psi$,
i.e., an isomorphism $\psi$ from $(H^\infty,M^{\rm mon})$
to a reference pair $\big(H^{{\rm ref},\infty},M^{\rm ref}\big)$.
We have to show the conditions~\eqref{3.39} and~\eqref{3.40}
for a good family of marked regular singular $(TE)$-structures.

By definition of a marking, $M$ is simply connected,
so especially it is connected. The subset
\begin{gather*}
M^{\rm [log]}:=\{t\in M\,|\, \UU|_t=0\}
=\{t\in M\,|\, \text{the }(TE)\text{-structure over }
t\text{ is logarithmic}\}
\end{gather*}
is a priori a subvariety (in fact, it is either $\varnothing$
or a hypersurface or equal to $M$).

First consider the case $M^{\rm [log]}=M$.
Choose any point $t^0\in M$ and any disk $\Delta\subset M$
through $t^0$. The restriction of the $(TE)$-structure
over the germ $\big(\Delta,t^0\big)$ is in the case $N^{\rm mon}=0$
isomorphic to one in the 7th or 8th or 9th case in Theorem~\ref{t6.3}. In~the case $N^{\rm mon}\neq 0$, it is
isomorphic to one in case (III) in Theorem~\ref{t6.7}. In~either case the leading exponents are constant on $\Delta$,
because of table~\eqref{6.27}
in Remark~\ref{t6.4}$(iv)$ and because of the definition
of case (III) in Theorem~\ref{t6.7}.
Therefore they are constant on $M$.
We call them $\alpha_1^{\rm gen}$ and $\alpha_2^{\rm gen}$.

Now consider the case $M^{\rm [log]}\subsetneqq M$.
For each $t^0\in M\setminus M^{\rm [log]}$, the $(TE)$-structure
over the germ $\big(M,t^0\big)$ has constant leading
exponents because of Remark~\ref{t7.2}$(iii)$.
Therefore the leading exponents are constant on
$M\setminus M^{\rm [log]}$. We call these generic leading exponents
$\alpha_1^{\rm gen}$ and $\alpha_2^{\rm gen}$.

For $t^0\in M^{\rm [log]}$ choose a generic small disk
$\Delta\subset M$ through $t^0$. Then
$\Delta\setminus \big\{t^0\big\}\subset M\setminus M^{\rm [log]}$. The restriction of the
$(TE)$-structure over the germ $\big(\Delta,t^0\big)$ is
in the case $N^{\rm mon}=0$ isomorphic to one in the 5th case
in Theorem~\ref{t6.3}. In~the case $N^{\rm mon}\neq 0$,
it is isomorphic to one in case~(I) or case (II)
in Theorem~\ref{t6.7}. In~either case, the leading exponents
$\big(\alpha_1\big(t^0\big),\alpha_2\big(t^0\big)\big)$ of the $(TE)$-structure
over $t^0$ are either $\big(\alpha_1^{\rm gen}+1,\alpha_2^{\rm gen}\big)$
or $\big(\alpha_1^{\rm gen},\alpha_2^{\rm gen}+1\big)$,
because of table~\eqref{6.27} in~Remark~\ref{t6.4}$(iv)$ and
because of the definition of the cases~(I) and~(II)
in Theorem~\ref{t6.7}.

Remark~\ref{t6.4}$(iv)$ and Theorem~\ref{t6.7}
provide generators of $\OO\big(H|_{\C\times(\Delta,t^0)}\big)_{(0,t^0)}$
which are certain linear combinations of elementary sections.
The shape of these generators and the almost constancy
of the leading exponents imply the two conditions,
\begin{gather*}
\OO\big(H|_{\C\times\{t\}}\big)_{(0,t)}\supset V^r\quad
\text{for any}\ t\in M,
\ \text{where} \
r:=\max\big(\Ree\big(\alpha_1^{\rm gen}\big)+1,\,\Ree\big(\alpha_2^{\rm gen}\big)+1\big),\\
\dim_\C\OO\big(H|_{\C\times\{t\}}\big)_{(0,t)}/V^r \quad
\text{is independent of $t\in M$}, 
\end{gather*}
which are the conditions~\eqref{3.39} and~\eqref{3.40}
for a good family of marked regular singular
$(TE)$-structures.
\end{proof}

The following theorem describes the moduli space
$M^{(H^{{\rm ref},\infty},M^{\rm ref}),{\rm reg}}$ from Theorem~\ref{t3.29}
for the marked regular singular rank $2$ $(TE)$-structures as infinite
unions of curves isomorphic to~$\P^1$ such that the families of
$(TE)$-structures over these curves are the $(TE)$-structures
in Remarks~\ref{t7.1}$(i){+}(ii)$.
Part $(a)$ treats the cases with $N^{\rm mon}=0$,
part $(b)$ treats the cases with $N^{\rm mon}\neq 0$.
Recall the definitions of $M^{(3),0,\alpha_1,\alpha_2}$,
$M^{(3),\neq 0,\alpha_1,\alpha_2}$ and
$M^{(3),0,\alpha_1,{\rm log}}$ in Remarks~\ref{t7.1}$(i)$ and~$(ii)$.

\begin{Theorem}\label{t7.4}
Let $\big(H^{{\rm ref},\infty},M^{\rm ref}\big)$ be a reference pair
with $\dim H^{{\rm ref},\infty}=2$.
Let $\Eig\big(M^{\rm ref}\big)=\{\lambda_1,\lambda_2\}$
be the set of eigenvalues of $M^{\rm ref}$.
Let $\beta_1,\beta_2\in\C$ be the unique numbers
with ${\rm e}^{-2\pi {\rm i}\beta_j}=\lambda_j$ and
$-1<\Ree\beta_j\leq 0$.

\begin{enumerate}\itemsep=0pt
\item[$(a)$] The case $N^{\rm mon}=0$.

\begin{enumerate}\itemsep=0pt
\item[$(i)$] The cases with $\lambda_1\neq\lambda_2$. Then
\begin{gather*}
M^{(H^{{\rm ref},\infty},M^{\rm ref}),{\rm reg}}=
\dot{\bigcup\limits_{l_1\in\Z}}\bigg(\bigcup_{l_2\in\Z}M^{(3),0,\beta_1+l_1+l_2,\beta_2-l_2}\bigg).
\end{gather*}
Its topological components are the unions in brackets, so
$\bigcup_{l_2\in\Z}M^{(3),0,\beta_1+l_1+l_2,\beta_2-l_2}$.
Each component is a chain of $\P^1$'s,
the point $\infty$ of $M^{(3),0,\alpha_1,\alpha_2}$ is
identified with the point $0$ of
$M^{(3),0,\alpha_1+1,\alpha_2-1}$.

\setlength{\unitlength}{1mm}

\begin{figure}[h]
\centering
\begin{picture}(120,41)

\multiput(10,15)(30,0){3}{\bezier{300}(0,0),(20,12),(40,0)}
\put(0,15){\bezier{20}(0,6),(10,6),(20,0)}
\put(100,15){\bezier{20}(0,0),(10,6),(20,6)}

\multiput(15,17.5)(30,0){4}{\circle*{2}}

\put(15,35){\makebox(0,0){\small$\begin{pmatrix}\alpha_1-1\\ \alpha_2+2\end{pmatrix}$}}
\put(30,30){\makebox(0,0){\small$\begin{pmatrix}\alpha_1-1\\ \alpha_2+1\end{pmatrix}$}}
\put(45,35){\makebox(0,0){\small$\begin{pmatrix}\alpha_1\\ \alpha_2+1\end{pmatrix}$}}
\put(60,30){\makebox(0,0){\small$\begin{pmatrix}\alpha_1\\ \alpha_2\end{pmatrix}$}}
\put(75,35){\makebox(0,0){\small$\begin{pmatrix}\alpha_1+1\\ \alpha_2\end{pmatrix}$}}
\put(90,30){\makebox(0,0){\small$\begin{pmatrix}\alpha_1+1\\ \alpha_2-1\end{pmatrix}$}}
\put(105,35){\makebox(0,0){\small$\begin{pmatrix}\alpha_1+2\\ \alpha_2-1\end{pmatrix}$}}

\multiput(15,7)(30,0){4}{\makebox(0,0){\small(Log)}}
\multiput(30,13)(30,0){3}{\makebox(0,0){\small(Reg)}}

\thinlines
\multiput(15,10)(30,0){4}{\vector(0,1){5}}
\multiput(15,30)(30,0){4}{\vector(0,-1){10}}
\end{picture}
\caption{One topological component in part $(a)$ $(i)$.}\label{figure1}
\end{figure}

\item[$(ii)$] The cases with $\lambda_1=\lambda_2$ $($so
$\beta_1=\beta_2)$. Then
\begin{gather}
M^{(H^{{\rm ref},\infty},M^{\rm ref}),{\rm reg}}
=\dot{\bigcup\limits_{l_1\in\Z}}\bigg(\bigcup_{l_2\in\N}
\F_2^{\beta_1+l_1+l_2,\beta_1+l_1-l_2}\bigg)\nonumber
\\ \hphantom{M^{(H^{{\rm ref},\infty},M^{\rm ref}),{\rm reg}}=}
\cup \dot{\bigcup\limits_{l_1\in\Z}}
\bigg(\www\F_2^{\beta_1+l_1+1,\beta_1+l_1}\cup
\bigcup_{l_2\in\N}
\F_2^{\beta_1+l_1+1+l_2,\beta_1+l_1-l_2}\bigg).
\label{7.10}
\end{gather}
Here $\F_2^{\alpha_1,\alpha_2}$ is for all possible
$\alpha_1$, $\alpha_2$ the Hirzebruch surface $\F_2$,
and $\www\F_2^{\alpha_1,\alpha_1-1}$ is the surface~$\www\F_2$, which is obtained from $\F_2$
by blowing down the unique $(-2)$-curve in $\F_2$.
The unions in brackets are the topological components.
They are chains of Hirzebruch surfaces.
A $(+2)$-curve of $\F_2^{\alpha_1,\alpha_2}$
is identified with the $(-2)$-curve of
$\F_2^{\alpha_1+1,\alpha_2-1}$
$($and a $(+2)$-curve of $\www\F_2^{\alpha_1,\alpha_1-1}$
is identified with the $(-2)$ curve in
$\F_2^{\alpha_1+1,\alpha_1-2})$. The $(TE)$-structures over
the points in the $(-2)$-curves are logarithmic,
and also the $(TE)$-structure over the singular point of
$\www \F_2^{\alpha_1,\alpha_1-1}$ is logarithmic.
The $(TE)$-structures over all other points of
$\F_2^{\alpha_1,\alpha_2}$ and $\www\F_2^{\alpha_1,\alpha_2}$
are regular singular, but not logarithmic, and have
leading exponents $\alpha_1$ and~$\alpha_2$.
For each $\F_2^{\alpha_1,\alpha_2}$, and also for
$\www\F_2^{\alpha_1,\alpha_2}$ after blowing up the singular
point to a $(-2)$-curve, the fibers of it as
a $\P^1$-fiber bundle over $\P^1$ are isomorphic
to $M^{(3),0,\alpha_1,\alpha_2}$.
The $(-2)$-curve in $\F_2^{\alpha_1,\alpha_1-2}$
$($the $\F_2$ with $l_2=1$ in each topological component in the first
line of~\eqref{7.10}$)$ is isomorphic to
$M^{(3),0,\alpha_1-1,{\rm log}}$, and the $(TE)$-structures
over its points are logarithmic with leading exponents
$\alpha_1$, $\alpha_1-1$.
\end{enumerate}

\setlength{\unitlength}{1mm}

\begin{figure}[h]
\centering
\begin{picture}(120,79) 

\multiput(10,35)(30,0){3}{\bezier{300}(0,0),(20,12),(40,0)}
\multiput(10,45)(30,0){3}{\bezier{300}(0,0),(20,12),(40,0)}
\multiput(10,55)(30,0){3}{\bezier{300}(0,0),(20,12),(40,0)}
\put(100,35){\bezier{20}(0,0),(10,6),(20,6)}
\put(100,45){\bezier{20}(0,0),(10,6),(20,6)}
\put(100,55){\bezier{20}(0,0),(10,6),(20,6)}

\multiput(15,37.5)(30,0){4}{\circle*{2}}
\multiput(15,47.5)(30,0){4}{\circle*{2}}
\multiput(15,57.5)(30,0){4}{\circle*{2}}

\multiput(15,32.5)(30,0){4}{\line(0,1){30}}

\put(15,75){\makebox(0,0){\small$\begin{pmatrix}\alpha_1\\ \alpha_1-1\end{pmatrix}$}}
\put(30,70){\makebox(0,0){\small$\begin{pmatrix}\alpha_1\\ \alpha_1-2\end{pmatrix}$}}
\put(45,75){\makebox(0,0){\small$\begin{pmatrix}\alpha_1+1\\ \alpha_1-2\end{pmatrix}$}}
\put(60,70){\makebox(0,0){\small$\begin{pmatrix}\alpha_1+1\\ \alpha_1-3\end{pmatrix}$}}
\put(75,75){\makebox(0,0){\small$\begin{pmatrix}\alpha_1+2\\ \alpha_1-3\end{pmatrix}$}}
\put(90,70){\makebox(0,0){\small$\begin{pmatrix}\alpha_1+2\\ \alpha_1-4\end{pmatrix}$}}
\put(105,75){\makebox(0,0){\small$\begin{pmatrix}\alpha_1+3\\ \alpha_1-4\end{pmatrix}$}}

\multiput(45,20)(30,0){3}{\makebox(0,0){\small(Log)}}
\multiput(30,33)(30,0){3}{\makebox(0,0){\small(Reg)}}
\put(10,20){\makebox{$M^{(3),0,\alpha_1-1,{\rm log}}$}}
\put(15,17){\makebox(0,0){(Log)}}

\thinlines
\multiput(15,25)(30,0){4}{\vector(0,1){5}}
\multiput(15,70)(30,0){4}{\vector(0,-1){5}}

\multiput(10,15)(30,0){3}{\bezier{50}(0,0)(0,-5)(20,-5)}
\multiput(30,15)(30,0){3}{\bezier{50}(0,-5)(20,-5)(20,0)}

\put(25,5){\makebox{$\F_2^{\alpha_1,\alpha_1-2}$}}
\put(55,5){\makebox{$\F_2^{\alpha_1+1,\alpha_1-3}$}}
\put(85,5){\makebox{$\F_2^{\alpha_1+2,\alpha_1-4}$}}
\end{picture}\vspace{-2ex}
\caption{One topological component in part $(a)$ $(ii)$.}\label{figure2}
\end{figure}
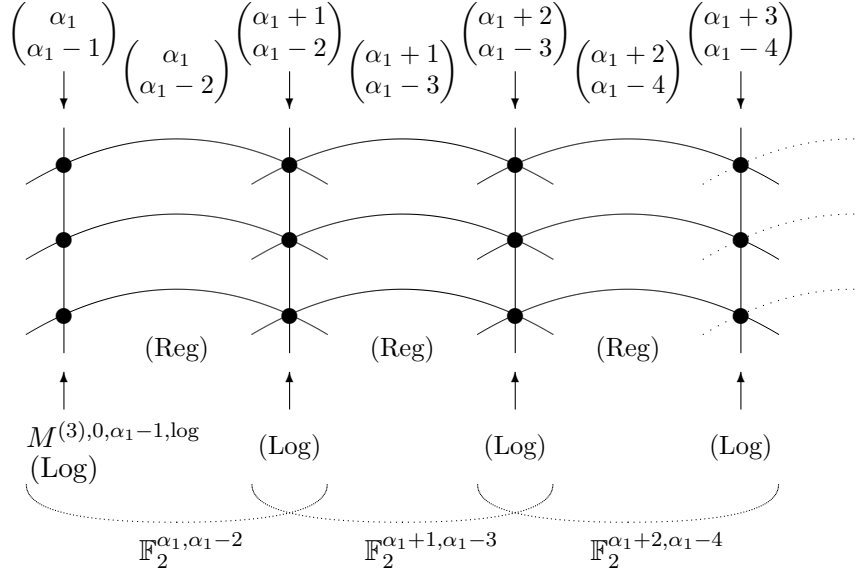

\begin{figure}[h]
\centering
\begin{picture}(120,82) 

\multiput(40,35)(30,0){2}{\bezier{300}(0,0),(20,12),(40,0)}
\multiput(40,45)(30,0){2}{\bezier{300}(0,0),(20,12),(40,0)}
\multiput(40,55)(30,0){2}{\bezier{300}(0,0),(20,12),(40,0)}

\put(10,41){\bezier{300}(0,0),(20,30),(40,13)}
\put(10,45){\bezier{300}(0,0),(20,12),(40,0)}
\put(10,48){\bezier{300}(0,0),(20,0),(40,-13.5)}

\put(100,35){\bezier{20}(0,0),(10,6),(20,6)}
\put(100,45){\bezier{20}(0,0),(10,6),(20,6)}
\put(100,55){\bezier{20}(0,0),(10,6),(20,6)}

\multiput(45,37.5)(30,0){3}{\circle*{2}}
\multiput(15,47.5)(30,0){4}{\circle*{2}}
\multiput(45,57.5)(30,0){3}{\circle*{2}}

\multiput(45,32.5)(30,0){3}{\line(0,1){30}}

\put(15,75){\makebox(0,0){\small$\begin{pmatrix}\alpha_1\\ \alpha_1\end{pmatrix}$}}
\put(30,70){\makebox(0,0){\small$\begin{pmatrix}\alpha_1\\ \alpha_1-1\end{pmatrix}$}}
\put(45,75){\makebox(0,0){\small$\begin{pmatrix}\alpha_1+1\\ \alpha_1-1\end{pmatrix}$}}
\put(60,70){\makebox(0,0){\small$\begin{pmatrix}\alpha_1+1\\ \alpha_1-2\end{pmatrix}$}}
\put(75,75){\makebox(0,0){\small$\begin{pmatrix}\alpha_1+2\\ \alpha_1-2\end{pmatrix}$}}
\put(90,70){\makebox(0,0){\small$\begin{pmatrix}\alpha_1+2\\ \alpha_1-3\end{pmatrix}$}}
\put(105,75){\makebox(0,0){\small$\begin{pmatrix}\alpha_1+3\\ \alpha_1-3\end{pmatrix}$}}

\multiput(15,20)(30,0){4}{\makebox(0,0){\small(Log)}}
\multiput(30,33)(30,0){3}{\makebox(0,0){\small(Reg)}}

\thinlines
\multiput(45,25)(30,0){3}{\vector(0,1){5}}
\multiput(45,70)(30,0){3}{\vector(0,-1){5}}
\put(15,25){\vector(0,1){15}}
\put(15,70){\vector(0,-1){15}}

\multiput(10,15)(30,0){3}{\bezier{50}(0,0)(0,-5)(20,-5)}
\multiput(30,15)(30,0){3}{\bezier{50}(0,-5)(20,-5)(20,0)}

\put(25,5){\makebox{$\www\F_2^{\alpha_1,\alpha_1-1}$}}
\put(55,5){\makebox{$\F_2^{\alpha_1+1,\alpha_1-2}$}}
\put(85,5){\makebox{$\F_2^{\alpha_1+2,\alpha_1-3}$}}
\end{picture}\vspace{-2ex}
\caption{Another topological component in part $(a)$ $(ii)$.}\label{figure3}
\end{figure}
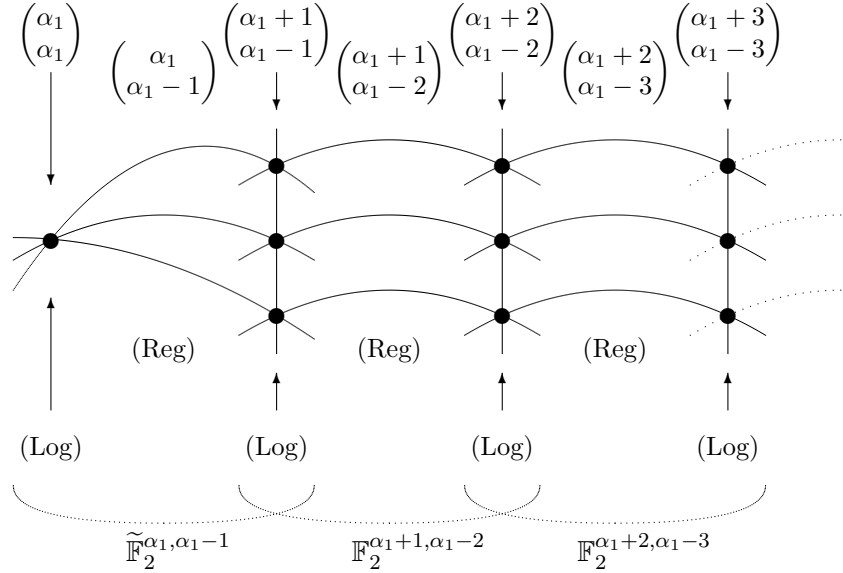

\item[$(b)$] The cases with $N^{\rm mon}\neq 0$ $($and thus
$\lambda_1=\lambda_2$, $\beta_1=\beta_2)$. Then
\begin{gather*}
M^{(H^{{\rm ref},\infty},M^{\rm ref}),{\rm reg}}=
\dot{\bigcup\limits_{l_1\in\Z}}\bigg(\bigcup_{l_2\in\N_0}
M^{(3),\neq 0,\beta_1+l_1+l_2,\beta_1+l_1-l_2}\bigg) \nonumber
\\ \hphantom{M^{(H^{{\rm ref},\infty},M^{\rm ref}),{\rm reg}}=}
\cup \dot{\bigcup\limits_{l_1\in\Z}}\bigg(\bigcup_{l_2\in\N_0}
M^{(3),\neq 0, \beta_1+l_1+1+l_2,\beta_1+l_1-l_2}\bigg).
\end{gather*}
Its topological components are the unions in brackets.
Each component is a chain of $\P^1$'s,
the point $\infty$ of $M^{(3),\neq0,\alpha_1,\alpha_2}$ is
identified with the point $0$ of
$M^{(3),\neq0,\alpha_1+1,\alpha_2-1}$.
\end{enumerate}

\setlength{\unitlength}{1mm}

\begin{figure}[h]
\centering
\begin{picture}(120,35)

\multiput(10,15)(30,0){3}{\bezier{300}(0,0),(20,12),(40,0)}
\put(100,15){\bezier{20}(0,0),(10,6),(20,6)}

\multiput(15,17.5)(30,0){4}{\circle*{2}}

\put(22,30){\makebox(0,0){\small$\begin{pmatrix}\alpha_1\\ \alpha_1\end{pmatrix}$}}
\put(45,35){\makebox(0,0){\small$\begin{pmatrix}\alpha_1+1\\ \alpha_1\end{pmatrix}$}}
\put(60,30){\makebox(0,0){\small$\begin{pmatrix}\alpha_1+1\\ \alpha_1-1\end{pmatrix}$}}
\put(75,35){\makebox(0,0){\small$\begin{pmatrix}\alpha_1+2\\ \alpha_1-1\end{pmatrix}$}}
\put(90,30){\makebox(0,0){\small$\begin{pmatrix}\alpha_1+2\\ \alpha_1-2\end{pmatrix}$}}
\put(105,35){\makebox(0,0){\small$\begin{pmatrix}\alpha_1+3\\ \alpha_1-2\end{pmatrix}$}}

\multiput(45,7)(30,0){3}{\makebox(0,0){\small(Log)}}
\put(22,13){\makebox(0,0){\small(Reg)}}
\multiput(60,13)(30,0){2}{\makebox(0,0){\small(Reg)}}

\thinlines
\multiput(45,10)(30,0){3}{\vector(0,1){5}}
\multiput(45,30)(30,0){3}{\vector(0,-1){10}}
\end{picture}\vspace{-2ex}
\caption{One topological component in part $(b)$.}\label{figure4}
\end{figure}
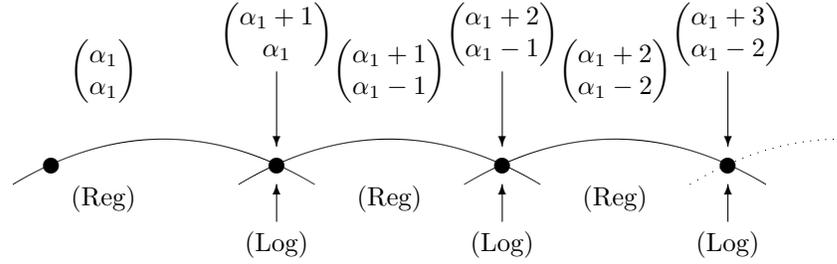

\begin{figure}[h]
\centering
\begin{picture}(120,35)

\multiput(10,15)(30,0){3}{\bezier{300}(0,0),(20,12),(40,0)}
\put(100,15){\bezier{20}(0,0),(10,6),(20,6)}

\multiput(15,17.5)(30,0){4}{\circle*{2}}

\put(15,35){\makebox(0,0){\small$\begin{pmatrix}\alpha_1\\ \alpha_1\end{pmatrix}$}}
\put(30,30){\makebox(0,0){\small$\begin{pmatrix}\alpha_1\\ \alpha_1-1\end{pmatrix}$}}
\put(45,35){\makebox(0,0){\small$\begin{pmatrix}\alpha_1+1\\ \alpha_1-1\end{pmatrix}$}}
\put(60,30){\makebox(0,0){\small$\begin{pmatrix}\alpha_1+1\\ \alpha_1-2\end{pmatrix}$}}
\put(75,35){\makebox(0,0){\small$\begin{pmatrix}\alpha_1+2\\ \alpha_1-2\end{pmatrix}$}}
\put(90,30){\makebox(0,0){\small$\begin{pmatrix}\alpha_1+2\\ \alpha_1-3\end{pmatrix}$}}
\put(105,35){\makebox(0,0){\small$\begin{pmatrix}\alpha_1+3\\ \alpha_1-3\end{pmatrix}$}}

\multiput(15,7)(30,0){4}{\makebox(0,0){\small(Log)}}
\multiput(30,13)(30,0){3}{\makebox(0,0){\small(Reg)}}

\thinlines
\multiput(15,10)(30,0){4}{\vector(0,1){5}}
\multiput(15,30)(30,0){4}{\vector(0,-1){10}}
\end{picture}\vspace{-2ex}
\caption{Another topological component in part $(b)$.}\label{figure5}
\end{figure}
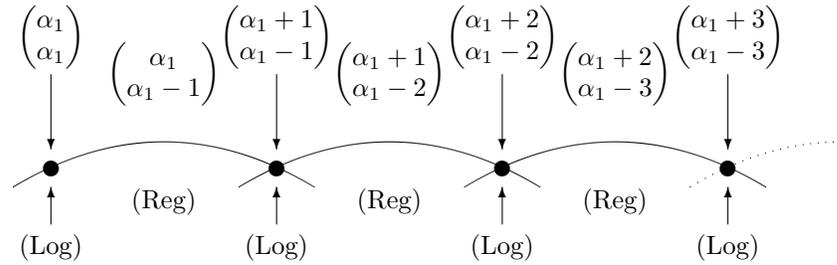

\end{Theorem}

\begin{proof}
We consider only marked $(TE)$-structures with a
fixed reference pair $\big(H^{{\rm ref},\infty},M^{\rm ref}\big)$.
Because of the markings, we can identify for each such
$(TE)$-structure its pair $(H^\infty,M^{\rm mon})$
with the reference pair $\big(H^{{\rm ref},\infty},M^{\rm ref}\big)$.
Thus also the spaces $C^\alpha$ can be identified
for all marked $(TE)$-structures.

$(a)$ $(i)$ and $(b)$ In both parts, there is no harm in fixing
elementary sections $s_1\in C^{\alpha_1}$ and
$s_2\in C^{\alpha_2}$ as in Theorem~\ref{t4.17}.
Then Theorem~\ref{t4.17} lists all marked $(TE)$-structures
with the given reference pair.
Remarks~\ref{t7.1}$(i)$ just put these marked
$(TE)$-structures into families parametrized by the spaces
$M^{(3),0,\alpha_1,\alpha_2}$ resp.~$M^{(3),\neq 0,\alpha_1,\alpha_2}$.
Most logarithmic $(TE)$-structures (which are classified
in Theorem~\ref{t4.20}) turn up in two such
families. This leads to the identification of the point
$\infty$ in $M^{(3),0/\neq 0,\alpha_1,\alpha_2}$
with the point $0$ in
$M^{(3),0/\neq 0,\alpha_1+1,\alpha_2-1}$.
Only each of the logarithmic $(TE)$-structures
with $N^{\rm mon}\neq 0$ and leading exponents
$\alpha_1=\alpha_2$ turns up in only one $\P^1$,
in the space $M^{(3),\neq 0,\alpha_1,\alpha_1-1}$.
There it is over the point $0$.

$(a)$ $(ii)$ Here the leading exponents satisfy
$\alpha_1-\alpha_2\in\Z\setminus\{0\}$, and we index them
such that $\alpha_1-\alpha_2\in\N$. We fix a basis
$\sigma_1$, $\sigma_2$ of $C^{\alpha_1}$
and define $\sigma_3:=z^{\alpha_2-\alpha_1}\sigma_1\in
C^{\alpha_2}$,
$\sigma_4\:=z^{\alpha_2-\alpha_1}\sigma_2\in C^{\alpha_2}$.
Then because of Theorem~\ref{t4.17}$(a)$,
we can write all marked regular singular, but not
logarithmic $(TE)$-structures with leading exponents
$\alpha_1$ and~$\alpha_2$ in two charts $\C\times\C^*$
with coordinates $(r_1,t_2)$ and $(r_2,t_3)$,
\begin{gather*}
\OO(H)_0= \C\{z\}(\sigma_1+t_2(\sigma_4+r_1\sigma_3))\oplus \C\{z\}(z(\sigma_4+r_1\sigma_3)),
\\
\OO(H)_0= \C\{z\}(\sigma_2+t_3(\sigma_3+r_2\sigma_4))\oplus \C\{z\}(z(\sigma_3+r_2\sigma_4)).\nonumber
\end{gather*}
The charts overlap where $r_1,r_2\in\C^*$, with
\begin{gather*}
r_2=r_1^{-1},\qquad
t_3=-t_2r_1^2.
\end{gather*}
Compactification to $t_2= 0$ and $t_2=\infty$
(and $t_3=0$ and $t_3=\infty$) gives
the Hirzebruch surface $\F_2=\F_2^{\alpha_1,\alpha_2}$.
The curve with $t_2=0$ (and $t_3=0$) is the $(-2)$-curve.
Over this curve, we have the family of marked logarithmic
$(TE)$-structures (see Theorem~\ref{t4.20}) with
leading exponents~$\alpha_1$ and~$\alpha_2+1$,
\begin{gather*}
\OO(H)_0= \C\{z\}(\sigma_1)\oplus \C\{z\}(z(\sigma_4+r_1\sigma_3)),
\\
\OO(H)_0= \C\{z\}(\sigma_2)\oplus \C\{z\}(z(\sigma_3+r_2\sigma_4)).
\end{gather*}

The curve with $t_2=\infty$ (and $t_3=\infty$)
is a $(+2)$-curve.
Over this curve, we have the family of marked logarithmic
$(TE)$-structures with
leading exponents $\alpha_1+1$ and $\alpha_2$.
Therefore the $(+2)$-curve in $\F_2^{\alpha_1,\alpha_2}$
must be identified with the $(-2)$-curve in
$\F_2^{\alpha_1+1,\alpha_2-1}$.

In the case $\alpha_1-\alpha_2=2$, the $(-2)$-curve
in $\F_2^{\alpha_1,\alpha_2}$ is the moduli space
$M^{(3),0,\alpha_1-1,{\rm log}}$ from Remark~\ref{t7.1}$(i)$.

In the case $\alpha_1-\alpha_2=1$, the $(-2)$-curve
in $\F_2^{\alpha_1,\alpha_2}$ has to be blown down,
as then for $t_2=0$
\begin{gather*}
\OO(H)_0=\C\{z\}(\sigma_1)\oplus\C\{z\}(z(\sigma_4+r_1\sigma_3))
=\C\{z\}C^{\alpha_1}=V^{\alpha_1}
\end{gather*}
is independent of the parameter $r_1$.

The projection $(r_1,t_2)\mapsto r_1$ extends
to the $\P^1$-fibration of $\F_2^{\alpha_1,\alpha_2}$
over $\P^1$.
The fibers are isomorphic to $M^{(3),0,\alpha_1,\alpha_2}$.
Affine coordinates on these fibers are $t_2$
and $\www t_2=t_2^{-1}$ or $t_3$ and $\www t_3=t_3^{-1}$.
\end{proof}

\begin{Remarks}\label{t7.5}\quad

\begin{enumerate}\itemsep=-1pt
\item[$(i)$] Consider a marked regular singular rank $2$ $(TE)$-structure
$((H\to\C\times M,\nnn),\psi)$. There is a unique
map $\varphi\colon M\to M^{(H^{{\rm ref},\infty},M^{\rm ref}),{\rm reg}}$,
which maps $t\in M$ to the unique point in~$M^{(H^{{\rm ref},\infty},M^{\rm ref}),{\rm reg}}$ over which one has up to
marked isomorphism the same marked $(TE)$-structure
as over $t$. Corollary~\ref{t7.3} and the fact that the
moduli space represents the moduli functor
$\MM^{(H^{{\rm ref},\infty},M^{\rm ref}),{\rm reg}}$, imply that
$\varphi$ is holomorphic.

Because $M$ is (simply) connected, the map $\varphi$ goes to
one irreducible component of the moduli space,
so to one $M^{(3),0/\neq 0,\alpha_1,\alpha_2}\cong\P^1$
in the parts $(a)$ $(i)$ and $(b)$ in Theorem~\ref{t7.4}
and to one $\F_2^{\alpha_1,\alpha_2}$ or to
$\www\F_2^{\alpha_1,\alpha_1-1}$ in part $(a)$ $(ii)$.

\item[$(ii)$] In fact, in part $(a)$ $(ii)$ the map $\varphi$ goes even
to a projective curve which is isomorphic to one curve
$M^{(3),0,\alpha_1,\alpha_2}$ or to the curve
$M^{(3),0,\alpha_1,{\rm log}}$.
This holds for the $(TE)$-structure over any manifold,
as it holds by Remark~\ref{t6.4}$(iv)$ and Theorem~\ref{t6.7} for the $(TE)$-structures over the
1-dimensional germ $(\C,0)$.

The curves isomorphic to $M^{(3),0,\alpha_1,\alpha_2}$
are the $(0)$-curves in the $\P^1$ fibration of
$\F_2^{\alpha_1,\alpha_2}$ over~$\P^1$ \big(in the case of
$\www\F_2^{\alpha_2+1,\alpha_2}$ each fiber of
$\F_2^{\alpha_2+1,\alpha_2}$ over $\P^1$ embeds also into
the blown down surface $\www\F_2^{\alpha_2+1,\alpha_2}$\big).
The curve isomorphic to $M^{(3),0,\alpha_1,{\rm log}}$
is the $(-2)$-curve in $\F_2^{\alpha_1,\alpha_1-2}$.

\item[$(iii)$] We have here a notion of horizontal directions which
is similar to that for classifying spaces of
Hodge structures. There it comes from Griffiths
transversality. Here it comes from the part of the
pole of Poincar\'e rank 1, which says that the covariant
derivatives $\nnn_{\paa_j}$ along vector fields on the
base space see only a pole of order 1.

In the cases of the $\F_2^{\alpha_1,\alpha_2}$
with $\alpha_1-\alpha_2\in\N\setminus \{1,2\}$, the horizontal
directions are the tangent spaces to the fibers of the
$\P^1$ fibration. In~the cases of $\F_2^{\alpha_1,\alpha_1-2}$ and
$\www \F_2^{\alpha_1,\alpha_1-1}$, the horizontal directions
contain these tangent spaces.
However, on points in the $(-2)$-curve in
$\F_2^{\alpha_1,\alpha_1-2}$
and on the singular point in
$\www\F_2^{\alpha_1,\alpha_1-1}$,
any direction is horizontal.
\end{enumerate}
\end{Remarks}

{\sloppy\begin{Remark}
If we forget the markings of the $(TE)$-structures in one
moduli space $M^{(H^{{\rm ref},\infty},M^{\rm ref}),{\rm reg}}$
and consider the unmarked $(TE)$-structures up to isomorphism,
we obtain in the cases $N^{\rm mon}=0$ countably many
points, one for each intersection point or intersection
curve of two irreducible components,
and one for each irreducible component.
On the contrary, in the cases $N^{\rm mon}\neq 0$,
the unmarked and the marked $(TE)$-structures almost coincide,
as the choice of an elementary section $s_1$ in
Theorem~\ref{t4.17}$(b)$ fixes uniquely the elementary
section~$s_2$ with~\eqref{4.50}.
The set of unmarked $(TE)$-structures up to isomorphism
is still almost in~bijec\-tion with the moduli space
$M^{(H^{{\rm ref},\infty},M^{\rm ref}),{\rm reg}}$ in the case $N^{\rm mon}\neq0$.
Only the components $M^{(3),\neq 0,\alpha_1,\alpha_1}-\{\infty\}$
boil down to single points.
\end{Remark}

}

\section[Unfoldings of rank 2 $(TE)$-structures of type (Log) over a point]
{Unfoldings of rank 2 $\boldsymbol{(TE)}$-structures of type (Log) \\over a point}\label{c8}

Sections~\ref{c5} and~\ref{c8} together
treat all rank $2$ $(TE)$-structures over germs $\big(M,t^0\big)$ of
manifolds. Section~\ref{c5} treated the unfoldings
of $(TE)$-structures of types (Sem) or (Bra) or (Reg)
over $t^0$.
Section~\ref{c8} will treat the unfoldings of
$(TE)$-structures of type (Log) over $t^0$.

It builds on Section~\ref{c6}, which classified the
unfoldings with trace free pole parts over
$\big(M,t^0\big)=(\C,0)$ of a logarithmic rank $2$ $(TE)$-structure
over $t^0$ and on Section~\ref{c7},
which treated arbitrary regular singular rank $2$
$(TE)$-structures.
Here Lemmata~\ref{t3.10} and~\ref{t3.11} are helpful.
They allow to go from arbitrary $(TE)$-structures
to $(TE)$-structures with trace free pole parts
and vice versa.

Section~\ref{c8.1} gives the classification results.
Section~\ref{c8.2} extracts from them a
characterization of the space of all $(TE)$-structures
with generically primitive Higgs fields over a given
germ of a~2-dimensional $F$-manifold with Euler field.
Section~\ref{c8.3} gives the proof of Theorem~\ref{t8.5}.

First we characterize in Theorem~\ref{t8.1}
the 2-parameter unfoldings of rank $2$
$(TE)$-structures of type (Log) over a point such that
the Higgs field is generically primitive
and induces an $F$-manifold structure on the underlying
germ $\big(M,t^0\big)$ of a manifold.
Theorem~\ref{t8.1} is a rather imme\-diate
implication of Theorem~\ref{t6.3} and Theorem~\ref{t6.7}
together with Lemmata~\ref{t3.10} and~\ref{t3.11}.
Part $(d)$ gives an explicit classification.
The other results in this section will all refer
to this classification.

Corollary~\ref{t8.3} lists for any logarithmic rank $2$
$(TE)$-structure over a point $t^0$ all unfoldings within
the set of $(TE)$-structures in Theorem~\ref{t8.1}$(a)$.
The proof consists of inspection of the explicit
classification in Theorem~\ref{t8.1}$(d)$.

Theorem~\ref{t8.5} is the main result of this section. It~lists a finite subset of the unfoldings in~Theo\-rem~\ref{t8.1}$(d)$ with the following property:
Any unfolding of a rank $2$ $(TE)$-structure of type (Log)
over a point is induced by a $(TE)$-structure in this list.
The $(TE)$-structures in the list turn out to be
universal unfoldings of themselves.

The proof of Theorem~\ref{t8.5} is long. It~is deferred
to Section~\ref{c8.3}. The results of Section~\ref{c6}
are crucial, especially Theorem~\ref{t6.3} and
Theorem~\ref{t6.7}.

Finally, Theorem~\ref{t8.6} lists the rank $2$
$(TE)$-structures over a germ $\big(M,t^0\big)$ of a manifold
such that the Higgs field is primitive (so that $\big(M,t^0\big)$
becomes a germ of an $F$-manifold with Euler field)
and the restriction over $t^0$ is of type (Log).
This list turns out to be a sublist of the one in
Theorem~\ref{t8.5}.
Theorem~\ref{t8.6} follows easily from Theorem~\ref{t8.1}.

Theorem~\ref{t8.6} is also contained in the papers~\cite{DH20-3}
 and~\cite{DH20-2}, the generic types (Bra), (Reg) and (Log)
are in~\cite{DH20-3}, the generic type (Sem)
is in~\cite{DH20-2}.
The proofs there are completely different. They build on the
formal classification of $(T)$-structures in~\cite{DH20-1}.

\subsection{Classification results}\label{c8.1}

\begin{Theorem}\label{t8.1}\quad
\begin{enumerate}\itemsep=0pt
\item[$(a)$]
Consider a rank $2$ $(TE)$-structure $\big(H\to\C\times \big(M,t^0\big),\nnn\big)$
over a $2$-dimen\-sio\-nal germ $\big(M,t^0\big)$
with restriction over $t^0$ of type $($Log$)$,
with generically primitive Higgs field,
and such that the induced $F$-manifold structure on generic
points of $M$ extends to all of $M$.

There is a unique rank $2$ $(TE)$-structure
$\big(H^{[3]}\to\C\times(\C,0),\nnn^{[3]}\big)$ over $(\C,0)$
(with coordinate $t_2$) with trace free pole part,
with nonvanishing Higgs field
and with logarithmic restriction over $t_2=0$ such that
$(\OO(H),\nnn)$ arises from $\big(\OO(H^{[3]}),\nnn^{[3]}\big)$
as follows. There are coordinates $t=(t_1,t_2)$
on $\big(M,t^0\big)$ such that $\big(M,t^0\big)=\big(\C^2,0\big)$
and a constant $c_1\in\C$ such that
\begin{gather}\label{8.1}
(\OO(H),\nnn)\cong\pr_2^*\big(\OO\big(H^{[3]}\big),\nnn^{[3]}\big)\otimes
\EE^{(t_1+c_1)/z},
\end{gather}
where $\pr_2\colon \big(M,t^0\big)\to(\C,0)$, $(t_1,t_2)\mapsto t_2$
$\big($see Lemma~$\ref{t3.10}(a)$ for $\EE^{(t_1+c_1)/z}\big)$.

The $(TE)$-structure $(H,\nnn)$ is of type $($Log$)$
over $(\C\times\{0\},0)$ and of one generic type
$($Sem$)$ or $($Bra$)$ or $($Reg$)$ or $($Log$)$ over $(\C\times\C^*,0)$.

\item[$(b)$] Vice versa, if $\big(H^{[3]},\nnn^{[3]}\big)$ is as in (a)
and $c_1\in\C$, then the $(TE)$-structure
$(\OO(H),\nnn):=\pr_2^*\big(\OO\big(H^{[3]}\big),\nnn^{[3]}\big)\otimes
\EE^{(t_1+c)/z}$ over $\big(M,t^0\big)=\big(\C^2,0\big)$ satisfies
the properties in $(a)$.

\item[$(c)$] The rank $2$ $(TE)$-structures $\big(H^{[3]},\nnn^{[3]}\big)$
over $(\C,0)$ with trace free pole part, nonvanishing Higgs field
and logarithmic restriction over $0$
are classified in Theorems~$\ref{t6.3}$ and~$\ref{t6.7}$.
They are in suitable coordinates the first $7$ of the $9$
cases in the list in Theorem~$\ref{t6.3}$
and the cases~\eqref{6.49} and~\eqref{6.50}
with $f=\frac{1}{k_1}t^{k_1}$ for some $k_1\in\N$
in Theorem~$\ref{t6.7}$. $($Though here the 6th case
in Theorem~$\ref{t6.3}$ is part of the cases~\eqref{6.49}
and~\eqref{6.50} in Theorem~$\ref{t6.7}.)$

\item[$(d)$] The explicit classification of the $(TE)$-structures
$(H,\nnn)$ in $(a)$ is as follows.
There are coordinates $(t_1,t_2)$ such that $\big(M,t^0\big)=\big(\C^2,0\big)$,
and there is a $\C\{t,z\}$-basis $\uuuu{v}$ of $\OO(H)_0$
whose matrices $A_1,A_2,B\in M_{2\times 2}(\C\{t,z\})$ with
$z\nnn_{\paa_i}\uuuu{v}=\uuuu{v}A_i$,
$z^2\nnn_{\paa_z}\uuuu{v}=\uuuu{v}B$ are in the following
list of normal forms. The normal form is unique.
We always have
\begin{gather*}
A_1=C_1.
\end{gather*}
Always $M$ is an $F$-manifold with Euler field
in one of the normal forms in Theorems~$\ref{t2.2}$
and~$\ref{t2.3}$ $($in the case $(i)$ the product
$\paa_2\circ\paa_2$ is only almost in the normal form
in Theorem~$\ref{t2.2}$;
in the case $(iii)$ with $\alpha_4=-1$
the Euler field is only almost in the normal form
in Theorem~$\ref{t2.3})$.

\begin{enumerate}\itemsep=0pt
\item[$(i)$] Generic type $($Sem$){:}$
invariants $k_1,k_2\in\N$ with $k_2\geq k_1$,
$c_1,\rho^{(1)}\in\C$, $\zeta\in\C$ if $k_2-k_1\in 2\N$,
$\alpha_3\in\R_{\geq 0}\cup\H$ if $k_1=k_2$,
\begin{gather*}
\gamma := \frac{2}{k_1+k_2},
\\
A_2=\begin{cases} -\gamma^{-1}
\big(t_2^{k_1-1}C_2 + t_2^{k_2-1}E\big)
\quad\text{if}\ k_2-k_1>0\ \text{is odd,}
\\[.5ex]
-\gamma^{-1}\big(t_2^{k_1-1}C_2 +\zeta t^{(k_1+k_2)/2-1}D+\big(1-\zeta^2\big)t_2^{k_2-1}E\big)
\quad \text{if}\ k_2-k_1\in 2\N,\!\!\!
\\[.5ex]
-\gamma^{-1}t_2^{k_1-1}D \quad\text{if}\ k_2=k_1,
\end{cases}
\\
B= \big({-}t_1-c_1+z\rho^{(1)}\big)C_1+(-\gamma t_2)A_2
+\begin{cases}
z\frac{k_1-k_2}{2(k_1+k_2)}D&\text{if}\ k_2>k_1,
\\
z\alpha_3 D&\text{if}\ k_2=k_1,\end{cases}
\\ \hphantom{B= }
\text{$F$-manifold }I_2(k_1+k_2)\ \big(\text{with }I_2(2)=A_1^2\big),\
\text{with}\ \paa_2\circ\paa_2
=\gamma^{-2}t_2^{k_1+k_2}\cdot \paa_1,
\\
E= (t_1+c_1)\paa_1 +\gamma t_2\paa_2\quad\text{Euler field}.
\end{gather*}

\item[$(ii)$] Generic type $($Bra$){:}$ invariants $k_1,k_2\in\N$,
$c_1,\rho^{(1)}\in\C$,
\begin{gather*}
\gamma := \frac{1}{k_1+k_2},
\\
A_2=-\gamma^{-1}\big(t_2^{k_1-1}C_2+t_2^{k_1+k_2-1}D-t_2^{k_1+2k_2-1}E\big),
\\
B= \big({-}t_1-c_1+z\rho^{(1)}\big)C_1+(-\gamma t_2)A_2+ z\frac{-k_2}{2(k_1+k_2)}D,
\\ \hphantom{B= }
\text{$F$-manifold }\NN_2,
\text{ with }\paa_2\circ\paa_2=0,
\\
E= (t_1+c_1)\paa_1 +\gamma t_2\paa_2\quad\text{Euler field}.
\end{gather*}

\item[$(iii)$] Generic type $($Reg$){:}$
invariants $c_1,\rho^{(1)}\in\C$,
$\alpha_4\in\C\setminus\{-1\}$ if~$N^{\rm mon}=0$,
$\alpha_4\in\Z$ if~$N^{\rm mon}\neq 0$,
$k_1\in\N$ if~$N^{\rm mon}=0$,
$\www k_1\in\N$ if~$N^{\rm mon}\neq 0$
$($with $k_1=\www k_1$ if~$\alpha_4\neq -1$,
and $k_1=2\www k_1$ if~$\alpha_4=-1)$,
\begin{gather*}
\gamma := \frac{1+\alpha_4}{k_1},
\\
A_2=\begin{cases}
-\gamma^{-1}t_2^{k_1-1}C_2&\text{if}\quad N^{\rm mon}=0,
\\[.5ex]
\www k_1t_2^{\www k_1-1}C_2&\text{if}\quad N^{\rm mon}\neq 0,
\end{cases}
\\
B= \big({-}t_1-c_1+z\rho^{(1)}\big)C_1+(-\gamma t_2)A_2+z\frac{1}{2}\alpha_4 D
\\ \hphantom{B= }
{}+ \begin{cases}
0 \qquad\text{if}\quad N^{\rm mon}=0,
\\
z^{\alpha_4+1}C_2\qquad\text{if}\quad N^{\rm mon}\neq 0,\quad
\alpha_4\in\N_0,
\\
-z^{-\alpha_4-1}t_2^{2\www k_1}C_2 +z^{-\alpha_4}t_2^{\www k_1}D+
z^{-\alpha_4+1}E\quad\text{if}\quad N^{\rm mon}\neq 0,\quad
\alpha_4\in\Z_{<0},\!\!\!\!\end{cases}
\\ \hphantom{B= }
\text{$F$-manifold }\NN_2,\
\text{with}\ \paa_2\circ\paa_2=0,
\\
E= \left\{\!\!\begin{array}{ll}
(t_1+c_1)\paa_1 +\gamma t_2\paa_2&\text{if}\quad\alpha_4\neq -1,\\[.5ex]
(t_1+c_1)\paa_1 +\frac{1}{\www k_1}t_2^{\www k_1+1} \paa_2
&\text{if}\quad\alpha_4=-1\end{array}\!\!\right\}
\quad\text{Euler field}.
\end{gather*}

\item[$(iv)$] Generic type $($Log$)$:
invariants $k_1\in\N$,
$c_1,\rho^{(1)}\in\C$,
\begin{gather*}
A_2=k_1t_2^{k_1-1}C_2,
\\
B= \big({-}t_1-c_1+z\rho^{(1)}\big)C_1-z\frac{1}{2} D,
\quad
\text{$F$-manifold }\NN_2,
\text{ with }\paa_2\circ\paa_2=0,\\
E= (t_1+c_1)\paa_1 \quad\text{Euler field}.
\end{gather*}
\end{enumerate}
\end{enumerate}
\end{Theorem}

Theorem~\ref{t8.1} is proved after Remark~\ref{t8.2}.

\begin{Remark}\label{t8.2}
The other normal forms in Remark~\ref{t6.6} for
the generic type (Sem) with $k_2-k_1\in 2\N$ and for the
generic type (Bra) give the following other normal
forms. In~both cases, the formulas for $A_1=C_1$,
$\gamma$, $B$, the $F$-manifold and $E$ are unchanged,
only the matrix $A_2$ changes.
For the generic type (Sem) with $k_2-k_1\in 2\N$, $A_2$ becomes
\begin{gather*}
A_2=-\gamma^{-1}\big(t_2^{k_1-1}C_2+t_2^{k_2-1}E\big) +
z\frac{k_2-k_1}{2}\zeta t_2^{(k_2-k_1-2)/2}E.
\end{gather*}
For the generic type (Bra), $A_2$ becomes
\begin{gather*}
A_2=-\gamma^{-1}t_2^{k_1-1}C_2+zk_2t_2^{k_2-1}E .
\end{gather*}
\end{Remark}

\begin{proof}[Proof of Theorem~\ref{t8.1}]
We prove the parts of Theorem~\ref{t8.1}
in the order $(c)$, $(d)$, $(b)$, $(a)$.

$(c)$ Consider a rank $2$ $(TE)$-structure
$\big(H^{[3]}\to\C\times(\C,0),\nnn^{[3]}\big)$
(with coordinate $t_2$ on $(\C,0)$)
with trace free pole part and with logarithmic
restriction over $t_2=0$. If it admits an
extension to a pure $(TLE)$-structure, it is contained in Theo\-rem~\ref{t6.3}. If not, then it is contained in~Theorem~\ref{t6.7}. The condition that the Higgs field
is not vanishing, excludes the 8th and 9th cases in Theo\-rem~\ref{t6.3} and the case~\eqref{6.51} = case (III)
in Theorem~\ref{t6.7}, see Remarks~\ref{t6.8}$(ii)$ and~$(iii)$.

$(d)$ Part $(d)$ makes for such a $(TE)$-structure
$\big(H^{[3]},\nnn^{[3]}\big)$ the $(TE)$-structure
$(\OO(H),\nnn)=\pr_2^*\big(\OO\big(H^{[3]}\big),\nnn^{[3]}\big)\otimes
\EE^{(t_1+c_1)/z}$ explicit.
The cooordinate $t$ and the matrix $A$ in Theorem~\ref{t6.3}
and in Remark~\ref{t6.8}$(iv)$ become now $t_2$ and $A_2$.
Here the matrices in the 6th case in Theorem~\ref{t6.3}
are not used, but the matrices in Remark~\ref{t6.8}$(iv)$.
The function $f$ in Remark~\ref{t6.8}$(iv)$ is now
specialized to $f=t_2^{k_1/2}$ if $\alpha_1=\alpha_2$
($\Rightarrow$ case (II) and~\eqref{6.54})
and to $f=t_2^{k_1}$ if $\alpha_1-\alpha_2\in\N$
(case~(I) and~\eqref{6.53} or case (II) and~\eqref{6.54}).
The new matrix $B$ is $(-t_1-c_1)C_1$ plus the matrix~$B$ in Theorem~\ref{t6.3} and in Remark~\ref{t6.8}$(iv)$.

In the normal forms in Remark~\ref{t6.8}$(iv)$ we replaced
$\alpha_1$ and $\alpha_2$ by $\rho^{(1)}$ and $\alpha_4$
as follows
\begin{gather*}
\rho^{(1)}:=\frac{\alpha_1+\alpha_2+1}{2},\qquad
\alpha_4:=\begin{cases}
\alpha_1-\alpha_2-1\in\N_0&\text{in}~\eqref{6.53},
\\
\alpha_2-\alpha_1-1\in\Z_{<0}&\text{in}~\eqref{6.54}.
\end{cases}
\end{gather*}

$(b)$ Now part $(b)$ follows from inspection of the normal forms
in part $(d)$.

$(a)$ Consider a $(TE)$-structure
as in $(a)$. Choose coordinates $t=(t_1,t_2)$ on $\big(M,t^0\big)$
such that $\big(M,t^0\big)=\big(\C^2,0\big)$ and the germ of the $F$-manifold
is in a normal form in Theorem~\ref{t2.2}
(especially $e=\paa_1$)
and the Euler field has the form $E=(t_1+c_1)\paa_1
+g(t_2)\paa_2$ for some $c_1\in\C$ and some
$g(t_2)\in\C\{t_2\}$.

Choose any $\C\{t,z\}$-basis $\uuuu{v}$ of
$\OO(H)_0$ and consider its matrices
$A_1$, $A_2 $, $B$ with $z\nnn_{\paa_i}\uuuu{v}=\uuuu{v}A_i$,
$z^2\nnn_{\paa_z}\uuuu{v}=\uuuu{v}B$.
Now $\paa_1=e$ implies $A_1^{(0)}=C_1$.
We make a base change with the matrix $T\in {\rm GL}_2(\C\{t,z\})$
which is the unique solution of the differential equation
\begin{gather*}
\paa_1 T=-\bigg(\sum_{k\geq 1}A_1^{(k)}z^{k-1}\bigg)T,\qquad
T(z,0,t_2)=C_1.
\end{gather*}
Then the matrices $\www A_1$, $\www A_2$, $\www B$
of the new basis $\uuuu{\www v}=\uuuu{v}T$ satisfy
\begin{gather}\label{8.10}
\www A_1=C_1,\qquad \paa_1\www A_2=0,\qquad \paa_1\www B=-C_1,
\end{gather}
because~\eqref{3.12} for $i=1$ and~\eqref{3.7}
and~\eqref{3.8} give
\begin{gather*}
0= z\paa_1T+A_1T-T\www A_1=C_1T-T\www A_1=T\big(C_1-\www A_1\big),
\\
0= z\paa_1\www A_2-z\paa_2\www A_1+\big[\www A_1,\www A_2\big]= z\paa_1 \www A_2,
\\
0= z\paa_1B-z^2\paa_z\www A_1+z\www A_1+\big[\www A_1,\www B\big]= z(\paa_1B+C_1).
\end{gather*}

In Lemmata~\ref{t3.10}$(c)$ and~\ref{t3.11}
we considered the $(TE)$-structure
$\big(\OO\big(H^{[2]}\big),\nnn^{[2]}\big)=(\OO(H),\nnn)\allowbreak\otimes \EE^{\rho^{(0)}/z}$
with trace free pole part. Here $\rho^{(0)}=-t_1-c_1$.

\eqref{8.10} shows that $\big(H^{[2]},\nnn^{[2]}\big)$ is the pull
back of its restriction $\big(H^{[3]},\nnn^{[3]}\big)$ to
$(\{0\}\times \C,0)\subset \big(\C^2,0\big)$. This and
$(\OO(H),\nnn)\cong \big(\OO\big(H^{[2]}\big),\nnn^{[2]}\big)\otimes
\EE^{-\rho^{(0)}/z}$ in Lemma~\ref{t3.10}$(c)$ imply~\eqref{8.1}.
\end{proof}

\begin{Corollary}\label{t8.3}
The following table gives for each logarithmic rank $2$
$(TE)$-structure over a point~$t^0$ its unfoldings within
the set of $(TE)$-structures in Theorem~$\ref{t8.1}(d)$.
Here the set $\big\{\alpha_1^0,\alpha_2^0\big\}\subset\C$ is the
set of leading exponents in Theorem~$\ref{t4.20}$ of the
logarithmic $(TE)$-structure over~$t^0$.
So, in the case $N^{\rm mon}=0$, $\alpha_1^0$ and $\alpha_2^0\in\C$
are arbitrary. In~the case $N^{\rm mon}\neq 0$, they satisfy
$\alpha_1^0-\alpha_2^0\in\N_0$. Two conditions are
$c_1=0$ and $\rho^{(1)}=\frac{\alpha_1^0+\alpha_2^0}{2}$.
The other conditions and the other invariants
$($though without their definition domains$)$
are given in the table.
All invariants in Theorem~$\ref{t8.1}(d)$, which are not
mentioned here, are $($intended to be$)$ arbitrary:
\begin{gather}\label{8.11}
\def\arraystretch{1.25}
\begin{tabular}{l|l|l|l}
\hline
\multicolumn{1}{c|}{Generic type}&\multicolumn{1}{c|}{Invariants} & \multicolumn{1}{c|}{$N^{\rm mon}$} &\multicolumn{1}{c}{Condition}
\\
\hline
$(Sem)\colon\ k_2>k_1$ & $k_1,k_2,\zeta$ & $=0$&
$\alpha_1^0-\alpha_2^0=\pm \frac{k_1-k_2}{k_1+k_2}$
\\
$(Sem)\colon\ k_2=k_1$ & $k_1,k_2,\alpha_3$ & $=0$&
$\alpha_1^0-\alpha_2^0=\pm 2\alpha_3$
\\
\hline
$(Bra)$ & $k_1,k_2$ &$ =0$&$\alpha_1^0-\alpha_2^0=\pm \frac{-k_2}{k_1+k_2} $
\\
\hline
$(Reg)$ & $k_1,\alpha_4$ & $=0$ &$\alpha_1^0-\alpha_2^0=\pm \alpha_4$
\\
\hline
$(Reg)$ & $\www k_1,\alpha_4$ & $\neq 0$ &$\alpha_1^0-\alpha_2^0=|\alpha_4| $
\\
\hline
$(Log)$ &$ k_1$ & $=0$ & $\alpha_1^0-\alpha_2^0=\pm 1$
\\
\hline
\end{tabular}
\end{gather}
\end{Corollary}

\begin{proof} This follows from inspection of the cases
in Theorem~\ref{t8.1}$(d)$.
\end{proof}

\begin{Remark}
Beware of the following:

\begin{enumerate}\itemsep=0pt
\item[$(i)$] In the generic case (Sem) with $k_1=k_2$ we have
$\alpha_3\in\R_{\geq 0}\cup\H$. Here $\www\alpha_3=-\alpha_3$
is excluded, as it gives an isomorphic unfolding.

\item[$(ii)$] In the generic cases (Reg) with
$\alpha_1^0-\alpha_2^0\in\C\setminus\{0\}$ almost always
$\alpha_4=\alpha_1^0-\alpha_2^0$ and $\www\alpha_4=-\alpha_4$
give (for the same $k_1\in\N$ respectively $\www k_1\in\N$)
two different unfoldings.
The~only exception is the case $N^{\rm mon}=0$ and
$\alpha_1^0-\alpha_2^0=\pm 1$, as then $\alpha_4=-1$
is not allowed.

\item[$(iii)$] In the generic case (Log), one has
one unfolding (and not two unfoldings) for each $k_1\in\N$.

\item[$(iv)$] Unfoldings of generic type (Sem) with $k_2>k_1$ and of
generic type (Bra) exist only if
$\alpha_1^0-\alpha_2^0\in (-1,1)\cap\Q^*$ and $N^{\rm mon}=0$.
\end{enumerate}
\end{Remark}

\begin{Theorem}\label{t8.5}\quad
\begin{enumerate}\itemsep=0pt
\item[$(a)$] Any unfolding of a rank $2$ $(TE)$-structure of type $($Log$)$
over a point is induced by one in the
following subset of $(TE)$-structures in Theorem~$\ref{t8.1}(d)$:
\begin{gather}\label{8.12}
\def\arraystretch{1.25}
\begin{tabular}{l|l}
\hline
\multicolumn{1}{c|}{Generic type and invariants} & \multicolumn{1}{c}{Condition}
\\ \hline
$($Sem$)\colon\ k_2-k_1>0$ odd & $\gcd(k_1,k_2)=1$
\\
$($Sem$)\colon\ k_2-k_1\in 2\N$, $\zeta= 0$ & $\gcd(k_1,k_2)=1$
\\
$($Sem$)\colon\ k_2-k_1\in 2\N$, $\zeta\neq 0$ &$\gcd\big(k_1,\frac{k_1+k_2}{2}\big)=1$
\\
$($Sem$)\colon\ k_2=k_1$ & $k_2=k_1=1 $
\\ \hline
$($Bra$)$ & $\gcd(k_1,k_2)=1$
\\ \hline
$($Reg$)\colon\ N^{\rm mon}=0$ & $k_1=1$
\\ \hline
$($Reg$)\colon\ N^{\rm mon}\neq 0$ & $\www k_1=1$
\\ \hline
$($Log$)\colon\ N^{\rm mon}=0$ &$ k_1=1$
\\
\hline
\end{tabular}
\end{gather}

\item[$(b)$] The inducing $(TE)$-structure is not unique only if
the original $(TE)$-structure has the form
$\varphi^*\big(\OO\big(H^{[5]}\big),\nnn^{[5]}\big)\otimes \EE^{-\rho^{(0)}/z}$,
where $\big(H^{[5]},\nnn^{[5]}\big)$ is a logarithmic
$(TE)$-structure over a point~$t^{[5]}$ and
$\varphi\colon \big(M,t^0\big)\to\big\{t^{[5]}\big\}$ is the projection,
and $\big(H^{[5]},\nnn^{[5]}\big)$ is not one with $N^{\rm mon}\neq 0$
and equal leading exponents $\alpha_1=\alpha_2$.
Then the original $(TE)$-structure is of type $($Log$)$
everywhere with Higgs field endomorphisms
$C_X\in\OO_{(M,t^0)}\cdot\id$ for any $X\in \TT_{(M,t^0)}$.

\item[$(c)$] The $(TE)$-structures in the list in $(a)$ are universal
unfoldings of themselves.
\end{enumerate}
\end{Theorem}

The proof of Theorem~\ref{t8.5} will be given in
Section~\ref{c8.3}.

\begin{Theorem}\label{t8.6}
The set of rank $2$ $(TE)$-structures
with primitive $($not just generically primitive$)$
Higgs field over a germ $\big(M,t^0\big)$
of an $F$-manifold and with restriction of type $($Log$)$ over $t^0$
is $($after the choice of suitable coordinates$)$
the proper subset of those in the list~\eqref{8.12}
in Theorem~$\ref{t8.5}$ which satisfy $k_1=1$
respectively $\www k_1=1$. In~the cases $($Reg$)$ and $($Log$)$, it coincides with
the list~\eqref{8.12}. In~the cases $($Sem$)$ and $($Bra$)$,
it is a proper subset.
\end{Theorem}

\begin{proof} The set of rank $2$ $(TE)$-structures
with primitive Higgs field over a germ $\big(M,t^0\big)$
of an~$F$-manifold and with restriction of type (Log) over $t^0$
consists by Theorem~\ref{t8.1}$(a){+}(d)$ of those
$(TE)$-structures in Theorem~\ref{t8.1}$(d)$
which satisfy $A_2(t_2=0)\notin\C\cdot C_1$.
This holds if and only if
$k_1=1$ respectively $\www k_1=1$
$\big(\www k_1=1$ if the generic type is (Reg) and $N^{\rm mon}\neq 0\big)$,
and then $A_2(t_2=0)\in\{-\gamma C_2,-\gamma D,C_2\}$.
Obviously, this is a proper
subset of those in table~\eqref{8.12} in the generic
cases (Sem) and (Bra), and it coincides with those
in table~\eqref{8.12} in the generic cases
(Reg) and (Log).
\end{proof}

\subsection[$(TE)$-structures over given $F$-manifolds with Euler fields]{$\boldsymbol{(TE)}$-structures over given $\boldsymbol{F}$-manifolds with Euler fields}\label{c8.2}

\begin{Remarks}
For a given germ $\big(\big(M,t^0\big),\circ,e,E\big)$ of an $F$-manifold
with Euler field, define
\begin{gather*}
B_1\big(\big(M,t^0\big),\circ,e,E\big):=
\big\{(TE)\text{-structures over }\big(M,t^0\big)
\text{ with generically primitive Higgs}
\\ \hphantom{B_1(\big(M,t^0\big),\circ,e,E):=\big\{}
\text{field, inducing the given $F$-manifold structure with Euler field}\big\},
\\
B_2\big(\big(M,t^0\big),\circ,e,E\big):=\{(TE)\text{-structures in }B_1
\text{ which are in table~\eqref{8.12}}\},
\\
B_3\big(\big(M,t^0\big),\circ,e,E\big):=
\{(TE)\text{-structures in }B_1\text{ with primitive Higgs fields}\}.
\end{gather*}
Now we can answer the questions, how big these sets are.
Often we write $B_j$ instead of $B_j\big(\big(M,t^0\big),\circ,e,E\big)$,
when the germ $\big(\big(M,t^0\big),\circ,e,E\big)$ is fixed.

\begin{enumerate}\itemsep=0pt
\item[$(i)$] First we consider the cases when
the germ $\big(\big(M,t^0\big),\circ, e,E\big)$ is regular.
Compare Remark~\ref{t2.6}$(ii)$ and Remark~\ref{t3.17}$(iii)$.
By Malgrange's unfolding result Theorem~\ref{t3.16}$(c)$,
any $(TE)$-structure over $\big(M,t^0\big)$ is the universal
unfolding of its restriction over $t^0$,
and it is its own universal unfolding.
So then $B_1=B_2=B_3$, and the classification of
the $(TE)$-structures over points in Section~\ref{c4}
determines this space $B_1$.

In the case of $A_1^2$ with $E=(u_1+c_1)e_1+(u_2+c_2)e_2$
with $c_1\neq c_2$, any $(TE)$-structure over~$t^0$
is of type (Sem). Theorem~\ref{t4.5}
tells that then $B_1$ is connected and 4-dimensional.
The parameters are the two regular singular exponents
and two Stokes parameters.

In the case of $\NN_2$ with $E=(t_1+c_1)\paa_1+\paa_2$,
any $(TE)$-structure over $t^0$ is either of type (Bra)
or of type (Reg). Then $B_1$ has one component
for type (Bra) and countably many components for type (Reg).

The component for type (Bra) is connected and 3-dimensional.
The parameters are given in Theorem~\ref{t4.11},
they are $\rho^{(1)}$, $\delta^{(1)}$ and
$\Eig(M^{\rm mon})$ \big(here $\rho^{(0)}\big(t^0\big)=-c_1$ is fixed, and
one eigenvalue and $\rho^{(1)}$ determine the other eigenvalue\big).

Corollary~\ref{t4.18} gives the countably many components
for type (Reg). One is 1-dimensional, the others are
2-dimensional.

\item[$(ii)$] Now we consider the cases when the germ
$\big(\big(M,t^0\big),\circ,e,E\big)$ is not regular.
Then $E|_{t^0}=c_1\paa_1$ for some $c_1\in\C$.
If $(\OO(H),\nnn)$ is a $(TE)$-structure in
$B_j\big(\big(M,t^0\big),\circ,e,E\big)$, then
$(\OO(H),\nnn)\otimes \EE^{-c_1/z}$ is a $(TE)$-structure in
$B_j\big(\big(M,t^0\big),\circ,e,E-c_1\paa_1\big)$.
Therefore we can and will restrict to the cases with
$E|_{t^0}=0$.

Theorem~\ref{t8.1}$(d)$ gives the $(TE)$-structures in $B_1$,
Theorem~\ref{t8.5}$(a)$ gives the $(TE)$-structures in $B_2$,
and Theorem~\ref{t8.6} gives the $(TE)$-structures in $B_3$.
For each germ $\big(\big(M,t^0\big),\circ,e,E\big)$ with $E|_{t^0}=0$
\begin{gather*}
B_1\supset B_2\supset B_3.
\end{gather*}
In the cases $A_1^2$ and $I_2(m)$, the Euler field with
$E|_{t^0}=0$ is unique on $\big(M,t^0\big)$, therefore we do not
write it down.

In {\it the case of $I_2(m)$ with $m\in 2\N$} (this includes
the case $A_1^2=I_2(2)$)
\begin{gather}
B_1(I_2(m))\cong \dot{\bigcup\limits_{(k_1,k_2)\in\N^2\colon
k_1+k_2=m,k_2\geq k_1}}\C^2,\nonumber
\\
B_2(I_2(m))\cong \dot{\bigcup\limits_{(k_1,k_2)\in\N^2\colon
k_1+k_2=m,k_2\geq k_1,\gcd(k_1,m/2)=1}}\C^2, \nonumber
\\
B_3(I_2(m))\cong \C^2,\quad\text{here}\quad(k_1,k_2)=(1,m-1).
\label{8.17}
\end{gather}
The 2 continuous parameters are the regular singular exponents of the
$(TE)$-structures at generic points in $M$.

In {\it the case of $I_2(m)$ with $m\geq 3$ odd},
\begin{gather}
B_1(I_2(m))\cong \dot{\bigcup\limits_{(k_1,k_2)\in\N^2\colon
k_1+k_2=m,k_2>k_1}}\C,\nonumber
\\
B_2(I_2(m))\cong \dot{\bigcup\limits_{(k_1,k_2)\in\N^2\colon
k_1+k_2=m,k_2>k_1,\gcd(k_1,k_2)=1}}\C, \nonumber
\\
B_3(I_2(m))\cong\C, \qquad\text{here}\quad (k_1,k_2)=(1,m-1).
\label{8.18}
\end{gather}
For odd $m\geq 3$, the regular singular exponents of the
$(TE)$-structures at generic points in~$M$ coincide and
give the continuous parameter.

Especially, for $m\in\{2,3\}$
\begin{gather*}
B_1(I_2(m))=B_2(I_2(m))=B_3(I_2(m))\cong
\begin{cases} \C^2&\text{if}\ m=2,\\
\C&\text{if}\ m=3.\end{cases}
\end{gather*}

The $F$-manifold $\NN_2$ allows by Theorem~\ref{t2.3}
many nonisomorphic Euler fields with $E|_{t^0}=0$, the
cases~\eqref{2.10}--\eqref{2.12} with $c_1=0$.

{\it The case~\eqref{2.10}}, $E=t_1\paa_1$:
Here each $(TE)$-structure has generic type (Log)
and semisimple monodromy. Here
\begin{gather*}
B_1(\NN_2,E)\cong \dot{\bigcup\limits_{k_1\in\N}}\C,
\\
B_2(\NN_2,E)= B_3(\NN_2,E) \cong \C,\qquad
\text{here}\quad k_1=1.\nonumber
\end{gather*}
The continuous parameter is $\rho^{(1)}$ in Theorem~\ref{t8.1}$(d)$ (iv) or, equivalently, one of the
two residue eigenvalues \big(which are $\rho^{(1)}\pm\frac{1}{2}$\big).

{\it The case~\eqref{2.12}},
$E=t_1\paa_1+t_2^r\big(1+c_3t_2^{r-1}\big)$
for some $r\in\Z_{\geq 2}$ and some $c_3\in\C$:
Here each $(TE)$-structure has generic type (Reg) and
satisfies $N^{\rm mon}\neq 0$. Here
\begin{gather*}
B_1(\NN_2,E) \begin{cases}
=\varnothing,& \text{if}\quad c_3\in\C^*,\\
\cong \C&\text{if}\quad c_3=0,\end{cases}
\\
B_2(\NN_2,E)=B_3(\NN_2,E) \begin{cases}
=\varnothing,& \text{if}\quad c_3\in\C^*\quad \text{or}\quad r\geq 3,\\
=B_1(\NN_2,E)&\text{if}\quad c_3=0\quad \text{and}\quad r=2.
\end{cases}
\end{gather*}
So, $(\NN_2,E)$ with $c_3\in\C^*$ does not allow
$(TE)$-structures over it, and $(\NN_2,E)$ with
$c_3=0$ and $r\geq 3$ does not allow $(TE)$-structures
over it with primitive Higgs field.
If $B_j(\NN_2,E)\allowbreak\neq\varnothing$ then $B_j(\NN_2,E)\cong\C$
and the continuous parameter is $\rho^{(1)}$ in
Theorem~\ref{t8.1}$(d)$ $(iii)$.

{\it The case~\eqref{2.11}}, $E=t_1\paa_1+c_2t_2\paa_2$
for some $c_2\in\C^*$: This is a rich case. Here we
decompose $B_j=B_j(\NN_2,E)$ as
\begin{gather*}
B_j=B_j^{({\rm Reg}),0}
\ \dot\cup\ B_j^{({\rm Reg}),\neq 0}
\ \dot\cup\ B_j^{({\rm Bra})},
\end{gather*}
where the first set contains $(TE)$-structures of generic
type (Reg) with $N^{\rm mon}=0$, the second set contains
$(TE)$-structures of generic type (Reg) with $N^{\rm mon}\neq 0$,
and the third set contains $(TE)$-structures of
generic type (Bra). Then
\begin{gather*}
B_1^{({\rm Reg}),0}\cong \dot{\bigcup\limits_{k_1\in\N}}\C,\qquad
B_2^{({\rm Reg}),0}=B_3^{({\rm Reg}),0}\cong\C,
\\
B_1^{({\rm Reg}),\neq 0}
\begin{cases}
=\varnothing&\text{if}\quad c_2\in\C^*\setminus\Q^*,
\\
\cong\dot\bigcup_{(k_1,\alpha_4)\in\N\times\Z\colon k_1c_2=1+\alpha_4}\C
&\text{if}\quad c_2\in\Q^*,
\end{cases} 
\\
B_2^{({\rm Reg}),\neq 0}=B_3^{({\rm Reg}),\neq 0}
\begin{cases}
=\varnothing&\text{if}\quad c_2\in\C\setminus\Z,
\\
\cong \C&\text{if}\quad c_2\in\Z\setminus\{0\},
\end{cases}
\\
B_1^{({\rm Bra})}=B_2^{({\rm Bra})}=B_3^{({\rm Bra})}=\varnothing
\qquad\text{if}\quad c_2^{-1}\in\C^*\setminus \Z_{\geq 2},
\\
\left.\!\!\! \begin{array}{l}
B_1^{({\rm Bra})} \cong\dot
\bigcup_{(k_1,k_2)\in\N^2\colon k_1+k_2=c_2^{-1}}\C,\\
B_2^{({\rm Bra})} \cong\dot
\bigcup_{(k_1,k_2)\in\N^2\colon k_1+k_2=c_2^{-1},\gcd(k_1,k_2)=1}
\C,\\
B_3^{({\rm Bra})} \cong\C,\quad\text{here }(k_1,k_2)=\big(1,c_2^{-1}-1\big)
\end{array}\!\!\right\}\qquad\text{if}\quad c_2^{-1}\in \Z_{\geq 2}.
\nonumber
\end{gather*}
\end{enumerate}
\end{Remarks}

\begin{Remarks}\label{t8.8}\quad

\begin{enumerate}\itemsep=0pt
\item[$(i)$] Theorem~\ref{t8.1}$(d)$ $(i)$ tells how many $(TE)$-structures
exist over the $F$-manifold with \mbox{Euler} field $I_2(m)$,
such that the Higgs bundle is generically primitive
and induces this $F$-manifold structure.
There are $\big[\frac{m}{2}\big]$ many holomorphic families
from the different choices of $(k_1,k_2)\allowbreak\in\N^2$ with
$k_2\geq k_1$ and $k_1+k_2=m$.
They have 2 parameters if $m$ is even and 1 parameter
if $m$ is odd, compare~\eqref{8.17} and~\eqref{8.18}.
For each $I_2(m)$, only one of these families consists
of $(TE)$-structures with primitive Higgs fields.

\item[$(ii)$] Consider $m\geq 3$.
Write $M=\C^2$ for the $F$-manifold $I_2(m)$
in Theorem~\ref{t2.2}, and $M^{\rm [log]}=\C\times\{0\}$
for the subset of points where the multiplication
is not semisimple. Over these points the restricted
$(TE)$-structures are of type (Log).
We checked that there are $\big[\frac{m}{2}\big]$ many Stokes
structures which give $(TE)$-structures on $M\setminus M^{\rm [log]}$.
Because of $(i)$, all these $(TE)$-structures extend
holomorphically over $M^{\rm [log]}$, and they give the
$\big[\frac{m}{2}\big]$ holomorphic families of $(TE)$-structures
on $I_2(m)$ in $(i)$.

\item[$(iii)$] Especially remarkable is the case $A_1^2=I_2(2)$.
There Theorem~\ref{t8.1}$(a){+}(d)$ $(i)$ implies directly that
each holomorphic $(TE)$-structure over $A_1^2$
with generically primitive Higgs field has primitive
Higgs field and is an elementary model (Definition~\ref{t4.4}),
so it has trivial Stokes structure.

\item[$(iv)$] This result is related to much more general work
in~\cite{CDG17} and~\cite{Sa19}
on meromorphic connections over the $F$-manifold
$A_1^n$ near points where some of the canonical
coordinates coincide.
Let us restrict to the special case of a neighborhood
of a point where all canonical coordinates coincide.
This generalizes the germ at 0 of $A_1^2$ to the germ at 0
of $A_1^n$.

\cite[Theorem~1.1]{CDG17} and~\cite[Theorem~3]{Sa19}
both give the triviality of the Stokes structure.
Though their starting points are slightly restrictive.
\cite{CDG17} starts in our notation from pure
$(TLE)$-structures with primitive Higgs fields.
The step before in the case of $A_1^2$,
passing from a $(TE)$-structure over $A_1^2$ to a pure
$(TLE)$-structure, is done essentially in our Theorem~\ref{t6.2} $(a)$ $(iii)$.
Our argument for the triviality of the Stokes structure
is then contained in the proof of Theorem~\ref{t6.3}.

\cite{Sa19} starts in our notation from $(TE)$-structures
which are already formally isomorphic to sums
$\bigoplus_{i=1}^n\EE^{u_i/z}z^{\alpha_i}$.
Then it is shown that they are also holomorphically
isomorphic to such sums. In~this special case, Corollary 5.7 in~\cite{DH20-2}
give this implication, too.

\item[$(v)$] In $(ii)$ we stated that in the case of $I_2(m)$
with $m\geq 3$,
each $(TE)$-structure on $M\setminus M^{\rm [log]}$ with primitive
Higgs field extends holomorphically to $M$. In~the case of $\NN_2$ this does not hold in general.
For example, start with the flat rank $2$ bundle
$H'\to\C^*\times M$, where $M=\C^2$ (with coordinates
$t=(t_1,t_2)$) with semisimple monodromy with two
different eigenvalues $\lambda_1$ and $\lambda_2$.
Choose $\alpha_1,\alpha_2\in\C$ with
${\rm e}^{-2\pi {\rm i}\alpha_j}=\lambda_j$. Let $s_j\in C^{\alpha_j}$
be generating elementary sections.
Define the new basis
\begin{gather*}
\uuuu{v}=(v_1,v_2)=\big({\rm e}^{t_1/z}\big(s_1+{\rm e}^{-1/t_2}s_2\big),{\rm e}^{t_1/z}(zs_2)\big)
\end{gather*}
on $H'|_{M'}$, where $M':=M\setminus \C\times\{0\}$. Then
\begin{gather*}
z\nnn_{\paa_1}\uuuu{v}=\uuuu{v}\cdot C_1,
\\
z\nnn_{\paa_2}\uuuu{v}=\uuuu{v}\cdot t_2^{-2}{\rm e}^{-1/t_2}C_2,
\\
z^2\nnn_{\paa_z}\uuuu{v}=\uuuu{v}\cdot
\bigg({-}t_1C_1+(\alpha_2-\alpha_1){\rm e}^{-1/t_2}C_2 +
z\begin{pmatrix}\alpha_1&0\\0&\alpha_2+1\end{pmatrix}\bigg).
\end{gather*}
So, we obtain a regular singular $(TE)$-structure on
$M'$ with primitive Higgs field.
The $F$-manifold structure on $M'$ is given by
$e=\paa_1$ and $\paa_2\circ\paa_2=0$, so it is $\NN_2$,
and the Euler field is
$E=t_1\paa_1 + (\alpha_1-\alpha_2)t_2^2\paa_2$.
$F$-manifold and Euler field extend from
$M'$ to $M$, but not the $(TE)$-structure.
\end{enumerate}
\end{Remarks}

\subsection{Proof of Theorem~\ref{t8.5}}\label{c8.3}

$(a)$ Let $\big(H\to\C\times \big(M,t^0\big),\nnn\big)$ be an unfolding of
a $(TE)$-structure of type (Log) over $t^0$.
The $(TE)$-structure $\big(H^{[2]}\to\C\times \big(M,t^0\big),\nnn^{[2]}\big)$
in Lemma~\ref{t3.10}$(c)$ with
$\big(\OO\big(H^{[2]}\big),\nnn^{[2]}\big)=(\OO(H),\nnn)\otimes
\EE^{-\rho^{(0)}/z}$ has trace free pole part.
Lemma~\ref{t3.10}$(d)$ and~$(e)$ apply.
Because of them, it is sufficient to prove that the
$(TE)$-structure $\big(H^{[2]},\nnn^{[2]}\big)$ is induced by
a~$(TE)$-structure
$\big(H^{[3]}\to\C\times \big(M^{[3]},t^{[3]}\big),\nnn^{[3]}\big)$
over $\big(M^{[3]},t^{[3]}\big)=(\C,0)$
via a map $\varphi\colon \big(M,t^0\big)\to\big(M^{[3]},t^{[3]}\big) $,
where the $(TE)$-structure $\big(H^{[3]},\nnn^{[3]}\big)$
is one of the $(TE)$-structures in the 1st to 7th cases
in Theorem~\ref{t6.3} or one of the $(TE)$-structures
in the cases (I) or (II) in Theorem~\ref{t6.7} with invariants as in table~\eqref{8.12}.
Then the $(TE)$-structure $\big(H^{[4]},\nnn^{[4]}\big)$ which is
constructed in Lemma~\ref{t3.10}$(d)$ from $\big(H^{[3]},\nnn^{[3]}\big)$
is one of the $(TE)$-structures in Theorem~\ref{t8.1}$(d)$
with invariants as in table~\eqref{8.12},
and it induces by Lemma~\ref{t3.10}$(e)$ the $(TE)$-structure
$(H,\nnn)$.

From now on we suppose $\rho^{(0)}=0$, so
$(H,\nnn)=\big(H^{[2]},\nnn^{[2]}\big)$.
We consider the invariants $\delta^{(0)},\delta^{(1)}\in
\OO_{M,t^0}$ and $\UU$ and the four possible {\it generic types}
(Sem), (Bra), (Reg) and (Log),
which are defined by the following table,
analogously to Definition~\ref{t6.1},{\samepage
\[
\def\arraystretch{1.3}
\begin{tabular}{c|c|c|c}
\hline
(Sem) & (Bra) & (Reg)& (Log)
\\ \hline
$\delta^{(0)}\neq 0$ & $\delta^{(0)}=0$, $\delta^{(1)}\neq 0$ &
$\delta^{(0)}=\delta^{(1)}=0$, $\UU\neq 0$ & $\UU=0$
\\
\hline
\end{tabular}
\]
First we treat the generic types (Reg) and (Log),
then the generic type (Sem) and (Bra).}

\medskip\noindent
{\it Generic types $($Reg$)$ and $($Log$)$.}
Then the $(TE)$-structure $(H,\nnn)$ is regular singular.
We can use the results in Section~\ref{c7}
(which built on Theorems~\ref{t6.3} and~\ref{t6.7}).
Choose a marking for the $(TE)$-structure $(H,\nnn)$.
Then by Remark~\ref{t7.5}$(i)$, there is a unique map
$\varphi\colon \big(M,t^0\big)\to M^{(H^{{\rm ref},\infty},M^{\rm ref}),{\rm reg}}$
which maps $t\in M$ to the point in the moduli space
over which one has up to isomorphism the same
marked $(TE)$-structure as over $t$. The map
$\varphi$ is holomorphic. By~Re\-mark~\ref{t7.5}$(i){+}(ii)$
it maps $\big(M,t^0\big)$ to one projective curve
which is isomorphic to
$M^{(3),0,\alpha_1,\alpha_2}$ or
$M^{(3),\neq 0,\alpha_1,\alpha_2}$ or
$M^{(3),0,\alpha_1,{\rm log}}$.
The $(TE)$-structure $(H,\nnn)$ is induced by the
$(TE)$-structure over this curve via the map $\varphi$.
The point $t^0$ is mapped to $0$ or $\infty$ in the
cases $M^{(3),0,\alpha_1,\alpha_2}$ or~$M^{(3),\neq 0,\alpha_1,\alpha_2}$
\big(not 0 in the case $M^{(3),\neq 0,\alpha_1,\alpha_1}$\big)
as the $(TE)$-structure over $t^0$ is logarithmic.
The germs at $0$ and $\infty$ in
$M^{(3),0,\alpha_1,\alpha_2}$ and
$M^{(3),\neq 0,\alpha_1,\alpha_2}$
\big(not 0 in the case $M^{(3),\neq 0,\alpha_1,\alpha_1}$\big)
and the germ at any point $t_2^{(3)}$ in
$M^{(3),0,\alpha_1,{\rm log}}$ are contained in table~\eqref{8.12}.
This shows Theorem~\ref{t8.5} for the generic
cases (Reg) and (Log).

\medskip\noindent
{\it Generic types $($Sem$)$ and $($Bra$)$.}
We choose a (connected and sufficiently small)
representative~$M$ of the germ
$\big(M,t^0\big)$, and we choose on it coordinates
$t=(t_1,\dots ,t_m)$ (with $m=\dim M$) with $t^0=0$.
We denote by $M^{\rm [log]}$ the analytic hypersurface
\begin{gather*}
M^{\rm [log]}:=\begin{cases}
\big(\delta^{(0)}\big)^{-1}(0),&\text{if the generic type is (Sem)},
\\[1ex]
\big(\delta^{(1)}\big)^{-1}(0),&\text{if the generic type is (Bra)}.
\end{cases}
\end{gather*}
It contains $t^0$.
Choose a disk $\Delta\subset M$ through $t^0$ with
$\Delta\setminus \{t^0\}\subset M\setminus M^{\rm [log]}$.
The restricted $(TE)$-structure
$(H,\nnn)|_{\C\times(\Delta,t^0)}$ has the same
generic type as the $(TE)$-structure $(H,\nnn)$.
The restricted $(TE)$-structure
$(H,\nnn)|_{\C\times(\Delta,t^0)}$
is isomorphic to a $(TE)$-structure in the cases
1, 2, 3 or 4 in Theorem~\ref{t6.3}.

The parameters of the restricted $(TE)$-structure
$(H,\nnn)|_{\C\times(\Delta,t^0)}$ are given in the
following table:
\[
\def\arraystretch{1.4}
\begin{tabular}{c|l}
\hline
\multicolumn{1}{c|}{Generic type} & \multicolumn{1}{c}{Parameters}
\\ \hline
(Sem)& $k_1,k_2\in\N$ with $k_2\geq k_1$, $\rho^{(1)}\in\C$,
\\
& $\begin{cases}
 \zeta\in\C & \text{if}\ k_2-k_1\in 2\N,\\
 \alpha_3\in \R_{\geq 0}\cup\H & \text{if}\ k_1=k_2
 \end{cases}$
 \\
(Bra) & $k_1,k_2\in\N$, $\rho^{(1)}\in\C$
\\
\hline
\end{tabular}
\]

There is a unique pair $\big(k_1^0,k_2^0\big)\in\N^2$ with
$(k_1,k_2)\in\Q_{>0}\cdot \big(k_1^0,k_2^0\big)$ and with
the conditions in table~\eqref{8.30},
\begin{gather}\label{8.30}
\def\arraystretch{1.4}
\begin{tabular}{l|l}
\hline
\multicolumn{1}{c|}{Generic type and invariants} & \multicolumn{1}{c}{Conditions}
\\ \hline
(Sem$)\colon\ k_2-k_1>0$ odd& $\gcd\big(k^0_1,k^0_2\big)=1$
\\
(Sem$)\colon\ k_2-k_1\in 2\N$, $\zeta= 0$ & $\gcd\big(k^0_1,k^0_2\big)=1$
\\
(Sem$)\colon\ k_2-k_1\in 2\N$, $\zeta\neq 0$ &$k^0_2-k^0_1\in 2\N$, $\gcd\left(k^0_1,\frac{k_1^0+k^0_2}{2}\right)=1$
\\
(Sem$)\colon k_2=k_1$ & $k^0_2=k^0_1=1$
\\
\hline
(Bra) & $\gcd\big(k^0_1,k^0_2\big)=1$
\\
\hline
\end{tabular}
\end{gather}
In fact, it is the pair $\big(k_1^0,k_2^0\big)\in\N^2$
of minimal numbers which satisfies{\samepage
\begin{gather*}
(k_1,k_2)\in \N\cdot \big(k_1^0,k_2^0\big)
\end{gather*}
and which satisfies in the case (Sem) with $k_2-k_1\in 2\N$
and $\zeta\neq 0$ additionally $k_2^0-k_1^0\in 2\N$.}

We denote by $\big(H^{[3]}\to\C\times \big(M^{[3]},t^{[3]}\big),\nnn^{[3]}\big)$
the $(TE)$-structure over $\big(M^{[3]},t^{[3]}\big)=(\C,0)$
which has $\big(k_1^0,k_2^0\big)$ instead of $(k_1,k_2)$,
but which has the same other parameters as the restricted
$(TE)$-structure $(H,\nnn)|_{\C\times(\Delta,t^0)}$.

We have seen in Remarks~\ref{t6.5}$(ii)$ and $(iii)$
that the restricted $(TE)$-structure
$(H,\nnn)|_{\C\times(\Delta,t^0)}$
is induced by the $(TE)$-structure $\big(H^{[3]},\nnn^{[3]}\big)$
via the branched covering
$\varphi^\Delta\colon (\Delta,t^0)\to\big(M^{[3]},t^{[3]}\big)$
with $\varphi^\Delta(\tau)= \tau^{k_1/k_1^0}$.
Here $\tau$ denotes {\it that} coordinate on $\Delta$
with which $(H,\nnn)|_{\C\times (\Delta,t^0)}$ can be
brought to a normal form in the cases 1, 2, 3 and 4 in
Theorem~\ref{t6.3}.

It rests to extend $\varphi^\Delta$ to a map
$\varphi\colon M\to M^{[3]}$ such that $(H,\nnn)$ is induced
by $\big(H^{[3]},\nnn^{[3]}\big)$ via this map $\varphi$.

\begin{claim}\label{cl8.9}
There exists a unique holomorphic function
$\varphi\in\OO_M$ with
\begin{gather}\label{8.32}
\varphi|_\Delta= \varphi^\Delta,
\\[.5ex]
\delta^{(0)}= -\varphi^{k_1^0+k_2^0}\qquad
\text{if the generic type is $($Sem$)$},\label{8.33}
\\
\delta^{(1)}= \frac{k_2^0}{k_1^0+k_2^0}\cdot\varphi^{k_1^0+k_2^0}
\qquad\text{if the generic type is $($Bra$)$}.\label{8.34}
\end{gather}
\end{claim}

\begin{proof}
Choose any point $t^{[1]}\in M^{\rm [log]}$ and any disk
$\Delta^{[1]}$ through $t^{[1]}$ with
$\Delta^{[1]}\setminus \big\{t^{[1]}\big\}\subset M\setminus M^{\rm [log]}$. In~order to show the existence of a function $\varphi\in\OO_M$
with~\eqref{8.33} respectively~\eqref{8.34},
it is sufficient to show that $\delta^{(0)}|_{\Delta^{[1]}}$
respectively $\delta^{(1)}|_{\Delta^{[1]}}$
has at $t^{[1]}$ a zero of an order which is a~multiple of
$k_1^0+k_2^0$.

The restricted $(TE)$-structure
$(H,\nnn)|_{\C\times(\Delta^{[1]},t^{[1]})}$ has the same generic
type as $(H,\nnn)$ and is isomorphic to a $(TE)$-structure
in the cases 1, 2, 3 or 4 in Theorem~\ref{t6.3}.
Its invariants $k_1$ and~$k_2$ are here called
$k_1^{[1]}$ and~$k_2^{[1]}$, in order to distinguish them
from the invariants of $(H,\nnn)|_{(\Delta,t^0)}$.
We~want to show
\begin{gather}\label{8.35}
\big(k_1^{[1]},k_2^{[2]}\big)\in \N\cdot \big(k_1^0,k_2^0\big).
\end{gather}

We did not say much about the Stokes structure.
Here we need the following properties of it, if the
generic type is (Sem):
\begin{gather*}
 k_2=k_1
 \\
\quad \stackrel{(1)}{\iff} (H,\nnn)|_{\C\times\{t^{[2]}\}}
\text{ has trivial Stokes structure for }
t^{[2]}\in\Delta\setminus \big\{t^0\big\}
 \\
\quad \stackrel{(2)}{\iff} (H,\nnn)|_{\C\times\{t^{[2]}\}}
\text{ has trivial Stokes structure for }
t^{[2]}\in\Delta^{[1]}\setminus \big\{t^{[1]}\big\}
 \\
\quad \stackrel{(3)}{\iff} k_2^{[1]}=k_1^{[1]}.
\end{gather*}
$\stackrel{(1)}{\Longrightarrow}$ and
$\stackrel{(3)}{\Longleftarrow}$ are obvious from the
normal form in the 3rd case in Theorem~\ref{t6.3}. It~is not hard to see that the normal forms for fixed
$t\in\C^*$ in the 1st and 2nd case in Theorem~\ref{t6.3}
are not holomorphically isomorphic to an
elementary model in Definition~\ref{t4.4}
(see also Remark~\ref{t8.8}$(ii)$).
This shows $\stackrel{(1)}{\Longleftarrow}$ and
$\stackrel{(3)}{\Longrightarrow}$.
The equivalence $\stackrel{(2)}{\iff}$ is a consequence
of the invariance of the Stokes structure within
isomonodromic deformations.

In the generic type (Sem) with $k_1=k_2$ we have also
$k_2^{[1]}=k_1^{[1]}$ and $k_2^0=k_1^0=1$, and thus
especially~\eqref{8.35}.

Now consider the cases with $k_2>k_1$.
This comprises the generic type (Bra) and
gives in the generic type (Sem) also $k_2^0>k_1^0$
and $k_2^{[1]}>k_1^{[1]}$.
So $(H,\nnn)|_{\C\times(\Delta^{[1]},t^{[1]})}$ is
in the 1st, 2nd or 4th case in Theorem~\ref{t6.3}.
The number $b_3^{(1)}$ in Theorem~\ref{t6.3} is uniquely
determined by the properties
$b_3^{(1)}\in\Q\,\cap \big]\frac{-1}{2},0\big[$ and
$\Eig(M^{\rm mon})=\big\{\exp\big({-}2\pi {\rm i} (\rho^{(1)}\pm b_3^{(1)}\big)\big)\big\}$
(see Remark~\ref{t6.4}$(i)$ for the second property).
Therefore
\begin{gather}
\frac{k_1^0-k_2^0}{2\big(k_1^0+k_2^0\big)}
=\frac{k_1-k_2}{2(k_1+k_2)} =b_3^{(1)}
=\frac{k_1^{[1]}-k_2^{[1]}}{2\big(k_1^{[1]}+k_2^{[1]}\big)}\qquad
\text{in the case (Sem)},\nonumber
\\
\frac{-k_2^0}{2\big(k_1^0+k_2^0\big)}
=\frac{-k_2}{2(k_1+k_2)} =b_3^{(1)}
=\frac{-k_2^{[1]}}{2\big(k_1^{[1]}+k_2^{[1]}\big)}\qquad
\text{in the case (Bra)}.\label{8.37}
\end{gather}
This implies $\big(k_1^{[1]},k_2^{[1]}\big)\in
\Q_{>0}\cdot\big(k_1^0,k_2^0\big)$. In~the cases with $\gcd\big(k_1^0,k_2^0\big)=1$~\eqref{8.35} follows.

If $\gcd\big(k_1^0,k_2^0\big)\neq 1$, then the generic type
is (Sem), $k_2-k_1\in 2\N$, $k_2^0-k_1^0\in 2\N$,
and the invariant $\zeta$ of
$(H,\nnn)|_{\C\times(\Delta,t^0)}$ is $\zeta\neq 0$.
However, then the regular singular
exponents $\alpha_1$ and $\alpha_2$ of the restriction
of the $(TE)$-structure $(H,\nnn)$ over points in
$M\setminus M^{\rm [log]}$ are invariants of the $(TE)$-structure
$(H,\nnn)$. By~\eqref{6.25} and~\eqref{8.37}
also $\zeta$ is an invariant of the $(TE)$-structure
$(H,\nnn)$. Now $\zeta\neq 0$ implies
$k_2^{[1]}-k_1^{[1]}\in 2\N$. Again~\eqref{8.35} follows.

Equations~\eqref{6.22} and~\eqref{6.26} imply that
$\delta^{(0)}|_{\Delta^{[1]}}$
respectively $\delta^{(1)}|_{\Delta^{[1]}}$
has at $t^{[1]}$ a zero of an order which is a multiple of
$k_1^0+k_2^0$. Therefore a function $\varphi\in\OO_M$
with~\eqref{8.33} respectively~\eqref{8.34} exists.

Equations~\eqref{6.22} and~\eqref{6.26} for $(H,\nnn)|_{\C\times(\Delta,t^0)}$
tell
\begin{gather*}
\delta^{(0)}|_\Delta = -\tau^{k_1+k_2}
=-\big(\tau^{k_1/k_1^0}\big)^{k_1^0+k_2^0}
=-\big(\varphi^\Delta\big)^{k_1^0+k_2^0}\qquad
\text{in the case (Sem),}
\\
\delta^{(1)}|_\Delta = \frac{k_2}{k_1+k_2}\tau^{k_1+k_2}
=\frac{k_2^0}{k_1^0+k_2^0}\big(\tau^{k_1/k_1^0}\big)^{k_1^0+k_2^0}
=\frac{k_2^0}{k_1^0+k_2^0}\big(\varphi^\Delta\big)^{k_1^0+k_2^0}\quad\
\text{in the case (Bra).}
\end{gather*}
Therefore a function $\varphi$ as in the claim exists
and is unique.
\end{proof}

Now compare the $(TE)$-structures $(H,\nnn)$ and
$\varphi^*\big(H^{[3]},\nnn^{[3]}\big)$ over $M$.
Both extend to pure $(TLE)$-structures.
For $\varphi^*\big(H^{[3]},\nnn^{[3]}\big)$, one uses the
pull back $\varphi^*\big(\uuuu{v}^{[3]}\big)$ of the basis
$\uuuu{v}^{[3]}$ which gives for $\big(H^{[3]},\nnn^{[3]}\big)$
the Birkhoff normal form in Theorem~\ref{t6.3}.
For $(H,\nnn)$ one starts with the analogous basis
$\uuuu{v}^\Delta$ for $H|_{\C\times\Delta}$ which gives
for $(H,\nnn)|_{\C\times\Delta}$ the Birkhoff normal form
in Theorem~\ref{t6.3}. It~has a unique extension $\uuuu{v}$ to
$\C\times M$ which still yields a Birkhoff normal form.
Compare Remark~\ref{t3.19}$(ii)$ for this.

Remarks~\ref{t6.5}$(ii)$ and $(iii)$ (or simply
the Birkhoff normal forms in Theorem~\ref{t6.3}) tell that
the map $\big(\varphi^*\uuuu{v}^{[3]}\big)|_{\C\times\Delta}
\mapsto \uuuu{v}^\Delta=\uuuu{v}|_{\C\times\Delta}$ is an
isomorphism of pure $(TLE)$-structures.

Now consider a point $t^{[2]}\in \Delta\setminus \big\{t^0\big\}$ and
its image $t^{[4]}:=\varphi\big(t^{[2]}\big)\in
M^{[3]}\setminus \big\{t^{[3]}\big\} =\C\setminus\{0\}$. Over the germ
$\big(M^{[3]},t^{[4]}\big)$, the $(TE)$-structure $\big(H^{[3]},\nnn^{[3]}\big)$
is the part with trace free pole part of a universal
unfolding of $\big(H^{[3]},\nnn^{[3]}\big)|_{\C\times\{t^{[4]}\}}$.
Therefore in a neighborhood $U\subset M$ of $t^{[2]}$,
the $(TE)$-structure $(H,\nnn)|_{\C\times U}$
is induced by
$\big(H^{[3]},\nnn^{[3]}\big)|_{\C\times(M^{[3]},t^{[4]})}$
via a map $\www\varphi\colon U\to M^{[3]}$.
We~can choose it such that
\begin{gather}\label{8.38}
\www\varphi|_\Delta=\varphi^\Delta.
\end{gather}
Equations~\eqref{6.22} and~\eqref{6.26} tell
\begin{gather}\label{8.39}
\delta^{(0)}|_U = -(\www\varphi)^{k_1^0+k_2^0}
\qquad\text{in the case (Sem)},
\\
\delta^{(1)}|_U =
\frac{k_2^0}{k_1^0+k_2^0}(\www\varphi)^{k_1^0+k_2^0}
\qquad\text{in the case (Bra)}.\label{8.40}
\end{gather}
Equations~\eqref{8.38}--\eqref{8.40} and Claim~\ref{cl8.9} imply
$\www\varphi=\varphi|_U$.
Therefore the matrices in Birkhoff normal form
for the basis $\uuuu{v}$ of $(H,\nnn)$ coincide
on $\C\times U$ with the matrices in Birkhoff normal form
for the basis $\varphi^*\big(\uuuu{v}^{[3]}\big)$ of
$\varphi^*\big(H^{[3]},\nnn^{[3]}\big)$.
As all matrices are holomorphic on $\C\times M$,
they coincide pairwise on $\C\times M$.
Therefore the pure $(TLE)$-structure $(H,\nnn)$
with basis $\uuuu{v}$ is isomorphic to the pure
$(TLE)$-structure $\varphi^*\big(H^{[3]},\nnn^{[3]}\big)$
with basis $\varphi^*\big(\uuuu{v}^{[3]}\big)$.
This finishes the proof of part~$(a)$ of Theorem~\ref{t8.5}.

\medskip
$(b)$ If the original $(TE)$-structure $(H,\nnn)$ has
the form
$\varphi^*\big(\OO(H^{[5]}),\nnn^{[5]}\big)\otimes\EE^{-\rho^{(0)}/z}$
then the $(TE)$-structure $\big(H^{[2]},\nnn^{[2]}\big)$
with trace free pole part which was associated to
$(H,\nnn)$ at the beginning of the proof of part~$(a)$,
has the form $\varphi^*\big(\OO(H^{[5]}),\nnn^{[5]}\big)$.

Then any $(TE)$-structure $\big(H^{[3]},\nnn^{[3]}\big)$
over $\big(M^{[3]},t^{[3]}\big)=(\C,0)$ works, whose restriction
over~$t^{[3]}$ is the given logarithmic $(TE)$-structure
$\big(H^{[5]},\nnn^{[5]}\big)$.

In the cases with $N^{\rm mon}=0$, table~\eqref{8.12} offers
one of generic type (Sem) with $k_1=k_2=1$
(and some with $k_2>k_1$ if $\alpha_1-\alpha_2\in\Q\cap (-1,1)$)
and one or two of generic type (Reg), see table~\eqref{8.11}. In~the cases with $N^{\rm mon}\neq 0$, table~\eqref{8.12}
offers two of generic type (Reg) if the leading exponents
$\alpha_1$ and $\alpha_2$ satisfy $\alpha_1-\alpha_2\in\N$,
and one if they satisfy $\alpha_1=\alpha_2$,
compare also Figures~\ref{figure4} and~\ref{figure5} in Theorem~\ref{t7.4}$(b)$.
Therefore the inducing $(TE)$-structure in table~\eqref{8.12}
is not unique except for the case $N^{\rm mon}\neq 0$ and
$\alpha_1=\alpha_2$,
if the original $(TE)$-structure has the form
$\varphi^*\big(\OO(H^{[5]}),\nnn^{[5]}\big)\otimes\EE^{-\rho^{(0)}/z}$.

{\sloppy
In the other cases, the proof of part~$(a)$ shows the
uniqueness of the $(TE)$-structure $\big(H^{[3]},\nnn^{[3]}\big)$.
The uniqueness of $\big(H^{[3]},\nnn^{[3]}\big)$
gives also the uniqueness of $\big(H^{[4]},\nnn^{[4]}\big)$
in the first paragraph of the proof of part~$(a)$.

}

\medskip
$(c)$ This follows from the uniqueness in part $(b)$.
\hfill$\qed$

\section[A family of rank 3 $(TE)$-structures with a functional parameter]
{A family of rank 3 $\boldsymbol{(TE)}$-structures\\ with a functional parameter}\label{c9}

M.~Saito presents in~\cite{SaM17} a family of
Gauss--Manin connections with a functional parameter. In~the arXiv paper~\cite{SaM17}, the bundle has rank
$n$, but in a preliminary version
it has rank 3 and is more transparent.

Here we translate the rank 3 example
by a Fourier--Laplace transformation
into a family of $(TE)$-structures
with primitive Higgs fields over a fixed
3-dimensional globally irreducible
$F$-manifold with an Euler field, such that the
$F$-manifold with Euler field is nowhere regular.
The family of $(TE)$-structures has a functional parameter
$h(t_2)\in\C\{t_2\}$.

In the following, we write down a $(TE)$-structure of rank 3
on a manifold $M=\C^3$ with coordinates $t_1$, $t_2$, $t_3$.
The restriction to
$\big\{t\in\C^3\,|\, t_1=0\big\}=\{0\}\times \C^2$
is a FL-transformation of~Saito's example.
The parameter $t_1$ and this $F$-manifold are
not considered in~\cite{SaM17}. There the base
space has only the two parameters $t_2$ and $t_3$.
Choose an arbitrary function $h(t_2)\in\C\{t_2\}$
with $h''(0)\neq 0$.

Let $H'\to\C^*\times M$ be a holomorphic vector bundle
with flat connection with trivial monodromy and basis of
global flat sections $s_1$, $s_2$, $s_3$. Define an extension
to a vector bundle $H\to\C\times M$ using the following
holomorphic sections of $H'$, which also form a basis
of sections of $H'$:
\begin{gather*}
v_1:= {\rm e}^{t_1/z}\cdot \big(zs_1+t_2\cdot zs_2+
h(t_2)\cdot zs_3+t_3\cdot z^2s_3\big),\\
v_2:= {\rm e}^{t_1/z}\cdot \big(z^2s_2+h'(t_2)\cdot z^2s_3\big),\qquad
v_3:= {\rm e}^{t_1/z}\cdot z^3s_3.\nonumber
\end{gather*}
Denote $\uuuu{v}:=(v_1,v_2,v_3)$.
Denote $\paa_{t_j}:=\paa_j$. Then
\begin{gather*}
z\nnn_{\paa_1}\uuuu{v}= \uuuu{v}\cdot {\bf 1}_3,\\
z\nnn_{\paa_2}\uuuu{v}= \uuuu{v}\cdot
\begin{pmatrix}0&0&0\\1&0&0\\0&h''(t_2)&0\end{pmatrix}\!,
\nonumber\\
z\nnn_{\paa_3}\uuuu{v}= \uuuu{v}\cdot
\begin{pmatrix}0&0&0\\0&0&0\\1&0&0\end{pmatrix}\!,\nonumber\\
z^2\paa_z\uuuu{v}= \uuuu{v}\cdot
\left({-}t_1\cdot {\bf 1}_3
+\begin{pmatrix}0&0&0\\0&0&0\\t_3&0&0\end{pmatrix}\right)
+z\cdot\uuuu{v}\cdot
\begin{pmatrix}1&0&0\\0&2&0\\0&0&3\end{pmatrix}\!.\nonumber
\end{gather*}

Write $\uuuu\paa:=(\paa_1,\paa_2,\paa_3)$.
The pole parts give the multiplication $\circ$ on the
$F$-manifold and the Euler field $E$ by
\begin{gather*}
\paa_1\circ\uuuu{\paa}=\uuuu{\paa}\cdot{\bf 1}_3,\\
\paa_2\circ\uuuu{\paa}= \uuuu{\paa}\cdot
\begin{pmatrix}0&0&0\\1&0&0\\0&h''(t_2)&0\end{pmatrix}\!,\nonumber\\
\paa_3\circ\uuuu{\paa}= \uuuu{\paa}\cdot
\begin{pmatrix}0&0&0\\0&0&0\\1&0&0\end{pmatrix}\!,\nonumber\\
E\circ\uuuu{\paa}= -\uuuu{\paa}\cdot
\left(-t_1\cdot {\bf 1}_3
+\begin{pmatrix}0&0&0\\0&0&0\\t_3&0&0\end{pmatrix}\right)\!,\qquad
\text{so}\quad E= t_1\cdot\paa_1-t_3\cdot\paa_3.\nonumber
\end{gather*}
One can introduce a new coordinate system
$\big(\www t_1,\www t_2,\www t_3\big)=\big(t_1,\www t_2,t_3\big)$ on the germ
$(M,0)$ with
\begin{gather*}
\paa_{\www t_2}=\frac{1}{\sqrt{h''(t_2)}}\cdot\paa_2.
\end{gather*}
Denote $\www\paa_j:=\paa_{\www t_j}$
and $\www{\uuuu\paa}:=\big(\www\paa_1,\www\paa_2,
\www\paa_3\big)=\big(\paa_1,\www\paa_2,\paa_3\big)$.
Introduce also the new section
\begin{gather*}
\www v_2:=\frac{1}{\sqrt{h''(t_2)}}\cdot v_2,
\end{gather*}
and the new basis
$\www{\uuuu v}=\big(\www v_1,\www v_2,\www v_3\big)
=\big(v_1,\www v_2,v_3\big)$
of the given $(TE)$-structure. Then
\begin{gather*}
z\nnn_{\www\paa_1}\uuuu{\www v}=\uuuu{\www v}\cdot {\bf 1}_3,
\\
z\nnn_{\www\paa_2}\uuuu{\www v}=\uuuu{\www v}\cdot
\begin{pmatrix}0&0&0\\1&0&0\\0&1&0\end{pmatrix}
+\uuuu{\www v}\cdot
\begin{pmatrix}0&0&0\\0&\paa_2\frac{1}{\sqrt{h''(t_2)}}&0
\\
0&0&0&\end{pmatrix}\!,
\\
z\nnn_{\www\paa_3}\uuuu{\www v}=\uuuu{\www v}\cdot
\begin{pmatrix}0&0&0\\0&0&0\\1&0&0\end{pmatrix}\!,
\\
z^2\paa_z\uuuu{\www v}=\uuuu{\www v}\cdot
\left(-t_1\cdot {\bf 1}_3
+\begin{pmatrix}0&0&0\\0&0&0\\t_3&0&0\end{pmatrix}\right)
+z\cdot\uuuu{\www v}\cdot
\begin{pmatrix}1&0&0\\0&2&0\\0&0&3\end{pmatrix}\!.\nonumber
\end{gather*}
In the new coordinates the multiplication becomes simpler
and independent of the choice of~$h(t_2)$ (as long as
$h''(t_2)\neq 0$):
\begin{gather*}
\www\paa_1\circ\uuuu{\www\paa}=\uuuu{\www\paa}\cdot {\bf 1}_3,
\\
\www\paa_2\circ\uuuu{\www\paa}=\uuuu{\www\paa}\cdot
\begin{pmatrix}0&0&0\\1&0&0\\0&1&0\end{pmatrix}\!,
\\
\www\paa_3\circ\uuuu{\www\paa}= \uuuu{\www\paa}\cdot
\begin{pmatrix}0&0&0\\0&0&0\\1&0&0\end{pmatrix}\!,
\\
E\circ\uuuu{\www\paa}= -\uuuu{\www\paa}\cdot\left({-}t_1\cdot {\bf 1}_3
+\begin{pmatrix}0&0&0\\0&0&0\\t_3&0&0\end{pmatrix}\right)\!,
\\
\text{so}\quad E= t_1\cdot\www\paa_1-t_3\cdot\www\paa_3.
\end{gather*}
This is the nilpotent $F$-manifold for $n=3$ in
\cite[Theorem~3]{DH17}.
However, the Euler field here is different from the one in
\cite[Theorem~3]{DH17}. The endomorphism $E\circ$ here is not
regular, but has only the one eigenvalue $t_1$ and
has for $t_3\neq 0$ one Jordan block of size
$2\times 2$ and one Jordan block of size $1\times 1$
and is semisimple for $t_3=0$.

The sections $v_1$, $v_2$, $v_3$ define also an
extension $\wwh{H}\to\P^1$ such that the $(TE)$-structure
extends to a pure $(TLE)$-structure.

Furthermore $v$ satisfies all properties of the
section $\zeta$ in Theorem~6.6(b) in~\cite{DH20-2}.
Thus the $F$-manifold with Euler field is enriched to a
flat $F$-manifold with Euler field (Definition~3.1(b) in~\cite{DH20-2}).

If we try to introduce a pairing which would make it
into a pure $(TLEP)$-structure, we get a constraint
$h''(t_2)={\rm const}$. However, probably similar higher dimensional
examples allow also an extension to pure $(TLEP)$-structures
while keeping the functional freedom.
This would give families of Frobenius manifolds with
Euler fields with functional freedom on a fixed $F$-manifold
with Euler field.

In the example above, $t_1$, $t_2$, $t_3$ are flat coordinates
and $\www t_1=t_1$, $\www t_2$, $\www t_3=t_3$ are
{\it generalized canonical coordinates}
(in which the multiplication has simple formulas).

\subsection*{Acknowledgements}

This work was funded by the Deutsche Forschungsgemeinschaft (DFG, German Research Foundation)~-- 242588615. I would like to thank Liana David for a lot of joint work on $(TE)$-structures.

\LastPageEnding

\end{document}